\documentclass[reqno,11pt,a4paper]{amsart}
\RequirePackage{amsmath}
\RequirePackage{amssymb}
\usepackage{paralist}
\usepackage{graphics} 
\usepackage{epsfig} 
\usepackage{graphicx} 
\usepackage[colorlinks=true]{hyperref}
\usepackage{amsmath}
\usepackage{amssymb}
\usepackage{mathrsfs}
\usepackage{bm}
 \usepackage{geometry}
 \usepackage{fullpage}
\newcommand{\qe}{}
\newtheorem{Thm}{Theorem}[section]
\newtheorem{Def}[Thm]{Definition}
\newtheorem{Lem}[Thm]{Lemma}

\newtheorem{Rem}[Thm]{Remark}

\newcommand{\1}{\mathbf{1}}
\newcommand{\R}{\mathbb{R}}
\newcommand{\Rd}{{\mathbb{R}^d}}

\renewcommand{\Re}{\text{Re}}
\renewcommand{\Im}{\text{Im}}
\newcommand{\F}{\mathscr{F}}
\newcommand{\N}{\mathbb{N}}
\newcommand{\C}{\mathbb{C}}
\newcommand{\g}{\mathbb{G}}
\renewcommand{\L}{\mathscr{L}}

\newcommand{\B}{\widehat{B}(\xi)}
\newcommand{\A}{\widehat{A}(\xi)}
\renewcommand{\Re}{\text{Re}}
\newcommand{\Ker}{\text{Ker}}
\renewcommand{\S}{\mathscr{S}}
\newcommand{\K}{\mathcal{K}}
\newcommand{\<}{\langle}
\renewcommand{\>}{\rangle}

\renewcommand{\H}{H(a^{1/2})}
\newcommand{\HH}{H(a)}

\usepackage{tikz,pgfplots}
\usepackage{tikz-cd}
\usepackage{tikz-3dplot}
\usepackage{amsmath,amssymb,amsthm,amsfonts,bm}
\usepackage{color}
\usepackage{tcolorbox}
\usepackage{pdfpages}
\usepackage{textcomp}
\usetikzlibrary{patterns}
\usetikzlibrary{shapes.geometric}
\usetikzlibrary{arrows.meta,arrows}
\usetikzlibrary{decorations.pathreplacing,decorations.pathmorphing}
\usetikzlibrary{hobby}
\usetikzlibrary{math}
\usepgfplotslibrary{fillbetween}
\usepackage{graphicx}
\usepackage{multicol}
\usepackage{float}
\usepackage{subfigure}

\usepackage{cite}

\allowdisplaybreaks

\newcommand{\vertiii}[1]{{\left\vert\kern-0.25ex\left\vert\kern-0.25ex\left\vert #1 \right\vert\kern-0.25ex\right\vert\kern-0.25ex\right\vert}}

\begin{document}

		\title[Global Existence of Boltzmann Equation]{Global Existence of Non-cutoff Boltzmann Equation in Weighted Sobolev Space}

	
\author{Dingqun Deng}
\address[D.-Q. Deng]{Beijing Institute of Mathematical Sciences and Applications and Yau Mathematical Science Center, Tsinghua University, Beijing, People's Republic of China}
\curraddr{}
\email{dingqun.deng@gmail.com}
\thanks{}
\date{}
\maketitle


	\begin{abstract}
	This article presents a new approach of semigroup analysis and pseudo-differential calculus for deriving the regularizing estimate on non-cutoff linearized Boltzmann equation. We are able to obtain regularizing estimate of semigroup $e^{tB}$ that is continuous from weighted Sobolev space $H(a^{-1/2})H^m_x$ to $H(a^{1/2})H^m_x$ with a sharp large time decay. With these properties, we prove the existence of global-in-time unique solution to the non-cutoff Boltzmann equation for hard potential on the whole space with weak regularity assumption on initial data. We consider the hard potential case since $H(a^{1/2})$ can be embedded in $L^2$. This work develops the application of pseudo-differential calculus, spectrum analysis and semigroup theory to non-cutoff Boltzmann equation.

%

	%
	%
\end{abstract}	

\tableofcontents

 \section{Introduction}
The non-cutoff Boltzmann equation for the long-range interaction potentials is a fundamental mathematical model in collisional kinetic theory, which describes the dynamics of a non-equilibrium rarefied gas. A well-established framework for studying global well-posedness is searching for solutions close to the global Maxwellian equilibrium in different function spaces. An open problem is how to characterize the optimal mathematical space of initial data with lower regularity in space and velocity variables such that the unique solution exists globally in time. 
This work gives one direction to this problem, by reducing the Sobolev space on velocity from $H^l_k$ to $\H$. Also, this work provides an approach to applying the semigroup theory to the non-cutoff Boltzmann equation, which is very helpful in the analysis of the solution to the non-cutoff Boltzmann equation.

\subsection{Equation}
 In the present paper, we consider the Boltzmann equation in $d$-dimension:
 \begin{align}\label{eq1}
   F_t + v\cdot\nabla_x F = Q(F,F),
 \end{align} 
where the unknown $F(x,v,t)$ represents the density of particles in phase space, with spatial coordinate $x\in\Rd$ and velocities $v\in\Rd$ with $d\ge 2$. The Boltzmann collision operator $Q(F,G)$ is a bilinear operator defined for sufficiently smooth functions $F,G$ by
 \begin{align*}
   Q(F,G)(v) := \int_{\Rd}\int_{S^{d-1}}B(v-v_*,\sigma) (F'_*G' - F_*G) \, d\sigma dv_*
 \end{align*}
 where $F'_* = F(x,v'_*,t)$, $G' = G(x,v',t)$, $F_* = F(x,v_*,t)$, $G =G(x,v,t)$. $(v,v_*)$ are the velocities of two gas particles before collision while
 $(v',v'_*)$ are the velocities after collision satisfying the following conservation laws of momentum and energy,
 \begin{align*}
   v+v_*=v'+v'_*,\ \ |v|^2+|v_*|^2=|v'|^2+|v'_*|^2.
 \end{align*}As a consequence, for $\sigma\in \mathbb{S}^{d-1}$, the unit sphere in $\R^d$, we have the $\sigma$-representation:
 \begin{align*}
   v' = \frac{v+v_*}{2}+\frac{|v-v_*|}{2}\sigma,\ \
   v'_* = \frac{v+v_*}{2}-\frac{|v-v_*|}{2}\sigma.
 \end{align*}
 Also, we define the angle $\theta$ between $v-v_*$ and $\sigma$ by
 \begin{align*}
   \cos\theta = \frac{v-v_*}{|v-v_*|}\cdot \sigma,
 \end{align*}where $\cdot$ denotes the usual inner product in $\R^d$.
 \subsection{Collision kernel} The collision kernel $B$ satisfies
 \begin{align*}
   B(v-v_*,\sigma) = |v-v_*|^\gamma b(\cos\theta),
 \end{align*}
 for some $\gamma\in\R$ and function $b$. Without loss of generality, we can assume that $B(v-v_*,\sigma)$ is supported on $(v-v_*)\cdot\sigma\ge 0$ which corresponds to $\theta\in[0,\pi/2]$, since $B$ can be replaced by its symmetrized form $\overline{B}(v-v_*,\sigma) = B(v-v_*,\sigma)+B(v-v_*,-\sigma)$.
 Moreover, we are going to work on the collision kernel without angular cut-off, which corresponds to the case of inverse power interaction laws between particles. That is,
 \begin{align*}
   b(\cos\theta)\approx \theta^{-d+1-2s}\ \text{ on }\theta \in (0,\pi/2),
 \end{align*}
where we assume 
 \begin{align*}
   s\in (0,1), \quad \gamma\in (-d,\infty).
 \end{align*}For Boltzmann equation without angular cut-off, it's customary to use term soft potential when $\gamma+2s< 0$, Maxwell molecules when $\gamma+2s=0$, and hard potential when $\gamma+2s\ge0$. For mathematical theory of Boltzmann equation, one may refer to \cite{Alexandre2001,Alexandre2009,Cercignani1994,Villani2002} for more introduction.

\subsection{Preliminary result}
 We will study the Boltzmann equation \eqref{eq1} near the global Maxwellian equilibrium
 \begin{align*}
   \mu(v)=(2\pi)^{-d/2}e^{-|v|^2/2}.
 \end{align*}
 Set $F = \mu + \mu^{\frac{1}{2}}f$ and then the Boltzmann equation \eqref{eq1} becomes
 \begin{align*}
   f_t + v\cdot\nabla_x f = Lf + \mu^{-1/2}Q(\mu^{1/2}f,\mu^{1/2}f),
 \end{align*}
 where $L$ is called the linearized Boltzmann operator given by
 \begin{align*}
   Lf = \mu^{-1/2}Q(\mu,\mu^{1/2}f) + \mu^{-1/2}Q(\mu^{1/2}f,\mu) = L_1f+L_2f.
 \end{align*}
$L_1,L_2$ are defined in \eqref{L1}\eqref{L2}. The kernel of $L$ is Span$\{\psi_i\}_{i=0}^{d+1}$ defined in \eqref{Kernel} and we denote the projections $P_0$ and $P_1$ from $L^2$ onto $\Ker L$ and $(\Ker L)^{\perp}$ respectively by 
\begin{align}\label{P}
P_0f \equiv Pf := \sum^{d+1}_{i=0}(f,\psi_i)_{L^2}\psi_i,\quad P_1 = I-P_0, 
\end{align}
with $I$ being the identity map on $L^2$. 

We want to apply the symbolic calculus in \cite{Global2019} for our study as the following. 
One may refer to Section \ref{PD} for the materials that are needed to understand pseudo-differential calculus, such as admissible metric and the definition of Sobolev space $H(c)$. Moreover, we recommend \cite{Lerner2010} for more information about pseudo-differential calculus. Let $\Gamma=|dv|^2+|d\eta|^2$ be an admissible metric. 
Define
 \begin{align}\label{a}
   a(v,\eta):=\<v\>^\gamma(1+|\eta|^2+|\eta\wedge v|^2+|v|^2)^s+K_0\<v\>^{\gamma+2s}
 \end{align}be a $\Gamma$-admissible weight, which is proved in \cite{Global2019}, where $K_0>0$ is chosen as the following and $|\eta\wedge v|=|\eta||v|\sin\theta_0$ with $\theta_0$ being the angle between $\eta$ and $v$. 
Applying Theorem 4.2 in \cite{Global2019} and Lemmas 2.1 and 2.2 in \cite{Deng2020a}, there exists $K_0>0$ such that the Weyl quantization $a^w:H(ac)\to H(c)$ and $(a^{1/2})^w:H(a^{1/2}c)\to H(c)$ are invertible, with $c$ be any $\Gamma$-admissible metric. The definition of Weyl quantization $(a^{1/2})^w$ is given in Section \ref{PD}. The weighted Sobolev spaces $H(c)$ is defined by \eqref{sobolev_space}. The real symbol $a$ gives the formal self-adjointness of Weyl quantization $a^w$, which is widely applied in our analysis. By the invertibility of $(a^{1/2})^w$, we have equivalence 
\begin{align*}
\|(a^{1/2})^w(\cdot)\|_{L^2}\approx\|\cdot\|_{\H},
\end{align*}and hence we will equip $\H$ with norm $\|(a^{1/2})^w(\cdot)\|_{L^2}$. Also, we will denote the weighted Sobolev norms that, for $l,n\in\R$,
\begin{align*}
\|f\|_{L^2_vH^m_x}:&=\|\<D_x\>^mf\|_{L^2_{x,v}},\\
\|f\|_{\H H^m_x}:&=\|\<D_x\>^m(a^{1/2})^wf\|_{L^2_{x,v}},
\end{align*}where $\<D_x\>^mf=\F^{-1}_x\<\xi\>^m\F_x f$ and $L^2_{x,v}=L^2(\R^{2d}_{x,v})$. Here $\xi$ is the dual variable of $x$.

To analyze the linearized Boltzmann operator rigorously, we use the Carleman representation \eqref{Carleman} to define $L=L_1+L_2$:
\begin{align}
	L_1f &= \lim_{\varepsilon\to 0}\int_{\R^d,|h|\ge\varepsilon}\,dh\int_{E_{0,h}}\,d\alpha\,\tilde{b}(\alpha,h)\1_{|\alpha|\ge|h|}\frac{|\alpha+h|^{\gamma+1+2s}}{|h|^{d+2s}}\mu^{1/2}(v+\alpha-h)\notag\\
	&\qquad\qquad\Big((\mu^{1/2}(v+\alpha)f(v-h)-\mu^{1/2}(v+\alpha-h)f(v)\Big),\label{L1}\\
	L_2f &= \lim_{\varepsilon\to 0} \int_{\R^d,|h|\ge\varepsilon}\,dh\int_{E_{0,h}}\,d\alpha\,\tilde{b}(\alpha,h)\1_{|\alpha|\ge|h|}\frac{|\alpha+h|^{\gamma+1+2s}}{|h|^{d+2s}}\mu^{1/2}(v+\alpha-h)\notag\\
\label{L2}	&\qquad\qquad\Big(\mu^{1/2}(v-h)f(v+\alpha)-\mu^{1/2}(v)f(v+\alpha-h)\Big),
\end{align}
which are well-defined for Schwartz function according to \cite{Global2019,Deng2020a}.
Here we use the principal value on $h$ in order to assure the integrals are well-defined when the two terms in the parentheses is separated into two integrals, where change of variable can be applied.
By Section 3 in \cite{Deng2020a}, $L=L_1+L_2$ can be regarded as the standard pseudo-differential operator with symbols in $S(a)$.
Then by the unique extension of continuous operator, $L$ is a linear continuous operator from $H(ac)$ into $H(c)$ for any $\Gamma-$admissible weight function $c$. Also, Lemma \ref{basicL} gives the formal self-adjointness of $L$. 
To better applying the previous result, we consider weighted Sobolev norm $\|(a^{1/2})^w(\cdot)\|_{L^2}$, triple norm $\vertiii{\cdot}$ in \cite{Alexandre2012} and the norm $|\cdot|_{N^{s,\gamma}}$ in \cite{Gressman2011},
where 
\begin{align*}
\vertiii{f}^2:&=\int B(v-v_*,\sigma)\Big(\mu_*(f'-f)^2+f^2_*((\mu')^{1/2}-\mu^{1/2})^2\Big)\,d\sigma dv_*dv,\\
|f|^2_{N^{s,\gamma}}:&=\|\<v\>^{\gamma/2+s}f\|^2_{L^2}+\int(\<v\>\<v'\>)^{\frac{\gamma+2s+1}{2}}\frac{(f'-f)^2}{d(v,v')^{d+2s}}\1_{d(v,v')\le 1},
\end{align*}with $d(v,v'):=\sqrt{|v-v'|^2+\frac{1}{4}(|v|^2-|v'|^2)^2}$. Then by (2.13)(2.15) in \cite{Gressman2011}, Proposition 2.1 in \cite{Alexandre2012} and Theorem 1.2 in \cite{Global2019}, for $f\in\S$, $l\in\R$, we have the equivalence of norms:
\begin{align}\label{equivalent_norm}
\|(a^{1/2})^wf\|^2_{L^2_v}\approx\vertiii{f}^2\approx|f|^2_{N^{s,\gamma}}\approx (-Lf,f)_{L^2_v}+\|\<v\>^lf\|_{L^2_v},
\end{align}where the constants depend on $l$. These norms  describe the essential behavior of Boltzmann collision operator. 

\subsection{Main result}
Our main result is the global existence of the Boltzmann equation without angular cutoff for hard potential in the weighted Sobolev space. Former results on spatially inhomogeneous Boltzmann equation were on torus $\mathbb{T}^d_x$, cf. \cite{Gressman2011}, or require more regularity on the initial data, cf. \cite{ALEXANDRE2011a}. This work provides the global solution on the whole space $\R^d_x$ and requires minimum assumption with respect to $v$ on initial data as Boltzmann equation with angular cutoff. According to the previous study of the non-cutoff Boltzmann equation, the dissipation estimate on $L$ is necessary for finding the solution. To extract the dissipation estimate of linearized Boltzmann equation, we define a new type of decomposition 
\begin{equation}\label{DefA}
\begin{aligned}
A &:= \left\{\begin{aligned}
&-L + P\<v\>^{\gamma+2s}P,&\text{ if }\gamma+2s\ge 0, \\
&-L + P,&\text{ if }\gamma+2s\le 0,
\end{aligned}\right.\\
K &:= A+L.
\end{aligned}
\end{equation}
where $P$ is the projection given in \eqref{P}. 
Then $L$ is decomposed as $-A+K$ and we have the following dissipation properties, which is valid for both hard and soft potential. 
\begin{Thm}\label{Thm0}The linearized Boltzmann equation can be split as $L=-A+K$ by its definition \eqref{DefA}. There exists $\nu_0>0$ such that for $f\in \S(\R^d)$,
\begin{equation}\label{A_positive}
	\Re(Af,f)_{L^2_v}\ge \nu_0\|(a^{1/2})^wf\|^2_{L^2_v}
\end{equation}
Consequently, 
\begin{align*}
\Re(Af,f)_{L^2_{x,v}}\ge \nu_0\|(a^{1/2})^wf\|^2_{L^2_{x,v}}.
\end{align*}Additionally, $A$ and $K$ can be regarded as Weyl quantizations with symbols belonging to $S(a)$ and $S(\<v\>^{-k}\<\eta\>^{-l})$respectively, for any $k,l\ge0$. Hence $K$ is compact on $L^2_v$ and maps $L^2_v$ function into $\S(\Rd)$.
\end{Thm}
Define an operator closure\begin{align*}
	B=\overline{-v\cdot\nabla_x+L}
\end{align*} as in \eqref{eq33}, then $B$ generates a strongly continuous semigroup $e^{tB}$ on $L^2$. The closure is necessary for generating the semigroup and hence it's necessary throughout our arguments. Then we can discuss the solution $e^{tB}f_0$ to the linearized Boltzmann equation:
\begin{equation*}\left\{\begin{aligned}
		f_t = Bf,\\
		f|_{t=0}=f_0.
	\end{aligned}\right.
\end{equation*}
This semigroup $e^{tB}$ generated by the linearized Boltzmann operator plays an important role in the perturbation theory of Boltzmann equation and kinetic equation, since the solution to Boltzmann equation can be written into a perturbation form of the solution to its linearized equation, for instance, \cite{Gualdani2017,Ukai1982}. Our first main result gives an optimal regularizing effect on the linearized Boltzmann equation for hard potential and a large time decay estimate. With this regularizing estimate established, we can apply the energy estimate to obtain a global solution. Our method is building on the whole space $\R^d_x$, which is different from torus $\mathbb{T}^d_x$. Thus, the analysis of the spectrum structure on linearized Boltzmann operator is required and provides a new approach to the existence theory of Boltzmann equation without angular cutoff. 
Before we state the regularizing estimate of the linearized Boltzmann operator, we denote a structural constant $\nu_0>0$ by applying Proposition 2.1 in \cite{Alexandre2012}
\begin{align*}
	(-Lf,f)_{L^2_v}\ge\nu_0\|(a^{1/2})^w(I-P)f\|^2_{L^2_v}.
\end{align*}
Then we can state the regularizing estimate of semigroup $e^{tB}$. 
\begin{Thm}\label{Thm1}Assume $\gamma+2s\ge0$. Then $B$ generates a strongly continuous semigroup $e^{tB}$ on $L^2_{x,v}$. 
	For any $f\in\S(\R^{2d}_{x,v})$, $k\ge 2$, $m\in\R$ and $p\in[1,2)$, we have 
	\begin{align*}
		\|e^{tB}f\|^2_{H(a^{1/2})H^m_x}
		&\le \frac{Ce^{-2\kappa t}}{t^{2k}}\|f\|^2_{H(a^{-1/2})H^m_x}
		+\frac{C}{(1+t)^{d/2(2/p-1)}}\|a^{-1/2}(v,D_v)f\|^2_{L^2_v(L^p_x)},
	\end{align*}for any $t>0$, where $\kappa>0$ is a constant depending on $\nu_0$
	and the constant $C$ is independent of $f$ and $p$.
\end{Thm}
Note that the space $H(a^{1/2})H^m_x$ is a better space than $H(a^{-1/2})H^m_x$ when $\gamma+2s\ge0$, which gives the so-called regular estimate. Also, our theorem gives the large time decay to Cauchy problem of linearized Boltzmann equation.

The smoothing effect of the linearized Boltzmann operator and Boltzmann collision operator for angular non-cutoff collision kernel was discussed in many contexts. In the beginning, entropy production estimates for non-cutoff assumption were established, as in \cite{Alexandre2000,Lions1998}. Their results were widely applied in the theory of non-cutoff Boltzmann equation. Late, many works discover the optimal regular estimate of Boltzmann collision operator in $v$ in a different setting. We refer to \cite{Alexandre2011,Global2019,Gressman2011a,Mouhot2007} for the dissipation estimate of collision operator, and \cite{ALEXANDRE2005,Alexandre2009a,Alexandre2010,Barbaroux2017,Barbaroux2017a,Chen2011,Chen2012,Lekrine2009,Lerner2014,Lerner_2015,Chen2021,Deng2021b,Duan2021a} for smoothing effect of the solution to Boltzmann equation in different aspect. With this kind of regular estimate, existence theory was well-discussed in many papers, for instance, \cite{Gressman2011,ALEXANDRE2011a}. These works show that the Boltzmann operator behaves locally like a fractional operator:
\begin{align*}
	Q(f,g)\sim (-\Delta_v)^sg+\text{lower order terms}.
\end{align*}
More precisely, according to the symbolic calculus developed by Alexandre-H{\'{e}}rau-Li \cite{Global2019}, the linearized Boltzmann operator behaves essentially as 
\begin{align*}
	L \sim \<v\>^\gamma(-\Delta_v-|v\wedge\partial_v|^2+|v|^2)^s+\text{lower order terms}.
\end{align*}
This diffusion property shows that the spatially homogeneous Boltzmann equation behaves like a fractional heat equation, while the spatially inhomogeneous Boltzmann behaves like the generalized Kolmogorov equation. We refer to \cite{Barbaroux2017a,Lekrine2009,Lerner_2015} for Kac equation, the one-dimensional model of Boltzmann equation, and \cite{Morimoto2009} for similar kinetic equation. Thus, the Cauchy problem to Boltzmann equation enjoys a smoothing effect at any positive time, which is essential to the existence theory. 

Once the large time behavior of linearized Boltzmann equation is established, we can find out the existence of global-in-time unique solution to Boltzmann equation. 
\begin{Thm}\label{Thm2}Suppose $d\ge3$, $m>\frac{d}{2}$, $\gamma+2s\ge 0$. There exists $\varepsilon_0,C>0$ so small that if 
	\begin{align*}
		\|f_0\|_{X}\le \varepsilon_0,
	\end{align*}
	where $X$ is defined by 
	\begin{align*}
		\|f\|^2_X = \delta\|f\|^2_{L^2_vH^m_x} + \int^{\infty}_0\|e^{\tau B}f\|^2_{L^2_vH^m_x}\,d\tau, 
	\end{align*}
	then there exists an unique global weak solution $f$ to Boltzmann equation 
	\begin{align}
		\label{eq00}f_t=Bf+\Gamma(f,f),\quad f|_{t=0}=f_0,
	\end{align} satisfying 
	\begin{align*}
		\|f\|_{L^\infty([0,\infty);L^2_vH^m_x)}+\|f\|_{L^2([0,\infty);\H H^m_x)}\le C\varepsilon_0,
	\end{align*}with some constant $C>0$. 
\end{Thm}
\begin{Rem}
	Together with the upcoming work \cite{Deng2020b} on the regularity, the solution $f$ we obtained here is smooth in the spatial variable $x$ and velocity $v$. Thus, the solution we derived here is the global-in-time smooth solution to the Boltzmann equation. 
\end{Rem}
By \eqref{DefX} and \eqref{eqq1}, the smallness assumption on initial data $f_0$ can be fulfilled if 
\begin{align*}
	\|f_0\|_{L^2_vH^m_x}+\|(a^{-1/2})^wf_0\|_{L^2_v(L^p_x)}
\end{align*}is sufficient small for some $p\in[1,\frac{2d}{d+2})$. Such choice of Sobolev space is similar to the cutoff case, cf. \cite{Ukai}.
The space $\H H^m_x$ is critical due to estimate \eqref{A_positive}, \eqref{eq88} on $L$ and $\Gamma(\cdot,\cdot)$ respectively. The weak solution $f$ means that 
\begin{align*}
	\int^\infty_0(f_t,\varphi)_{L^2_{x,v}}\,dt = \int^\infty_0(f,B^*\varphi)_{L^2_{x,v}}\,dt+\int^\infty_0(\Gamma(f,f),\varphi)_{L^2_{x,v}}\,dt,
\end{align*}
for $\varphi\in \mathscr{D}((0,\infty);\S(\Rd))$, where $B^*=\overline{v\cdot\nabla_x+L}$. For Cauchy problem to Boltzmann equation near the global Maxwellian without angular cut-off, we refer to \cite{Gressman2011} for the existence theory on torus and \cite{ALEXANDRE2011a,Alexandre2011aa,Alexandre2012} on the whole space. We improve the space $H^k_l$ used in \cite{ALEXANDRE2011a} and the space $\H H^m_x$ we introduce in this paper is actually equivalent the norms $|\cdot|_{N^{s,\gamma}}$ and $\vertiii{\cdot}$, thanks to \eqref{equivalent_norm}. The assumption $\gamma+2s\ge 0$ is needed for a spectral gap \ref{spectrum_struture} and then we can deduce a polynomial time decay on the semigroup $e^{tB}$. Hence, we can use the norm $X$ \eqref{DefX} to describe the large time behavior of solution, which was also discussed in \cite{Carrapatoso2016,Gualdani2017,Herau2020}. But our new result builds on the whole space $\R^d_x$, which is essentially different from torus $\mathbb{T}^d_x$.

 \subsection{Organization of the article}
This article is organized as follows. In Section \ref{sec2}, we provide a proof of the dissipation of $A$ and the solution $g$ to equation $(v\cdot\nabla_x+A)g=f$, a linearized form of Boltzmann equation. The dissipation and existence to this linearized equation gives the rigorously proof of generating the semigroup $e^{tB}$ on $L^2$. Note that the Section \ref{sec2} is valid for both hard and soft potential. In Section \ref{sec3}, we write a rigorous argument for obtaining the spectrum structure to linearized Boltzmann operator for hard potential. In Section \ref{sec4}, after establishing the spectrum structure, we can discuss the regular estimate of semigroup $e^{tB}$. In Section \ref{sec5}, we deduce the global existence of Boltzmann equation on the whole space. The Appendix \ref{sec6} gives some general theory on Boltzmann equation, functional analysis and pseudo-differential calculus.

\subsection{Notations}
Throughout this article, we shall use the following notations. $\S(\Rd)$ is the set of Schwartz functions while $\mathscr{D}(0,\infty)$ is the set of smooth functions with compact support in $(0,\infty)$.
 For any $v\in\Rd$, we denote $\<v\>=(1+|v|^2)^{1/2}$. The gradient in $v$ is denoted by $\partial_v$. Denote a complex number $\lambda = \sigma + i\tau$. 

 The notation $a\approx b$ (resp. $a\gtrsim b$, $a\lesssim b$) for positive real function $a$, $b$ means there exists $C>0$ not depending on possible free parameters such that $C^{-1}a\le b\le Ca$ (resp. $a\ge C^{-1}b$, $a\le Cb$) on their domain. $\Re (a)$ means the real part of complex number $a$. $[a,b]=ab-ba$ is the commutator between operators. 

 For pseudo-differential calculus, we refer to Section \ref{PD} for the definition of admissible metric. Then we write $\Gamma_v=|dv|^2+|d\eta|^2$, $\Gamma_{x,v}=|dx|^2+|d\xi|^2+|dv|^2+|d\eta|^2$ to be admissible metrics, where $(x,v)\in \Rd\times\Rd$ is the space-velocity variable and $(\xi,\eta)\in \Rd\times\Rd$ is the corresponding variable in dual space (the variable after Fourier transform).
 Let $m_v$ be $\Gamma_v$-admissible weight functions, $m_{x,v}$ be $\Gamma_{x,v}$-admissible weight functions. We will write $S(m_v):=S(m_v,\Gamma_v)$, $H_v(m_v):=H_v(m_v,\Gamma_v)$ be the weighted Sobolev space on velocity variable $v$ and $H_{x,v}(m_{x,v}):=H_{x,v}(m_{x,v},\Gamma_{x,v})$ be the weighted Sobolev space on space-velocity variable $(x,v)$. 

We will use $
C_1:=\sup_{\varphi\in L^2}\frac{\|\varphi\|_{L^2}}{\|\varphi\|_{\H}}$ in Section \ref{sec2} for hard potential case. $\nu_1$ is equal to $\min\{\frac{\nu_0}{2},\frac{\nu_0}{2C_1}\}$, if $\gamma+2s\ge 0$ and to $0$ if $\gamma+2s<0$, with $\nu_0$ defined in \eqref{A_positive}.
The norm of $X_v(Y_x)$ is defined by
\begin{align*}
\|f\|_{X_v(Y_x)}:=\big\|\|f\|_{Y_x}\big\|_{X_v}.
\end{align*}

\section{Hypoelliptic estimate of the linearized Boltzmann operator}\label{sec2}

In this section, we are trying to solve equation 
\begin{align}\label{eq18}
(v\cdot\nabla_x+A)g=f,
\end{align}where $g$ is unknown and $f\in\S$. Here $A$ is a elliptic-type operator proved in the Theorem \ref{Thm0}. Thus a standard-type argument in solving elliptic equation but in the language of pseudo-differential calculus will provide the solution and regularity to equation \eqref{eq18}. 
We will also consider its dual equation
\begin{align*}
(2\pi iv\cdot \xi+A)g=f
\end{align*} in $L^2_v$, in order to apply the spectrum analysis to operator $2\pi iv\cdot \xi+A$. These result allows us to apply the semigroup theory on $L^2_{x,v}$ and $L^2_v$.

\subsection{Splitting of the linearized operator}
\begin{proof}[Proof of Theorem \ref{Thm0}]
1. We will continue the argument in \cite{Deng2020a,Global2019} and notice that the following statement are valid for both hard and soft potential. 
By Section 3 in \cite{Deng2020a}, the linearized Boltzmann operator can be written into 
 \begin{align*}
   L&=-b^w+\K^w,
 \end{align*}where $b\in S(a)$, $\K\in S(\<v\>^{\gamma+2s})$ are defined as the following.
Fix $0<\delta\le 1$. Let $\varphi(t)$ be a positive smooth radial function that equal to $1$ when $|t|\le 1/4$ and $0$ when $|t|\ge 1$.
Let $\varphi_\delta(v)=\varphi(|v|^2/\delta^2)$ and $\widetilde{\varphi}_\delta(v)=1-\varphi_\delta(v)$. Then $\varphi_\delta(v)=\varphi(|v|^2/\delta^2)$ equal to $0$ for $|v|\ge \delta$ and $1$ for $|v|\le \delta/2$. Then for $f\in\S$,
\begin{align*}  
b^wf:&=-\int\int B\varphi_\delta(v'-v)(\mu'_*)^{1/2}(f'-f)\big((\mu_*)^{1/2}-(\mu'_*)^{1/2}\big)\,dv_*d\sigma\\
&\quad-\int\int B\varphi_\delta(v'-v)\mu'_*(f'-f)\,dv_*d\sigma \\
&\quad + f(v)\int\int B\widetilde{\varphi}_\delta(v'-v)\mu_*\,dv_*d\sigma,\\
\K^wf:&=\int\int B(\mu_*)^{1/2}\left((\mu')^{1/2}f'_*-\mu^{1/2}f_*\right)\,dv_*d\sigma,\\
&\quad+ \int\int B\widetilde{\varphi}_\delta(v'-v)(\mu_*)^{1/2}(\mu'_*)^{1/2}f'\,dv_*d\sigma,\\
&\quad+f(v)\int\int B \varphi_\delta(v'-v)(\mu'_*-\mu_*)\,dv_*d\sigma,\\
&\quad+f(v)\int\int B \varphi_\delta(v'-v)(\mu'_*)^{1/2}((\mu_*)^{1/2}-(\mu'_*)^{1/2})\,dv_*d\sigma,
\end{align*}where we use Carleman representation for writing into the form of pseudo-differential operator, cf. \cite{Deng2020a,Global2019}.
These operators satisfy
\begin{equation}\label{eq13}\begin{split}
\Re(b^wf,f)_{L^2}+C\|&\<v\>^{\gamma/2+s}f\|^2_{L^2}\approx \|(a^{1/2})^wf\|^2_{L^2},\\
|(\K^w f,f)_{L^2}|&\le C\|\<v\>^{\gamma/2+s}f\|^2_{L^2}.
\end{split}
\end{equation}
By Proposition 2.1 in \cite{Alexandre2012}, there exists $\nu_0>0$ such that 
\begin{align}\label{eq17}
\nu_0\|(a^{1/2})^w(I-P)f\|^2_{L^2}\le (-Lf,f)_{L^2}.
\end{align}
Choosing $K_0>0$ in \eqref{a} sufficiently large, we have $\<v\>^{\gamma/2+s}\le a^{1/2}$ and hence, $\|\<v\>^{\gamma/2+s}f\|_{L^2}\le C\|a^{1/2}f\|_{L^2}$. By Lemma \ref{inverse_bounded_lemma}, we have  
\begin{align}
\|\<v\>^{\gamma/2+s}(I-P)f\|^2_{L^2}&\le C(-Lf,f)_{L^2},\notag\\
\|\<v\>^{\gamma/2+s}f\|^2_{L^2}&\le C ((-L+P\<v\>^{\gamma+2s}Pf,f)_{L^2}.\label{eq14}
\end{align}
In particular, when $\gamma+2s\le 0$, we have 
\begin{align*}
\|\<v\>^{\gamma/2+s}f\|^2_{L^2}&\le C((-L+P)f,f)_{L^2}.
\end{align*}
Therefore, by \eqref{eq13}\eqref{eq14}, choosing $C>0$ sufficiently large, we have 
\begin{align*}
\|(a^{1/2})^wf\|^2_{L^2}&\le C\big( \Re(b^wf,f)_{L^2}+\Re(-\K^wf,f)_{L^2}+C\|\<v\>^{\gamma/2+s}f\|^2_{L^2}\big)\\
&\le C'((-L+P\<v\>^{\gamma+2s}P)f,f)_{L^2},
\end{align*}and similarly for soft potential $\gamma+2s< 0$, 
\begin{align*}
\|(a^{1/2})^wf\|^2_{L^2}
&\le C((-L+P)f,f)_{L^2}.
\end{align*}

2. Now we can write 
\begin{equation*}
\begin{aligned}
A &:= \left\{\begin{aligned}
&-L + P\<v\>^{\gamma+2s}P,&\text{ if }\gamma+2s\ge 0, \\
&-L + P,&\text{ if }\gamma+2s\le 0.
\end{aligned}\right.\\
K &:= A+L.
\end{aligned}
\end{equation*} The symbol of $A$ belongs to $S(a)$ is given by Theorem 3.1 and 3.2 in \cite{Deng2020a}. On the other hand, for the symbol of $K$, we write
\begin{align*}
P\<v\>^{\gamma+2s}Pf &= \sum^{d+1}_{i=0}(\<v\>^{\gamma+2s}Pf,\psi_i)_{L^2}\,\psi_i(v)\\
&= \int\sum^{d+1}_{i=0} \widehat{f}(\eta)\F(P\<\cdot\>^{\gamma+2s}\psi_i)(\eta)\psi_i(v)\,d\eta\\
&= \int e^{2\pi i v\cdot\eta}\widehat{f}(\eta)\sum^{d+1}_{i=0}e^{-2\pi i v\cdot\eta}\F(P\<\cdot\>^{\gamma+2s}\psi_i)(\eta)\psi_i(v)\,d\eta.
\end{align*}Thus $K$ has symbol $\sum^{d+1}_{i=0}e^{-2\pi i v\cdot\eta}\F(P\<\cdot\>^{\gamma+2s}\psi_i)(\eta)\psi_i(v)\in S(\<v\>^{-k}\<\eta\>^{-l})$, for $k,l>0$ as a standard pseudo-differential operator, since $\psi_i\in\S$ has exponential decay. By Theorem 4.28 in \cite{Zworski2012}, we have $\lim_{|(v,\eta)|\to\infty}\<v\>^{-k}\<\eta\>^{-l}$ and thus $K$ is compact on $L^2_v$.  
\qe\end{proof}

\subsection{Hypoellipticity of $v\cdot\nabla_x+A$}
If $\gamma+2s\ge 0$, we have $\|\cdot\|_{L^2}\le C\|(a^{1/2})^w(\cdot)\|_{L^2}$ by Lemma \ref{inverse_bounded_lemma}, since $1\le a$ and $(a^{1/2})^w$ is invertible. We thus write 
\begin{align}\label{C1}
C_1:=\sup_{f\in L^2}\frac{\|\varphi\|_{L^2}}{\|\varphi\|_{\H}}.
\end{align}
\begin{Thm}\label{L2xv_regularity}
Write $L^2=L^2_{x,v}(\R^{2d})$, $\H=H_{x,v}(a^{1/2})(\R^{2d})$, $\S=\S(\R^{2d})$ in this theorem. Denote 
\begin{align*}
\|\cdot\|_H:=\|\cdot\|_{L^2}+\|\cdot\|_{\H},\quad  \|\cdot\|_{H_{-1}}:=\min\{\|\cdot\|_{L^2},\|\cdot\|_{H(a^{-1/2})}\}.
\end{align*}
Let 
\begin{equation*}\left\{
\begin{aligned}
\Re\zeta>-\frac{\nu_0}{C_1},\ \text{ if }\gamma+2s\ge 0,\\
\Re\zeta>0,\qquad \text{ if }\gamma+2s<0.
\end{aligned}\right.
\end{equation*}
Then for any $f\in\S$, there exists unique $g\in H(\Lambda_k)$ such that
\begin{align*}
(\zeta I+ v\cdot\nabla_x+A)g=f,
\end{align*}where $\Lambda_k = (K_0+|\xi|^4+|v|^2+|\eta|^2)^k$ with $K_0>>1$.
Hence, the range of $\zeta I+ v\cdot\nabla_x+A$ with domain $H(a)\cap H(\<v\>\<\xi\>)$ is dense in $L^2$.
\end{Thm}
\begin{Rem}\label{L2xv_rem}
The operator $v\cdot\nabla_x$ can be replaced by its adjoint $-v\cdot\nabla_x$.
\end{Rem}
\begin{proof}
Write $F=v\cdot\nabla_x$. Then $F$ is a pseudo-differential operator with symbol in $S(\<v\>\<\xi\>)$, $\partial_vF\in Op(\<\xi\>)$ and $(Ff,f)_{L^2}=(f,-Ff)_{L^2}$ for $f\in\S(\R^{2d})$.

1. Fix any $f\in\S$, we are going to find the strong solution of
\begin{align*}
  \zeta g + F g + Ag = f.
\end{align*}
We start with finding the weak solution.
For any $\varphi\in\S$, by \eqref{A_positive}\eqref{C1}, we have
\begin{equation*}
  \Re(\bar{\zeta}\varphi-F\varphi+A\varphi,\varphi)_{L^2}\ge \nu_0\|\varphi\|^2_{\H} + \Re\bar{\zeta}\|\varphi\|^2_{L^2}\ge C_{\zeta}\|\varphi\|^2_{H},
\end{equation*}
Then for any $f$ satisfying $\|f\|_{H_{-1}}<\infty$,
\begin{align*}
  \|\varphi&\|^2_{H}\le C\|\bar{\zeta}\varphi-F\varphi+A\varphi\|_{H_{-1}}\|\varphi\|_{H},\\
  |(f,\varphi)|&\le C\|f\|_{H_{-1}}\|\varphi\|_{H}\le C\|\bar{\zeta}\varphi-F\varphi+A\varphi\|_{H_{-1}}.
\end{align*}
Let $\Im(\bar{\zeta} I-F+A):=\{(\bar{\zeta} I-F+A)\varphi:\varphi\in\S\}$. 
The operator $T_1:H(a^{-1/2})\supset \Im(\bar{\zeta} I-F+A)\to \C$ and $T_2:L^2\supset \Im(\bar{\zeta} I-F+A)\to \C$ sending $\psi:=\bar{\zeta}\varphi-F\varphi+A\varphi$ to $(f,\varphi)$ are linear continuous.
Hence $T_1$, $T_2$ extend to a linear functional on $H(a^{-1/2})$ and $L^2$ respectively. Note that from Theorem 2.6.17 in \cite{Lerner2010}, $(\H)^*=H(a^{-1/2})$ and $(L^2)^*=L^2$, there exists unique $u_1\in H(a^{1/2})$ and $u_2\in L^2$ such that for $\psi\in \S$,
\begin{align*}
  (u_1,\psi)_{L^2} = T_1\psi,\ \ \ (u_2,\psi)_{L^2} = T_2\psi.
\end{align*}Thus $g:=u_1=u_2\in H(a^{1/2})\cap L^2$ satisfies that for $\varphi\in\S$,
\begin{align}\label{eq9}
  (g, \bar{\zeta}\varphi-F\varphi+A\varphi)_{L^2} = (f,\varphi)_{L^2}.
\end{align}

2. We next show that $g\in \HH\cap H(\<v\>\<\xi\>)$.
Let $k\in\R$, $\Lambda_k = (K_0+|\xi|^4+|v|^2+|\eta|^2)^k$ be admissible metric and $n=\max\{2k,1,\gamma+2s,2s\}$. Then by Lemma \ref{inverse_pseudo}, we choose $K_0>1$ so large that $\Lambda^w_k\in Op(\Lambda_k)$ is invertible with $(\Lambda^w_k)^{-1}\in Op(\Lambda_{-k})$. Such choice of $n$ assures $\varphi$ is good enough that the following estimates about adjoint are valid by Lemma \ref{formal_ajoint}. Also equation \eqref{eq9} is valid for $\varphi\in H(\Lambda_n)$ by density. Thus for any $\varphi\in H(\Lambda_n)$,
\begin{align}\notag
  \Re(\zeta &\varphi+F\varphi+A\varphi,\Lambda^w_{k}\Lambda^w_{k}\varphi)_{L^2}\\
  &= \Re(\Lambda^w_{k}(\zeta\varphi+F\varphi+A\varphi),\Lambda^w_{k}\varphi)_{L^2}\notag\\
  &= \Re((\zeta I+F+A)\Lambda^w_{k}\varphi,\Lambda^w_{k}\varphi)_{L^2} + \Re([\Lambda^w_{k},F+A]\varphi,\Lambda^w_{k}\varphi)_{L^2}\notag\\
  &\ge \frac{1}{C}\|\Lambda^w_{k}\varphi\|^2_{H} - |([\Lambda^w_{k},F+A]\varphi,\Lambda^w_{k}\varphi)_{L^2}|.\label{eq8}
\end{align}
Note that $\<\eta\>\<v\>\lesssim (1+|v|^2+|\eta|^2)$. 
Then by Young's inequality with $\frac{1}{3}+\frac{1}{3/2}=1$, we have
\begin{align}\label{eq10}
  \<\eta\>\<\xi\>\lesssim \<\eta\>^{3/2}+\<\xi\>^3\lesssim (1+|\eta|^{3/2}+|\xi|^3)^{\frac{4}{3}\cdot \frac{3}{4}}\lesssim (1+|\eta|^2+|\xi|^4)^{3/4}\lesssim \Lambda_{3/4}.
\end{align}
To control the commutators, by \eqref{sharp_theta} and the comments therein, we obtain  
\begin{align*}
  [\Lambda^w_k,F] &\in Op(\<\eta\>\<\xi\>\Lambda_{k-1})\subset Op(\Lambda_{k-\frac{1}{4}}),\\
  [\Lambda^w_k,A] &\in Op(a\Lambda_{k-\frac{1}{2}}),
\end{align*}and hence 
\begin{align*}
  \big|\big([\Lambda^w_k,F]\varphi,\Lambda^w_k\varphi\big)_{L^2}\big|&=\big|\big(\underbrace{(\Lambda^w_{-\frac{1}{8}})^{-1}[\Lambda^w_k,F]}_{\in Op(\Lambda_{k-1/8})}\varphi,
  \underbrace{\Lambda^w_{-\frac{1}{8}}\Lambda^w_k}_{\in Op(\Lambda_{k-1/8})}\varphi\big)_{L^2}\big|
  \lesssim \big\|\Lambda^w_{k-\frac{1}{8}}\varphi\big\|^2_{L^2},\\
  \big|\big([\Lambda^w_k,A]\varphi,\Lambda^w_k\varphi\big)_{L^2}\big|&=\big|\big(\underbrace{(\Lambda^w_{-\frac{1}{4}})^{-1}((a^{1/2})^w)^{-1}[\Lambda^w_k,A]}_{\in Op(a^{1/2}\Lambda_{k-1/4})}\varphi,
  \underbrace{\Lambda^w_{-\frac{1}{4}}(a^{1/2})^w\Lambda^w_k}_{\in Op(a^{1/2}\Lambda_{k-1/4})}\varphi\big)_{L^2}\big|\\
  &\lesssim \big\|(a^{1/2})^w\Lambda^w_{k-\frac{1}{4}}\varphi\big\|^2_{L^2}.
\end{align*}
Therefore, \eqref{eq8} becomes 
\begin{align}\label{eq12}
  \big\|\Lambda^w_k\varphi\big\|^2_{H}
  &\lesssim  \big|\big([\Lambda^w_k,F+A]\varphi,\Lambda^w_k\varphi\big)_{L^2}\big|+ \Re(\zeta\varphi+F\varphi+A\varphi,\Lambda^w_k\Lambda^w_k\varphi)_{L^2}\notag \\
  &\lesssim \big\|\Lambda^w_{k-\frac{1}{8}}\varphi\big\|^2_{L^2} + \big\|(a^{1/2})^w\Lambda^w_{k-\frac{1}{4}}\varphi\big\|^2_{L^2} + \Re\big(\varphi,(\bar{\zeta}-F+A)\Lambda^w_k\Lambda^w_k\varphi\big)_{L^2},
\end{align}where the constant depends on $\zeta$.

Let $\delta\in(0,1)$ and $\Phi_\delta = (K_0+\delta^2(|\xi|^4+|v|^2+|\eta|^2))^{-n/2}$. Then we choose $K_0>0$ sufficiently large such that $\Phi^w_\delta\in S(\Phi_{\delta})$ is invertible by Lemma \ref{inverse_pseudo}.
Choose $\varphi = \Phi^w_\delta g \in H(\Lambda_n)$, then
\begin{align}\label{eq7}
  \big|\big(\varphi,(\bar{\zeta}&-F+A)\Lambda^w_k\Lambda^w_k\varphi\big)_{L^2}\big|\notag\\
  &= \big|\big(\Phi^w_\delta g ,(\bar{\zeta}-F+A)\Lambda^w_k\Lambda^w_k\varphi\big)_{L^2}\big|\notag\\
  &\le \big|\big(g,[\Phi^w_\delta,-F+A]\Lambda^w_k\Lambda^w_k\varphi\big)_{L^2}\big|
  +\big|\big(g,(\bar{\zeta}-F+A)\Phi^w_\delta\Lambda^w_k\Lambda^w_k\varphi\big)_{L^2}\big|.
\end{align}
For the commutator $[\Phi^w_\delta,-F+A]$, by \eqref{eq10},
\begin{align*}
  [\Phi_\delta^w,-F]&\in Op\big(\delta^{1/2}\Phi_\delta\big),\ \
[\Phi_\delta^w,A]\in Op\big(\delta a\Phi_\delta\big),
\end{align*}uniformly in $\delta$ and thus
\begin{align*}
  \big|\big(g,[\Phi^w_\delta,-F]\Lambda^w_k\Lambda^w_k\varphi\big)&_{L^2}\big|
  = \big|\big(\Lambda^w_k\Phi^w_\delta g,\underbrace{(\Lambda^w_k)^{-1}(\Phi^w_\delta)^{-1} [\Phi^w_\delta,F]\Lambda^w_k\Lambda^w_k}_{\in Op(\delta^{1/2}\Lambda_k)}\Phi^w_\delta g\big)_{L^2}\big|\\
  &\quad\lesssim \delta^{1/2}\big\|\Lambda^w_k\Phi^w_\delta g\big\|^2_{L^2}\\
\big|\big(g,[\Phi^w_\delta,A]\Lambda^w_k\Lambda^w_k\varphi\big)_{L^2}\big|
&=\big|\big((a^{1/2})^w\Lambda^w_k\Phi^w_\delta g, \underbrace{((a^{1/2})^w)^{-1}(\Lambda^w_k)^{-1}(\Phi^w_\delta)^{-1}[\Phi^w_\delta,A]\Lambda^w_k}_{\in Op(\delta a^{1/2})}\Lambda^w_k\Phi^w_\delta g\big)_{L^2}\big|\\
&\lesssim \delta\big\|(a^{1/2})^w\Lambda^w_k\Phi^w_\delta g\big\|^2_{L^2}.
\end{align*}
Together with \eqref{eq9}, \eqref{eq7} becomes 
\begin{align*}
  \big|\big(\varphi,&(\bar{\zeta}-F+A)\Lambda^w_k\Lambda^w_k\varphi\big)_{L^2}\big|\\
  &\le \big|\big(g,[\Phi^w_\delta,(-F+A)]\Lambda^w_k\Lambda^w_k\varphi\big)_{L^2}\big|
  +\big|\big(g,(\bar{\zeta}-F+A)\Phi^w_\delta\Lambda^w_k\Lambda^w_k\varphi\big)_{L^2}\big|\\
  &\lesssim \delta^{1/2}\big\|\Lambda^w_k\varphi\big\|^2_{L^2}
  +\delta\big\|(a^{1/2})^w\Lambda^w_k\varphi\big\|^2_{L^2}
  +\big|\big(\Lambda^w_k\Phi^w_\delta f,\Lambda^w_k\varphi\big)_{L^2}\big|\\
  &\lesssim\delta^{1/2}\big\|\Lambda^w_k \varphi\big\|^2_{L^2}  +\delta\big\|(a^{1/2})^w\Lambda^w_k \varphi\big\|^2_{L^2} +C_\kappa\big\|\Lambda^w_k\Phi^w_\delta f\big\|^2_{H_{-1}} +\kappa\big\|\Lambda^w_k \varphi\big\|^2_{H},
\end{align*}for $\kappa>0$. 
Note that $\Phi_\delta\le 1$. Return to \eqref{eq12}, noticing $\Phi_\delta\in S(1)$ uniformly in $\delta$ and choosing $\delta,\kappa$ sufficiently small, we have $\|\Phi_{\delta}^w(\cdot)\|_{L^2}\le C\|\cdot\|_{L^2}$ and 
\begin{align*}
  \|(a^{1/2})^w\Lambda^w_k\varphi\|_{H(\Phi^w_\delta)}+\|\Lambda^w_k\varphi\|_{H(\Phi^w_\delta)}
  &\lesssim \|\Lambda^w_{k-\frac{1}{8}}g\|_{L^2} + \|(a^{1/2})^w\Lambda^w_{k-\frac{1}{4}}g\|_{L^2}
  + \|\Lambda^w_k f\|_{H_{-1}},
\end{align*}whenever the right hand side is well-defined, where the constant is independent of $\delta$.
Recall the Definition \eqref{sobolev_space} of Sobolev space $H(\Phi^w_\delta)$ and $H(1)=L^2$, we let $\delta\to 0$ to conclude that
\begin{align*}
  \|\Lambda^w_k\varphi\|_{H}
  &\lesssim \|\Lambda^w_{k-\frac{1}{8}}g\|_{H}+\|\Lambda^w_k f\|_{H_{-1}}\lesssim \cdots\lesssim \|\Lambda^w_{-n}g\|_{H}+\|\Lambda^w_k f\|_{H_{-1}}.
\end{align*}if the right hand side is well-defined. 
Since $g\in H(a^{1/2})\subset H(\Lambda_{-n})$, we obtain that $g\in H(\Lambda_k)$ for any $k\ge 0$. 
Hence $g\in H(a)\cap H(\<v\>\<\xi\>)$ if we choose $k$ sufficiently large and the range of $\zeta I+F+A$ with domain $H(a)\cap H(\<v\>\<\xi\>)$ is dense in $L^2$.
\qe\end{proof}

\begin{Thm}\label{L2v_regularity}
Write $L^2=L^2_{v}(\R^{d})$, $\H=H_{v}(a^{1/2})(\R^{d})$, $\S=\S(\R^{d})$ in this theorem. Denote 
\begin{align*}
\|\cdot\|_H:=\|\cdot\|_{L^2}+\|\cdot\|_{\H}, \quad\|\cdot\|_{H_{-1}}:=\min\{\|\cdot\|_{L^2},\|\cdot\|_{H(a^{-1/2})}\}.
\end{align*}
Let $\xi\in\R^d$ be fixed, $K_0>>1$,  $\Lambda_{k,l}:=(K_0+|v|^2)^{k/2}(K_0+|\eta|^2)^{l/2}$ and 
\begin{equation*}\left\{
\begin{aligned}
\Re\zeta>-\frac{\nu_0}{C_1},\ \text{ if }\gamma+2s\ge 0,\\
\Re\zeta>0,\qquad \text{ if }\gamma+2s<0.
\end{aligned}\right.
\end{equation*}
Suppose $f\in\S$, then there exists unique solution $g\in\S$ to 
\begin{align}\label{equation_zeta}
(\zeta I+2\pi i v\cdot \xi+A)g = f.
\end{align}
Moreover, for $k,l\ge 0$, 
  \begin{align}\label{varphi_v_beta}
    \|\<v\>^k\<D_v\>^l g\|_{H}\lesssim \|\<v\>^k\<D_v\>^lf\|_{H_{-1}}.
  \end{align}
\end{Thm}\begin{Rem}
Later, the Lemma \ref{ASchwarz} will shows that the estimate \eqref{varphi_v_beta} actually give the control for $g=(\lambda I-\A)^{-1}f$.
\end{Rem}
\begin{proof}
Notice that the constants in the following estimates will depend on $y$.
Let $F=2\pi i v\cdot \xi$. Then $F$ can be regarded as a pseudo-differential operator with symbol in $S(\<v\>)$. $\partial_vF\in Op(1)$ and $(Ff,f)_{L^2_v}=(f,-Ff)_{L^2_v}$ if $f\in H(\<v\>)$.

1. A similar argument to step one in Theorem \ref{L2xv_regularity}, with replacing $L^2_{x,v}$ by $L^2_v$ and $H_{x,v}$ by $H_v$, gives that for any $f$ with $\|f\|_{H_{-1}}<\infty$, there exists unique $g\in \H\cap L^2$ such that 
\begin{align}\label{eq133}
  (g, \bar{\zeta}\varphi-F\varphi+A\varphi)_{L^2} = (f,\varphi)_{L^2},
\end{align}for all $\varphi\in\S$.

2. 
Let $k,l\in\R$, $n=2k^++2l^++2+\max\{\gamma+2s,2s\}$ and $\Lambda_{k,l} = (K_0+|v|^2)^{k/2}(K_0+|\eta|^2)^{l/2}$ be an admissible metric, where $k^+=\max\{0,k\}$.
Then by Lemma \ref{inverse_pseudo}, we choose $K_0$ sufficiently large such that $\Lambda^w_{k,l}$ is invertible and $(\Lambda^w_{k,l})^{-1}\in S(\Lambda_{-k,-l})$.
The choice of $n$ assures $\varphi$ below is good enough that the following equations on adjoint are valid by Lemma \ref{formal_ajoint}.
Since $\S$ is dense in $H(\Lambda_{n,n})$, the equation \eqref{eq133} is valid for $\varphi\in H(\Lambda_{n,n})$. For such $\varphi$, we have $\zeta\varphi+F\varphi+A\varphi\in H(\Lambda_{k,l})$ and hence 
\begin{align}
  \Re(\zeta\varphi&+F\varphi+A\varphi,\Lambda^w_{k,l}\Lambda^w_{k,l}\varphi)_{L^2_v}\notag\\
  &= \Re(\Lambda^w_{k,l}(\zeta\varphi+F\varphi+A\varphi),\Lambda^w_{k,l}\varphi)_{L^2_v}\notag\\
  &= \Re((\zeta I+F+A)\Lambda^w_{k,l}\varphi,\Lambda^w_{k,l}\varphi)_{L^2_v} + \Re([\Lambda^w_{k,l},F+A]\varphi,\Lambda^w_{k,l}\varphi)_{L^2}\notag\\
  &\ge \frac{1}{C}\|\Lambda^w_{k,l}\varphi\|^2_{H} - |([\Lambda^w_{k,l},F+A]\varphi,\Lambda^w_{k,l}\varphi)_{L^2_v}|.\label{eq144}
\end{align}
For the commutators, we have  
\begin{align*}
  [\Lambda^w_{k,l},F+A] &= \Big(\int^1_0\partial_\eta\Lambda_{k,l}\#_\theta\partial_v(F+A)\,d\theta\Big)^w
-\Big(\int^1_0\partial_v\Lambda_{k,l}\#_\theta\partial_\eta (F+A)\,d\theta\Big)^w,
\end{align*}where we denote $F$, $A$ to be their corresponding symbols for convenience of notation. By \eqref{sharp_theta},  
\begin{align*}
  \int^1_0\partial_\eta\Lambda_{k,l}\#_\theta\partial_v(F+A)\,d\theta&\in S((1+a)\Lambda_{k,l-1})\\
  \int^1_0\partial_v\Lambda_{k,l}\#_\theta\partial_\eta(F+A)\,d\theta &\in S((1+a)\Lambda_{k-1,l}).
\end{align*}In particular when $l=0$, $\partial_\eta\Lambda_{k,l}=0$ and hence $[\Lambda^w_{k,0},F+A]\in Op((1+a)\Lambda_{k-1,0})$ while when $k=0$, $\partial_v\Lambda_{k,l}=0$ and hence $[\Lambda^w_{0,l},F+A]\in Op((1+a)\Lambda_{0,l-1})$. 
Similar to the estimate on commutators in the second step of theorem \ref{L2xv_regularity}, we have 
\begin{align*}
  \big|\big([\Lambda^w_{k,l},F+A]\varphi,\Lambda^w_{k,l}\varphi\big)_{L^2_v}\big|&\lesssim \big\|\Lambda^w_{k-\frac{1}{2},l}\varphi\big\|^2_{H}+\big\|\Lambda^w_{k,l-\frac{1}{2}}\varphi\big\|^2_{H},\\
  \big|\big([\Lambda^w_{k,0},F+A]\varphi,\Lambda^w_{k,0}\varphi\big)_{L^2}\big|&\lesssim \big\|\Lambda^w_{k-\frac{1}{2},0}\varphi\big\|^2_{H},\\
  \big|\big([\Lambda^w_{0,l},F+A]\varphi,\Lambda^w_{0,l}\varphi\big)_{L^2}\big|&\lesssim \big\|\Lambda^w_{0,l-\frac{1}{2}}\varphi\big\|^2_{H},
\end{align*}
and hence \eqref{eq144} gives 
\begin{align}
  \|\Lambda^w_{k,l}\varphi\|^2_{H}
  &\lesssim \|\Lambda^w_{k-\frac{1}{2},l}\varphi\|^2_{H}+\|\Lambda^w_{k,l-\frac{1}{2}}\varphi\|^2_{H} 
+  \Re((\zeta+F+A)\varphi,\Lambda^w_{k,l}\Lambda^w_{k,l}\varphi)_{L^2_v},\label{eq11}\\
\|\Lambda^w_{k,0}\varphi\|^2_{H}&\lesssim \|\Lambda^w_{k-\frac{1}{2},0}\varphi\|^2_{H}+\Re((\zeta+F+A)\varphi,\Lambda^w_{k,0}\Lambda^w_{k,0}\varphi)_{L^2_v},\label{eq144_1}\\
\|\Lambda^w_{0,l}\varphi\|^2_{H}&\lesssim \|\Lambda^w_{0,l-\frac{1}{2}}\varphi\|^2_{H}+\Re((\zeta+F+A)\varphi,\Lambda^w_{0,l}\Lambda^w_{0,l}\varphi)_{L^2_v}.\label{eq144_2}
\end{align}where the constant depends on $\zeta$ and $y$.
Let $\delta\in(0,1)$ and define $\Phi_\delta = (1+\delta^2(|v|^2+|\eta|^2))^{-n}$.
Choose $\varphi = \Phi^w_\delta g \in H(\Lambda_{n,n})$, then $\Phi^w_\delta\Lambda^w_{k,l}\Lambda^w_{k,l}\varphi\in H(a)\cap H(\<v\>)$ and \eqref{eq133} gives 
\begin{align*}
  ((\zeta+F+A)\varphi,\Lambda^w_{k,l}\Lambda^w_{k,l}\varphi)_{L^2_v}
  &= (\Phi^w_\delta g,(\bar{\zeta}-F+A)\Lambda^w_{k,l}\Lambda^w_{k,l}\varphi)_{L^2_v}\\
  &= (g,[\Phi^w_\delta,-F+A]\Lambda^w_{k,l}\Lambda^w_{k,l}\varphi)_{L^2_v}
  +(f,\Phi^w_\delta\Lambda^w_{k,l}\Lambda^w_{k,l}\varphi)_{L^2_v}.
\end{align*}
Here  
\begin{align*}
  [\Phi^w_\delta,-F+A]\in Op(\delta(1+a)\Phi_\delta),
\end{align*}uniformly in $\delta$ and hence similar to theorem \ref{L2xv_regularity}, we have 
\begin{align*}
  \big|\big(g,[\Phi^w_\delta,-F+A]\Lambda^w_{k,l}\Lambda^w_{k,l}\varphi\big)_{L^2}\big|
  &\lesssim \delta\|\Lambda^w_{k,l}\Phi^w_\delta g\|^2_{H}.
\end{align*}
Thus
\begin{align}\label{eq35}\notag
  \big|\big((\zeta+F+A)\varphi,\Lambda^w_{k,l}\Lambda^w_{k,l}\varphi\big)_{L^2_v}\big|
  &\lesssim \delta\|\Lambda^w_{k,l}\Phi^w_\delta g\|^2_{H}
  +\big|\big(\Lambda^w_{k,l}\Phi^w_\delta f,\Lambda^w_{k,l}\Phi^w_\delta g\big)_{L^2}\big|\\
  &\lesssim 2\delta\|\Lambda^w_{k,l}\Phi^w_\delta g\|^2_{H}
  + C_\delta\|\Lambda^w_{k,l}\Phi^w_\delta f\|^2_{H_{-1}}.
\end{align}
Substitute this into \eqref{eq11}, by picking $\delta$ sufficiently small and note that $\Phi_\delta\in S(1)$ uniformly in $\delta$, we have $\|\Phi^w_\delta(\cdot)\|_{L^2}\le C\|\cdot\|_{L^2}$ and 
\begin{align*}
 \|(a^{1/2})^w\Lambda^w_kg\|_{H(\Phi^w_\delta)}+\|\Lambda^w_kg\|_{H(\Phi^w_\delta)} 
 &\lesssim \|\Lambda^w_{k-\frac{1}{2},l} g\|_{H} +\|\Lambda^w_{k,l-\frac{1}{2}}g\|_{H}+\|\Lambda^w_{k,l}f\|_{H_{-1}},
\end{align*}whenever the right hand side is well defined, where the constant is independent of $\delta$.
Recall the Definition \eqref{sobolev_space} and let $\delta\to0$, we obtain 
\begin{align}
  \|\Lambda^w_{k,l}g\|_{H}
  &\lesssim \|\Lambda^w_{k-\frac{1}{2},l} g\|_{H} +\|\Lambda^w_{k,l-\frac{1}{2}}g\|_{H}+\|\Lambda^w_{k,l}f\|_{H_{-1}},\label{eq155}
\end{align}
whenever the right hand side is well-defined. Similarly, substituting \eqref{eq35} into \eqref{eq144_1}\eqref{eq144_2} and letting $\delta\to0$, we have 
\begin{align}\label{eq166}
\|\Lambda^w_{k,0}\, g\|_{H}&\lesssim \|\Lambda^w_{k-\frac{1}{2},0}\,g\|_{H}+\|\Lambda^w_{k,0}\, f\|_{H_{-1}},\\
\|\Lambda^w_{0,l}\, g\|_{H}&\lesssim \|\Lambda^w_{0,l-\frac{1}{2}}\,g\|_{H}+\|\Lambda^w_{0,l}\,f\|_{H_{-1}}\label{eq166_1}.
\end{align}

3. Let $k,l\in\R$, and recall $\|g\|_H<\infty$. Note that $\|\Lambda^w_{k_1,0} g\|_{L^2}\lesssim \|\Lambda^w_{k_2,0} g\|_{L^2}$ for $k_1\le k_2$. Then \eqref{eq166} yields 
\begin{align*}
\|\Lambda^w_{k,0\,} g\|_{H}
&\lesssim \|\Lambda^w_{k-\frac{1}{2},0}g\|_{H}+\|\Lambda^w_{k,0} f\|_{H_{-1}}\lesssim\cdots\lesssim\|\Lambda^w_{-n,0}g\|_{H}+\|\Lambda^w_{k,0}f\|_{H_{-1}}.
\end{align*}
Similarly \eqref{eq166_1} gives 
\begin{align*}
\|\Lambda^w_{0,l} g\|_{H}&\lesssim \|\Lambda^w_{0,-n}g\|_{H}+\|\Lambda^w_{0,l} f\|_{H_{-1}}.
\end{align*}
Finally, substitute these two estimate into \eqref{eq155},
\begin{align}
\|\Lambda^w_{k,l}g\|_{H}
  &\lesssim \|\Lambda^w_{k-\frac{1}{2},l}g\|_{H}+\|\Lambda^w_{k,l-\frac{1}{2}} g\|_{H} + \|\Lambda^w_{k,l}f\|_{H_{-1}}\notag\\
&\lesssim\cdots\notag\\
&\lesssim \|\Lambda^w_{0,l}g\|_{H}+\|\Lambda^w_{k,0} g\|_{H} + \|\Lambda^w_{k,l}f\|_{H_{-1}}\notag\\
&\lesssim \|\Lambda^w_{0,-n}g\|_{H}+\|\Lambda^w_{-n,0}g\|_{H}+\|\Lambda^w_{k,l}f\|_{H_{-1}}\notag\\
&\lesssim \|g\|_{H}+\|\Lambda^w_{k,l}f\|_{H_{-1}},\label{esti1}
\end{align}for $f$ satisfying $\|\Lambda^w_{k,l}f\|_{H_{-1}}<\infty$.
Therefore if $f\in\S$, then $g\in H(\Lambda_{k,l})$ for any $k,l\ge 0$ and by Sobolev embedding $g\in\S$ is a strong solution to \eqref{equation_zeta}. Taking inner product in \eqref{equation_zeta} with $g$, we obtain  
\begin{align*}
\Re((\zeta I+2\pi i v\cdot \xi+A)g,g)_{L^2} &= (f,g)_{L^2},\\
\|g\|^2_{H}\lesssim \Re\zeta\|g\|^2_{L^2}+\nu_0\|g\|^2_{\H}&\lesssim \|f\|_{H_{-1}}\|g\|_H.
\end{align*}Thus $\|g\|_H\lesssim \|f\|_{H_{-1}}$. Substitute this into \eqref{esti1}, we obtain \eqref{varphi_v_beta}, since $\|\<v\>^k\<D_v\>^l(\cdot)\|_{L^2}$ is equivalent to $\|\Lambda^w_{k,l}(\cdot)\|_{L^2}$. 
\qe\end{proof}

\section{Spectrum structure}\label{sec3}
With the tools in Section \ref{sec2}, we can give the rigorous proof of generating strongly continuous semigroup as well as the spectrum structure to $\B$ in this section. We will denote $L^2=L^2(\R^d_v)$ in this section. 

Recall that $C_1:=\sup_{f\in L^2}\|\varphi\|_{L^2}/\|\varphi\|_{\H}$.
Let $\zeta\in\C$ satisfies $\Re\,\zeta>-2\nu_1$ with
\begin{equation}\label{nu1}\nu_1=\left\{
\begin{aligned}
\min\{\frac{\nu_0}{2},\frac{\nu_0}{2C_1}\},\ \text{ if }\gamma+2s\ge 0,\\
0,\ \text{ if }\gamma+2s<0,
\end{aligned}\right.
\end{equation}
where $\nu_0>0$ is given in \eqref{eq17}.

\subsection{Generating strongly continuous semigroup}
By Theorem \ref{L2v_regularity} and \eqref{A_positive}, the operator $-\zeta-2\pi i \xi\cdot v-A$ with domain $D(-\zeta-2\pi i \xi\cdot v-A):=\S$ on Banach space $L^2(\R^d_v)$ is densely defined, dissipative and has dense range in $L^2_v$. Thus by Theorem \ref{semigroup}, its closure $\overline{-\zeta-2\pi i \xi\cdot v-A}$ generates a contraction semigroup. Since $K$ is compact on $L^2$, by Theorem \ref{semigroup_bounded_per}, $\overline{-\zeta-2\pi i \xi\cdot v+L}$ generates a strongly continuous semigroup on $L^2(\R^d_v)$. 

Similarly, Theorem \ref{L2xv_regularity} and \eqref{A_positive} show that the operator $-\zeta I- v\cdot\nabla_x-A$ with domain $H(a)\cap H(\<v\>\<\xi\>)$ on $L^2(\R^{2d}_{x,v})$ is densely defined, dissipative and has dense range. Thus its closure $\overline{-\zeta I- v\cdot\nabla_x-A}$ generates a contraction semigroup and $\overline{-\zeta I- v\cdot\nabla_x+L}$ generates a strongly continuous semigroup on $L^2(\R^{2d}_{x,v})$. 

Recall the Definition \ref{closure} of the closure, and that $K$ is continuous, we can define 
\begin{align}
\A:&=\overline{-2\pi i \xi\cdot v-A},\notag\\
\B:&=\overline{-2\pi i\xi\cdot v+L} = \overline{-2\pi i \xi\cdot v-A}+K,\notag\\
B:&=\overline{-v\cdot\nabla_x+L},\label{eq33}
\end{align}with $D(\B):=D(\A)$. Then $\A$, (resp. $\B$, $B$) generates strongly continuous semigroup on $L^2(\R^d_v)$, (resp. $L^2_v$, $L^2_{x,v}$).
By Hille-Yoshida Theorem, for $\Re\lambda>-2\nu_1$, $(\lambda I-\A)^{-1}:L^2_v\to L^2_v$ is linear continuous and there exists $C>0$ such that for $\Re\lambda>C$, $(\lambda I-\B)^{-1}:L^2_v\to L^2_v$ and $(\lambda I-B)^{-1}:L^2_{x,v}\to L^2_{x,v}$ are linear continuous.

By \eqref{A_positive}, we have for $f,g\in\S$, 
\begin{align*}
\|f\|^2_{\H}+\|f\|^2_{L^2}&\lesssim \Re((I+2\pi i \xi\cdot v+A)f,f)_{L^2},\\
((-2\pi i \xi\cdot v-A)f,g)_{L^2} &= (f,(2\pi i \xi\cdot v-A)g)_{L^2},
\end{align*}where $I$ is the identity mapping.
Applying Lemma \ref{adjointofclosure} below to $I+2\pi i \xi\cdot v+A$, we have 
\begin{align}\label{Aadjoint}
\A = \widehat{A}(-y)^*,\\
\B = \widehat{B}(-y)^*,\notag
\end{align}since $I$ and $K$ are self-adjoint bounded operator on $L^2$.

\begin{Lem}\label{adjointofclosure}
Let $(A,D(A))$, $(B,D(B))$ be two densely defined linear operator on complex Hilbert space $(H,\|\cdot\|)$ such that $D(A)=D(B)$, $\overline{\Im(A)}=\overline{\Im(B)}=H$. Let $C>0$ and suppose for $f,g\in D(A)$, 
\begin{align*}
\Re(Af,f)&\ge \frac{1}{C}\|f\|^2,\\
(Af,g)&=(f,Bg).
\end{align*}
Then $\overline{A} = B^*$ and $\overline{A}^{-1} = \overline{A^{-1}}$ is continuous on $H$. Also for $f\in D(\overline{A})$, 
\begin{align*}
\Re(\overline{A}f,f)&\ge \frac{1}{C}\|f\|^2.
\end{align*}
\end{Lem}
\begin{proof}
From the assumption, we have for $f\in D(A)$ that 
\begin{align*}
\|f\|\le C\|Af\|,\ \|f\|\le C\|Bf\|.
\end{align*}
Thus $A,B$ are injective and $A^{-1}$, $B^{-1}$ are densely defined operator and for $f=A\phi\in D(A^{-1})=\Im(A)$, $g=B\psi\in D(B^{-1})=\Im(B)$, with $\phi,\psi\in D(A)=D(B)$, we have 
\begin{align*}
\|A^{-1}f\|\le C\|f\|,\  \ &\|B^{-1}g\|_{L^2}\le C\|g\|,\notag\\
(A^{-1}f,g)=(\phi,B\psi)&=(A\phi,\psi)=(f,B^{-1}g).
\end{align*}
In particular, 
\begin{align*}
\Re(A^{-1}f,f)=\Re(\phi,B\phi)\ge \frac{1}{C}\|\phi\|^2=\frac{1}{C}\|A^{-1}f\|^2.
\end{align*}
Thus, $A^{-1},B^{-1}$ are closable and $\overline{A^{-1}},\overline{B^{-1}}$ are linear bounded operator on $H$. Indeed, for $f\in L^2$, there exists $f_n\in D(A)$ such that $f_n\to f$ in $H$ as $n\to\infty$ and hence $A^{-1}f_n$ is Cauchy in $H$. So $(f_n,A^{-1}f_n)\in G(A)$ converges and $f\in D(\overline{A^{-1}})$. Thus $\overline{A^{-1}}$ is an closed operator defined on the whole space $H$, and hence continuous on $H$. $\overline{B^{-1}}$ is similar. 
Thus, for $f,g\in H$, by density argument, 
\begin{align*}
(\overline{A^{-1}}f,g) &= (f,\overline{B^{-1}}g) = (\overline{B^{-1}}^*f,g) = ((B^*)^{-1}f,g),\\
(\overline{A^{-1}}f,f)&\ge \frac{1}{C}\|\overline{A^{-1}}f\|^2.
\end{align*}That is $\overline{A^{-1}}=(B^*)^{-1}$ is injective on $H$.
On the other hand, $(\overline{A^{-1}})^{-1} = \overline{A}$, since the graph 
\begin{align*}
G((\overline{A^{-1}})^{-1})
&= \{(g,(\overline{A^{-1}})^{-1}g):g\in D((\overline{A^{-1}})^{-1})=\Im(\overline{A^{-1}})\}\\
&= \{(\overline{A^{-1}}f,f):f\in L^2\}\\
&= \overline{\{(A^{-1}f,f):f\in D(A^{-1})=\Im(A)\}}\\
&= \overline{\{(g,Ag):g\in D(A)\}}\\
&= G(\overline{A}). 
\end{align*}
Thus $A$ is closable and $\overline{A} = (\overline{A^{-1}})^{-1} = B^*$. Also for $f\in D(\overline{A})=\Im(\overline{A}^{-1})$, there exists $g\in D(\overline{A}^{-1})=H$ such that $f=\overline{A}^{-1}g$ and hence 
\begin{align*}
\Re(\overline{A}f,f)= \Re(g,\overline{A}^{-1}g)\ge\frac{1}{C}\|\overline{A}^{-1}g\|^2=\frac{1}{C}\|f\|^2.
\end{align*}
\qe\end{proof}

By the definition of closure, $\A$ is also a dissipative operator. Then we have the following basic boundedness on $(\lambda I-\A)^{-1}$.
\begin{Lem}\label{ASchwarz}
(1). $D(\A)\subset\H$ and 
\begin{align*}
\|f\|^2_{\H} \le \frac{1}{\nu_0}\Re(-\A f,f)_{L^2_v}
\end{align*}
(2). Let $\Re\lambda\ge-\nu_1$, $f\in H(a^{-1/2})\cap L^2$, then
$(\lambda I-\A)^{-1}f\in \H\cap L^2$ and  
\begin{align}
\|(\lambda I-\A)^{-1}f\|_{\H}&\le\frac{2}{\nu_0} \|f\|_{H(a^{-1/2})},\label{esti_on_H}\\
\|(\lambda I-\A)^{-1}f\|_{L^2}&\le \frac{2}{\nu_1}\|f\|_{L^2}.\label{estiA}
\end{align}
Note that the constants are independent of $\lambda$. In particular, if $f\in\S$, then $(\lambda I-\A)^{-1}f\in\S$.
\end{Lem}
\begin{proof}
1. By the non-positiveness \eqref{A_positive} of $A$, for $g\in\S$, we have 
\begin{align*}
\Re(-2\pi i\xi\cdot vg-Ag,g)_{L^2_v}&\le -\nu_0\|g\|^2_{\H}.
\end{align*}
Fix $f\in D(\A)$. Note that since the graph $G(\A)=\overline{G(-2\pi i\xi\cdot v-A)}$, there exists $f_n\in D(-2\pi i\xi\cdot vf-A)$ such that $f_n\to f$ and $\A f_n\to \A f$ in $L^2$. Thus $\{f_n\}_{L^2_v}$, $\{\A f_n\}_{L^2_v}$ are bounded set and hence  
\begin{align*}
\|f_n\|^2_{\H}\le\frac{\Re(-\A f_n,f_n)_{L^2_v}}{\nu_0}
\end{align*}is bounded. By Banach-Alaoglu Theorem, $f_n$ is weakly* compact in $\H$. That is, there exists a sub-sequence $\{f_{n_k}\}\subset \{f_n\}$ and $g\in\H$ such that for $\varphi\in\S$, 
\begin{align*}
(f_{n_k},\varphi)_{\H}\to (g,\varphi)_{\H} = (g,(a^{1/2})^w(a^{1/2})^w\varphi)_{L^2_v}, \text{ as } n_k\to\infty.
\end{align*}
For $\psi\in\S$, choose $\varphi = ((a^{1/2})^w)^{-1}((a^{1/2})^w)^{-1}\psi\in\S$, then 
\begin{align*}
(f_{n_k},\psi)_{L^2} \to (g,\psi)_{L^2},\text{ as }n_k\to\infty.
\end{align*}
On the other hand, $f_n\to f$ in $L^2$. Thus $f=g\in\H$ and 
\begin{align*}
\nu_0\|f\|^2_{\H} &\le \nu_0\liminf_{n_k\to\infty}\|f_{n_k}\|^2_{\H}\\&\le \liminf_{n_k\to\infty}\Re(-\A f_{n_k},f_{n_k})_{L^2_v}\\&\le\Re(-\A f,f)_{L^2_v}.
\end{align*}
Therefore, $f\in\H$ and so $D(\A)\subset\H$.
Now we let $f\in L^2\cap H(a^{-1/2})$, $\varphi:=(\lambda I-\A)^{-1}f$, then 
\begin{align*}
\nu_0\|\varphi\|^2_{\H}+\Re\lambda\|\varphi\|^2_{L^2}&\le \Re((\lambda-\A) \varphi,\varphi)_{L^2_v}\\&\le |(f,\varphi)_{L^2}|\\&\le \|f\|_{H(a^{-1/2})}\|\varphi\|_{\H},\\
\|\varphi\|_{\H}&\le\frac{2}{\nu_0} \|f\|_{H(a^{-1/2})},
\end{align*}by our choice \eqref{nu1} of $\nu_1$ and $\Re\lambda\ge-\nu_1$.

2. By the first step of Theorem \ref{L2v_regularity}, for $\Re\lambda\ge-\nu_1$, $f\in H(a^{-1/2})$, there exists unique $g\in\H\cap L^2$ such that for $\varphi\in\S$, 
\begin{align*}
(g,(\lambda-2\pi i \xi\cdot v+A)\varphi)_{L^2_v} = (f,\varphi)_{L^2_v}. 
\end{align*}
Notice $(\lambda-\A)(\lambda I-\A)^{-1}f=f$ and use \eqref{Aadjoint}, we have 
\begin{align*}
((\lambda I-\A)^{-1}f,(\lambda-2\pi i \xi\cdot v+A)\varphi)_{L^2_v} = (f,\varphi)_{L^2_v}. 
\end{align*}
For $\psi\in\S$, again by Theorem \ref{L2v_regularity}, there exists $\varphi\in\S$ such that $(\lambda-2\pi i \xi\cdot v+A)\varphi=\psi$. Combine the above two identity, we have for $\psi\in\S$, 
\begin{align*}
&((\lambda I-\A)^{-1}f-g,\psi)_{L^2_v}=0,\\
&(\lambda I-\A)^{-1}f = g\in\H\cap L^2.
\end{align*}
In particular when $f\in\S$, by Theorem \ref{L2xv_regularity}, we have $g\in\S$. 
Also, since $\frac{3\nu_1}{2}+\A$ generates a strongly continuous semigroup, by Hille-Yoshida Theorem, we have 
  \begin{align*}
    \|(\lambda I-\frac{3\nu_1}{2}-\A)^{-1}f\|_{L^2}\le \frac{1}{\Re\lambda}\|f\|_{L^2},&\text{ for } \Re\lambda>0,\\
    \|(\lambda I-\A)^{-1}f\|_{L^2}\le \frac{1}{\Re\lambda+\frac{3\nu_1}{2}}\|f\|_{L^2},&\text{ for } \Re\lambda>-\frac{3\nu_1}{2}.
  \end{align*}This proves \eqref{estiA}.
\qe\end{proof}

\subsection{Spectrum structure of $\widehat{B}(\xi)$}
We next analyze the spectrum structure of $\B$ for hard potential $\gamma+2s\ge 0$. This theorem yields the essential spectrum of $\B$ lies in $\{\lambda:\Re\lambda\le -\nu_1\}$, while the complementary set contains only eigenvalues of $\B$ which lies in $\{\lambda\in\C:-\nu_1+\delta<\Re\lambda\le0\}$. That is, $\B$ has spectral gap. Also the corresponding eigenfunctions are Schwartz function.

\begin{Thm}\label{spectrum_struture}Write $\lambda = \sigma +i\tau$. Suppose $\gamma+2s \ge 0$. 

(1). For any $\xi\in\R^d$, 
\begin{align*}
\sigma(\B)\subset \{\lambda\in\C:\Re\lambda\le 0\}.
\end{align*}

(2). There exists $\tau_1,\xi_1>0$ such that for $\xi\in\R^d$, 
\begin{align*}
\sigma(\B)\cap\{\lambda\in\C:-\nu_1<\Re\lambda\le0\}\subset\{\lambda\in\C:|\Im\lambda|\le\tau_1\},
\end{align*} and for $|\xi|\ge \xi_1$, 
\begin{align}\label{sigmaB_ylarge}
\sigma(\B)\cap\{\lambda\in\C:-\nu_1<\Re\lambda\le0\}=\emptyset.
\end{align}

(3). For any $\delta>0$, the set
\begin{align*}
\sigma(\B)\cap\{\lambda\in\C:-\nu_1+\delta<\Re\lambda\le0\}
\end{align*}consists of finitely many discrete eigenvalues of finite type without accumulation point.
If $\lambda\in\sigma(\B)\cap\{\lambda:-\nu_1<\Re\lambda\le0\}$ is an eigenvalue and $f\in D(\B)$ is the corresponding eigen-function, then $f\in\S$. Also 
\begin{equation*}
\sigma(\B)\cap\{\Re\lambda=0\}=\left\{
\begin{aligned}
\emptyset\ ,\text{ if }\xi\neq 0,\\
\{0\},\text{ if }\xi=0.
\end{aligned}\right.
\end{equation*}In particular, if $\xi=0$, then $\Ker\overline{L}=\Ker L$.

(4). For any $\xi_2$, there exists $\sigma_1\in (0,\nu_1)$ such that for $|\xi|\ge \xi_2$, 
\begin{align*}
\sigma(\B)\cap\{\lambda\in\C:-2\sigma_1\le\Re\lambda\le0\}=\emptyset. 
\end{align*}
Figure \ref{fig1} gives the localization of spectrum with respect to $|\xi|$ for statement (2), (3) and (4). 
\begin{figure}[htbp]
	\centering
	\subfigure[$|\xi|\in \R^3$]{
		\begin{minipage}[t]{0.25\linewidth}
			\centering
			\begin{figure}[H]
				\centering
				\begin{tikzpicture}[>=Stealth,scale=0.4]
					\draw[thick,->] (-6,0)--(4,0);
					\draw[thick,->] (0,-4)--(0,4);
					\draw[dashed,thick,blue] (-2.5,-3.8)--(-2.5,3.8);
					\draw (-2,-4) node[black,below]{\tiny $\Re\lambda=-\nu_1$}; 
					\draw (0,2.4) node[black,below]{\quad\tiny\ $\tau_1$};  
					\draw (0,-1.6) node[black,below]{\quad\tiny\ \ $-\tau_1$}; 
					\fill [color=blue,pattern=north east lines, pattern color = blue] (-2.5,-2) rectangle (0,2);
					\fill [color=blue,pattern=north east lines, pattern color = blue] (-6,-4) rectangle (-2.5,4);
				\end{tikzpicture}
			\end{figure} 
		\end{minipage}%
	}%
	\hspace{3mm}
	\subfigure[$|\xi|\ge \xi_1$]{
		\begin{minipage}[t]{0.25\linewidth}
			\centering
			\begin{figure}[H]
				\centering
				\begin{tikzpicture}[smooth,>=Stealth,scale=0.4]
					\draw[thick,->] (-6,0)--(4,0);
					\draw[thick,->] (0,-4)--(0,4);
					\draw[dashed,thick,blue] (-2.5,-3.8)--(-2.5,3.8);
					\draw (-2,0) node[black,below]{\tiny\ $-\nu_1$}; 
					\fill [color=blue,pattern=north east lines, pattern color = blue] (-6,-4) rectangle (-2.5,4);
				\end{tikzpicture}
			\end{figure}
		\end{minipage}%
	}%
\\
\centering
\subfigure[$\xi=0$]{
	\begin{minipage}[t]{0.25\linewidth}
		\centering
		\begin{figure}[H]
			\centering
			\begin{tikzpicture}[>=Stealth,scale=0.4]
				\draw[thick,->] (-6,0)--(4,0);
				\draw[thick,->] (0,-4)--(0,4);
				\draw[dashed,thick,blue] (-2.5,-3.8)--(-2.5,3.8);
				\draw[blue] (0,0) node[black,below right]{\tiny $0$};
				\filldraw[blue] (0,0) circle(2mm);
				\fill [color=blue,pattern=north east lines, pattern color = blue] (-6,-4) rectangle (-2.5,4);
			\end{tikzpicture}
		\end{figure} 
	\end{minipage}%
}%
\hspace{3mm}
\subfigure[$|\xi|\ge \xi_2$]{
	\begin{minipage}[t]{0.25\linewidth}
		\centering
		\begin{figure}[H]
			\centering
			\begin{tikzpicture}[smooth,>=Stealth,scale=0.4]
				\draw[thick,->] (-6,0)--(4,0);
				\draw[thick,->] (0,-4)--(0,4);
				\draw[dashed,thick,blue] (-2.5,-3.8)--(-2.5,3.8);
				\draw[blue] (-2.7,0) node[black,below right]{\tiny $-2\sigma_1$};
				\fill [color=blue,pattern=north east lines, pattern color = blue] (-6,-4) rectangle (-2.5,4);
			\end{tikzpicture}
		\end{figure}
	\end{minipage}%
}%
	\centering
	\caption{Spectrum for $\B$ in the complex plane with the parameter $\xi\in\Rd$}
	\label{fig1}
\end{figure}
\end{Thm}
\begin{proof}
1.
For $\Re\lambda> 0$, $\zeta>-\nu_1$, by Hille-Yoshida Theorem, $\lambda +\zeta-\A$ is invertible on $L^2$ and hence 
\begin{align*}
\sigma(\A)\subset \{\lambda\in\C:\Re\lambda\le -\nu_1\}.
\end{align*}
Since $K$ is compact on $L^2$, $\B=\A+K$, Corollary XVII.4.4 in \cite{Gohberg1990} gives that essential spectrum of $\B$ is contained in the essential spectrum of $\A$:
\begin{align*}
\sigma_\text{ess}({\B})\subset \sigma(\A)\subset \{\lambda\in\C:\Re\lambda\le -\nu_1\},
\end{align*}
while $\sigma(\B)\cap\{\lambda\in\C:\Re\lambda> -\nu_1\}$ consists of eigenvalues of finite type of $\B$ with possible accumulation point on $\{\Re=-\nu_1\}$.

2. Lemma \ref{Largeness} shows that there exists $\tau_1\equiv \xi_1>0$ such that for $|\Im\lambda|+|\xi|\ge\tau_1$, $\|(\lambda I-\A)^{-1}K\|_{\mathscr{L}(L^2_v)}\le 1/2$ and thus $(I-(\lambda I-\A)^{-1}K)^{-1}$ exists with
\begin{align}\label{eq25}
\|(I-(\lambda I-\A)^{-1}K)^{-1}\|_{\mathscr{L}(L^2)}\le 2.
\end{align}
For $\delta>0$, $\xi\in\R^d$, since $\rho(\A)\supset\{\lambda\in\C:\Re\lambda> -\nu_1\}$ and 
\begin{align*}
(\lambda I-\B)^{-1} = (I-(\lambda I-\A)^{-1}K)^{-1}(\lambda I-\A)^{-1},
\end{align*}
we have $\sigma(\B)\cap \{\lambda:-\nu_1+\delta\le\Re\lambda\le 0\}\subset \{\lambda:|\Im\lambda|\le \tau_1\}$, as a bounded set, consists of discrete eigenvalues without accumulation point and hence the number of such eigenvalues is finite. 
On the other hand, for $|\xi|\ge \xi_1$, $\Re\lambda\ge -\nu_1+\delta$, we have that $I-(\lambda I-\A)^{-1}K$ and $\lambda I-\A$ are always invertible and hence $\sigma(\B)\cap \{\lambda:-\nu_1+\delta\le\Re\lambda\le 0\}=\emptyset$.

3. 
Note that for $f\in \S$, we have $\Re(f,(-2\pi i\xi\cdot v+L)f)_{L^2}\le0$. Thus by definition of closure, 
\begin{align*}
\{\Re(f&,\B f)_{L^2}:f\in D(\B)\}\\ &\subset \overline{\{\Re(f,(-2\pi i\xi\cdot v+L)f)_{L^2}:f\in \S\}}\subset (-\infty,0].
\end{align*}
Let $\lambda\in \sigma(\B)$ satisfies $\Re\lambda>-\nu_1$, then $\lambda$ is an eigenvalue to $\B$. Suppose $f\in D(\B)$ is the corresponding eigen-function, then 
\begin{align*}
\B f &= \lambda f,\\
(\B f,f)_{L^2}&=\lambda \|f\|^2_{L^2}.
\end{align*}
Taking the real part, we have $\Re\lambda\le 0$.

For this eigen-function $f\in D(\B)\subset\H\cap L^2$, we have 
\begin{align*}
\A f =  Kf-\lambda f.
\end{align*}
By \eqref{Aadjoint}, we have for $\varphi\in \S$ that 
\begin{align*}
(\lambda f-\A f,\varphi)_{L^2} = (f,(\overline{\lambda}-2\pi i\xi\cdot v+A)\varphi)_{L^2} = (Kf,\varphi)_{L^2}
\end{align*}
Thus $f$ is a weak solution of $(\lambda+2\pi i\xi\cdot v+A)f=Kf$, i.e. \eqref{eq133}. 
Notice that from Theorem \ref{Thm0}, $K\in S(\<v\>^{-n}\<\eta\>^{-n})$ is a good pseudo-differential operator.
Apply the estimate \eqref{varphi_v_beta} to $f$, for $k,l\ge 0$, we have 
\begin{align*}
\|(a^{1/2})^w\<v\>^k\<D_v\>^lf\|_{L^2_v}
&\lesssim\|(a^{-1/2})^w\<v\>^k\<D_v\>^lKf\|_{L^2_v}\\
&\lesssim \|f\|_{L^2_v}.
\end{align*}whenever the right hand side is well-defined.  
Thus $f\in H(\<v\>^k\<\eta\>^l)$ for any $k,l\in\R$ and hence belongs to $\S$ by Sobolev embedding theorem. 
Now $f\in\S$ is smooth enough that the closure in $\A$ can be canceled and the eigen-equation becomes
\begin{align*}
(-2\pi i \xi\cdot v+L)f &= \lambda f,\\
(Lf,f)_{L^2} &= \Re\lambda\|f\|_{L^2}.
\end{align*}
If $\Re\lambda = 0$, then $Lf=0$ and $-2\pi i \xi\cdot vf=i\Im\lambda f$, which implies $y=\lambda=0$.
Also $\overline{L}f=0$ yields $f\in \S$, hence $\Ker\overline{L}\in \Ker L$.

4. It suffices to prove (4) when $\xi_2\in(0,\xi_1)$. We claim that there exists $0<\sigma_1<\nu_1$ such that for $|\xi|\in[\xi_2,\xi_1]$,  $\sigma(\B)\cap\{\lambda:-2\sigma_1\le\Re\lambda\le 0\}=\emptyset$. 
We prove this by contradiction. Suppose this fails, then for $n\in\N$, there exists eigenvalues 
\begin{align*}
\lambda_n\in\sigma(\B),\ \ \xi_n\in\R^d,
\end{align*}  with $-\frac{1}{n}\le\Re\lambda\le 0$, $|\xi_n|\in[\xi_2,\xi_1]$. Since $\sigma(\B)\subset\{\lambda:|\Im\lambda|\le\tau_1\}$, we find that $\{\lambda_n\}$ is a bounded sequence.
Let $f_n\in D(\B)$ be the corresponding eigen-function to $\lambda_n$ such that $\|f_n\|_{L^2}=1$. 
Then by statement (3), we have $f_n\in\S$ and hence 
\begin{align}
&(2\pi i\xi_n\cdot v+L)f_n=\lambda_n f_n,\label{eq48}\\
&\Re\lambda_n = (Lf_n,f_n)_{L^2} \le -\nu_0\|P_1f_n\|_{L^2}^2,\notag
\end{align}by \eqref{eq17} for some $\nu_0>0$. Thus $\lim_{n \to\infty}\|P_1f_n\|_{L^2}= 0$ and $\lim_{n \to\infty}\|P_0f_n\|_{L^2}=1$. Since \Ker$L$ is a finite dimensional space, $P_0f_n$ converges in $L^2$. Thus up to a sub-sequence, we obtain 
\begin{align*}
\xi_n\to \xi_0,\ \lambda \to i\lambda_0,\ f_n\to f_0,\text{ as }n\to\infty,
\end{align*}for some $\xi_0\in[\xi_2,\xi_1]$, $\lambda_0\in\R$, $f_0\in\Ker L$. 
Also for $\varphi\in\S$,  $(Lf_n-Lf_0,\varphi)_{L^2}=(f_n-f_0,L\varphi)_{L^2}\to 0$. Up to a sub-sequence, we have $\lim_{n \to\infty}Lf_n=Lf_0$ almost everywhere. Thus taking limit in \eqref{eq48},
\begin{align*}
(2\pi i\xi_0\cdot v+L)f_0=2\pi i\xi_0\cdot vf_0 =i\lambda_0 f_0.
\end{align*}Hence $\lambda_0=\xi_0=0$, which contradicts to $\xi_0\in[\xi_2,\xi_1]$.
Therefore, there exists $\sigma_1>0$ such that for $|\xi|\in[\xi_2,\xi_1]$, 
\begin{align*}
\sigma(\B)\cap\{\lambda:-2\sigma_1\le\Re\lambda\le 0 \}=\emptyset. 
\end{align*}
Together with \eqref{sigmaB_ylarge}, we prove (4).
\qe\end{proof}

The existence of eigenvalues and eigenfunctions to $\B$ with expansions and derivatives have been well studied in different contexts, cf. \cite{Yang2016,Liu2011}.
But since our Definition \eqref{L1} \eqref{L2} on $L$ and $\A$, $\B$ are different from the above works, we provide a slightly different proof to make our argument self-contained. 
\begin{Thm}\label{eigenstruture}Assume $\gamma+2s\ge 0$ and write $\lambda = \sigma+i\tau$. 

  (1). There exists $\xi_0>0$, $\sigma_0\in(0,\nu_1/2)$ and $\lambda_j(|\xi|)\in C^\infty([0,\xi_0])$ such that for $|\xi|\le \xi_0$,
  \begin{align*}
    \sigma(\B)&\cap \{\lambda:\Re\lambda\ge -\nu_1\}=\{\lambda_j(|\xi|)\}_{j=0}^{d+1},\\
\rho(\B)&\supset \{\lambda:-2\sigma_0\le\Re\lambda\le -\frac{\sigma_0}{2}\}
\cup\{\lambda:\Re\lambda\ge-2\sigma_0,|\lambda|\ge \frac{\sigma_0}{2}\}.
  \end{align*}
  The eigenvalues $\lambda_j(\xi)$ and corresponding eigenfunctions $\varphi_j(\xi)$ have the asymptotic expansions:
  \begin{align*}
    &\lambda_j(\xi) = -2\pi i\eta_{0,j}|\xi| + \eta_{1,j}|\xi|^2 + O(|\xi|^3),\ (|\xi|\to0),\\
	&\varphi_j(y,v)= \varphi_{0,j} + |\xi|\varphi_{1,j}(\xi),\notag
  \end{align*}
  with $\eta_{0,j}\in\R$ selected from Lemma \ref{eigenP0} and $\eta_{1,j}<0$, $\varphi_{0,j},\varphi_{1,j}(\xi)\in\S$. The structure of $\rho(\B)$ is given in Figure \ref{fig2}. 
  
  \begin{figure}[htbp]
  	\centering
  		\begin{minipage}[t]{0.25\linewidth}
  			\centering
  			\begin{figure}[H]
  				\centering
  				\begin{tikzpicture}[smooth,>=Stealth,scale=0.46]
  					\draw[thick,->] (-6,0)--(4,0);
  					\draw[thick,->] (0,-4)--(0,4);
  					\draw[blue] (-2.7,0) node[black,below right]{\tiny $\frac{\sigma_0}{2}$};
  					\draw[blue] (-4,0) node[black,below left]{\tiny $-2\sigma_0$};
  					\draw[blue] (0,2) node[black,below left]{\tiny $\frac{\sigma_0}{2}$};
  					\draw[blue] (0,-2) node[black,above left]{\tiny $-\frac{\sigma_0}{2}$};
  					\fill [color=blue,pattern=north east lines, pattern color = blue] (-4,-4) rectangle (-2.5,4);
  					\fill [color=blue,pattern=north east lines, pattern color = blue] (-2.5,-4) rectangle (0,-2);
  					\fill [color=blue,pattern=north east lines, pattern color = blue] (-2.5,2) rectangle (0,4);
  					\fill [color=blue,pattern=north east lines, pattern color = blue] (0,-4) rectangle (4,4);
  				\end{tikzpicture}
  			\end{figure}
  		\end{minipage}%
  	\centering
  	\caption{Resolvent set for $\B$ in the complex plane with the parameter $\xi\in\Rd$}
  	\label{fig2}
  \end{figure}
  
(2). Denote the eigen-projection to the eigenvalue $\lambda_j(\xi)$ by $P_j(\xi)$. Then $P_j$ is of finite rank and there exists $C>0$ such that for $|\xi|\le \xi_0$, 
\begin{align*}
\|P_j(\xi)f\|_{\H}\le C\|f\|_{H(a^{-1/2})}.
\end{align*}

\end{Thm}
\begin{proof}
Let $\varphi\in D(\B)$ be the eigenfunction corresponding to eigenvalue $\lambda\in \sigma(\B)$ with $\Re\lambda\in(-\nu_1,0]$, then $\varphi\in\S$ by Theorem \ref{spectrum_struture} (3). 
Hence the eigenvalue problem $\B \varphi=\lambda\varphi$ becomes 
\begin{align*}
(-2\pi i\xi\cdot v+L)\varphi = \lambda\varphi,
\end{align*}where $\lambda=\lambda(\xi)$ is the eigenvalue and $\varphi=\varphi(\xi)$ is the eigenfunction. 

1. We claim that $\lambda(\xi)$ depends only $|\xi|$. 
Indeed, let $R$ be an orthogonal matrix acting on $v$, i.e. $R$ is a rotation. 
We denote $R^*\varphi$ to be the pull back of $\varphi$ by $R$ as 
\begin{align*}
	R^*\varphi(v) = \varphi(Rv). 
\end{align*}
Then by Lemma \ref{basicL},
\begin{align*}
\lambda(\xi)R^*\varphi = -2\pi i\xi\cdot Rv\,\varphi(Rv) + R^*L\varphi = -2\pi iR^T \xi\cdot vR^*\varphi + LR^*\varphi.
\end{align*}
Here $R^T$ is the transpose of $R$. 
Thus $\lambda(\xi)$ is an eigenvalue of both $\widehat{B}(R^T\xi)$ and $\B$. So $\lambda$ depends only on $|\xi|$ and we will write $\lambda(\xi)=\lambda(|\xi|)$ for convenience.
Now we pick a orthogonal matrix $R$ such that $R^T\xi=(|\xi|,0,\cdots,0)\in\Rd$, then $R(\B\varphi)=R(\lambda\varphi)$ becomes 
\begin{align*}
(2\pi i|\xi|v_1+L)R^*\varphi(\xi) = \lambda(|\xi|)R^*\varphi(\xi),
\end{align*}and hence $R^*\varphi(\xi)$ depends only of $|\xi|$. $R\varphi(\xi)$ is the eigenfunction of $2\pi i|\xi|v_1+L$. 

2. Denote $r=|\xi|$. Notice that we assume $r\in\R$ in the following, though $|\xi|\ge 0$. For $-\nu_1<\Re\lambda\le 0$, next we will solve eigen-equation:
\begin{align}\label{46}
(2\pi irv_1+L)\varphi = \lambda(r)\varphi.
\end{align}
Applying the Macro-micro decomposition, i.e. projection $P_0$ and $P_1$ on \eqref{46}, we obtain 
\begin{equation}\label{eq38}
\left\{
\begin{aligned}
P_0(-2\pi irv_1(P_0\varphi+P_1\varphi)) = \lambda P_0\varphi,\\
P_1(-2\pi irv_1\varphi) + LP_1\varphi = \lambda P_1\varphi.
\end{aligned}\right.
\end{equation}
Note that $LP_0=P_0L=0$ implies $LP_1=P_1L$. 
The second equation yields 
\begin{align}\label{eq39}
(-\lambda - P_1(2\pi irv_1) + L)P_1\varphi = P_1(2\pi irv_1P_0\varphi).
\end{align}

3. For $\Re\lambda>-\nu_1$, denote
\begin{align*}
S:={-\lambda P_1 - P_1(2\pi irv_1)P_1 + L},
\end{align*}with domain $D(S):=\S\cap (\Ker L)^{\perp}$. Regard $(\Ker L)^\perp$ equipped with $L^2$ norm as the whole space in this step, then $D(S)$ is densely defined and hence closable.
We claim that $\big(\overline{S}\big)^{-1}$ exists on $(\Ker L)^\perp$ with the closure taken in $(\Ker L)^\perp$. 
Indeed, recall \eqref{Aadjoint} that 
\begin{align*}
\overline{2\pi irv_1-A}^{L^2} = (-2\pi irv_1-A)^{*{L^2}},
\end{align*} where the closure and adjoint are taken on $L^2$ with domain $\S$, denoted by $\overline{(\cdot)}^{L^2}$ and $(\cdot)^{*L^2}$ respectively while the usual notation $\overline{(\cdot)}$, $(\cdot)^*$ are taken on $(\Ker L)^\perp$. 
By Lemma \ref{semigroup}, it suffices to show that $S$ and its adjoint are dissipative on $(\Ker L)^\perp$. However, it's hard to prove $S$ has a dense range on $(\Ker L)^\perp$ and hence we can't use the former technique.
For $f\in \S$, $\Re\lambda>-\nu_1\ge-\nu_0$, by \eqref{eq17}, 
\begin{align}\label{eq40}
\Re(Sf,f)_{L^2}\le (-\nu_1-\Re\lambda)\|P_1f\|^2_{L^2} \le 0.
\end{align}
Hence the closure $\overline{S}$ of $S$ on $(\Ker L)^\perp$ and closure $\overline{S}^{L^2}$ on $L^2$ satisfies
\begin{equation}\label{eq43}\begin{split}
\Re(\overline{S}f,f)_{L^2}&\le (-\nu_1-\Re\lambda)\|P_1f\|^2_{L^2}\le 0,\ \text{for $f\in D(\overline{S})$,}\\
\Re(\overline{S}^{{L^2}}f,f)_{L^2}&\le 0,\  \text{for $f\in D(\overline{S}^{L^2})$}.\end{split}
\end{equation}
Notice that
\begin{align}\label{eq44}
S = -2\pi irv_1-A - \underbrace{\lambda P_1 + (2\pi irv_1)P_0 + P_0(2\pi irv_1)P_1 + K}_{\text{bounded part}},
\end{align}
we can find the domain of $S^*$, the adjoint taken on $(\Ker L)^\perp$:
\begin{align*}
D(S^*) &= \{g\in(\Ker L)^\perp:f\mapsto (Sf,g)_{L^2} \text{ is continuous on } \S\cap (\Ker L)^\perp\}\\
&\subset \{g\in L^2:f\mapsto (Sf,g)_{L^2} \text{ is continuous on } \S\cap L^2\}\\
&= D((-2\pi irv_1-A)^{*L^2}).
\end{align*}
Then by Definition \ref{Defadjoint}, for $f\in D(S^*)$, 
\begin{align*}
\Re(S^*f,f)_{L^2} &= \Re((-2\pi irv_1-A)^{*L^2}f,f)_{L^2}\\
&\qquad + \Re((-\lambda P_1 + (2\pi irv_1)P_0 + P_0(2\pi irv_1)P_1 + K)^{*L^2}f,f)_{L^2}\\
&= \Re(f,\overline{(2\pi irv_1-A)}^{L^2}f)_{L^2}\\
&\qquad + \Re(f,(-\bar{\lambda} P_1 + P_0(-2\pi irv_1) + P_1(-2\pi irv_1)P_0 + K)f)_{L^2}\\
&= \Re(f,\overline{(-\bar{\lambda} P_1+P_1(2\pi irv_1)P_1+L)}^{L^2}f)_{L^2}\\
&\le 0,
\end{align*}where the last inequality is similar to \eqref{eq43}.
Thus $S$ and its adjoint are dissipative on Banach space $(\Ker L)^\perp$. By Theorem \ref{semigroup}, $\overline{S}$ generates a contraction semigroup on $(\Ker L)^\perp$. 
Hence, by \ref{semigroup_Hille}, $(\overline{-(\lambda+\eta) P_1 - P_1(2\pi irv_1)P_1 + L})^{-1}$ exists on $(\Ker L)^\perp$, for $\Re\eta>0$, $\Re\lambda>-\nu_1$. Thus $\overline{S}$ is invertible on $(\Ker L)^\perp$.
Also by \eqref{eq40}, for $f\in(\Ker L)^\perp$,
\begin{align*}
\|\overline{S}^{-1}f\|_{L^2}\le \frac{1}{\nu_1+\Re\lambda}\|P_1f\|_{L^2}.
\end{align*}
One can easily apply the spectrum theory to get that $\overline{S}^{-1}$ is smooth with respect to $\lambda$. In the following, we will regard $\overline{S}$ and $(\overline{S})^{-1}$ as operators on $(\Ker L)^\perp$.

4. In this step, we claim that for $f,g\in\S$, $((\overline{S_r})^{-1}f,g)_{L^2}$ is smooth with respect to $r$, where we write $S_r$ to show the dependence on $r$. Indeed, for $g\in \S\cap (\Ker L)^\perp$, let $f:=(\overline{S})^{-1}g\in (\Ker L)^\perp$, then by \eqref{eq44},
\begin{align*}
\overline{\lambda +2\pi irv_1+A}f &= Kf+ \underbrace{\lambda P_0f + (2\pi irv_1)P_0f + P_0(2\pi irv_1)P_1f  - g}_{\in\S}.
\end{align*}
Applying the argument of step 4 in Theorem \ref{spectrum_struture}, we have $f\in\S$ and for $k,l\ge 0$, 
\begin{align}\label{eq44P0a}
\|\<v\>^k\<D_v\>^lf\|_{L^2} \lesssim \|f\|_{L^2}+\|\<v\>^k\<D_v\>^lg\|_{L^2}.
\end{align}
Thus $(\overline{S})^{-1}$ maps $\S$ into $\S$. Also $P_1=I-P_0$ maps $\S$ into $\S$. 
For $r_1,r_2\in\R$, $f,g\in\S$, 
\begin{align*}
\big((\overline{S}_{r_1})^{-1}f-(\overline{S}_{r_2})^{-1}f,g\big)_{L^2}
&=2\pi i\big((\overline{S}_{r_1})^{-1}(P_1(r_1-r_2)v_1P_1)(\overline{S}_{r_2})^{-1}f,g\big)_{L^2},\\
\big|\big((\overline{S}_{r_1})^{-1}f-(\overline{S}_{r_2})^{-1}f,g\big)_{L^2}\big|
&\le\frac{2\pi|r_1-r_2|}{\nu_1+\Re\lambda}\|(P_1(v_1)P_1)(\overline{S}_{r_2})^{-1}f\|_{L^2}\|g\|_{L^2}\\
&\to 0, \ \text{ as }{r_1\to r_2}.
\end{align*}
Hence $\partial_{r}((\overline{S_r})^{-1}f,g)_{L^2}=\big((\overline{S}_{r})^{-1}(P_12\pi iv_1)(\overline{S}_{r})^{-1}f,g\big)_{L^2}$. Noticing that 
\begin{align*}
&\quad\ (\overline{S}_{r_1})^{-1}(P_1v_1)(\overline{S}_{r_1})^{-1}-(\overline{S}_{r_2})^{-1}(P_1v_1)(\overline{S}_{r_2})^{-1} \\
&= (\overline{S}_{r_1})^{-1}(P_1v_1)(\overline{S}_{r_1})^{-1}-(\overline{S}_{r_2})^{-1}(P_1v_1)(\overline{S}_{r_1})^{-1}\\
&\quad\ + (\overline{S}_{r_2})^{-1}(P_1v_1)(\overline{S}_{r_1})^{-1}-(\overline{S}_{r_2})^{-1}(P_1v_1)(\overline{S}_{r_2})^{-1},
\end{align*}we can apply induction to find $((\overline{S_r})^{-1}f,g)_{L^2}$ is smooth with respect to $r$ and 
\begin{align*}
\partial^n_{r}((\overline{S_r})^{-1}f,g)_{L^2}=\big(\big((\overline{S}_{r})^{-1}(P_12\pi iv_1)\big)^n(\overline{S}_{r})^{-1}f,g\big)_{L^2}.
\end{align*}

5.
Now we can return to \eqref{eq39}. Notice that $2\pi i rv_1P_0\varphi\in\S$ and $P_1=I-P_0$ maps $\S$ into $\S$, we have 
\begin{align}\label{eq45}
P_1\varphi = \overline{S}^{-1}P_1(2\pi i rv_1P_0\varphi)\in\S.
\end{align}
Substitute this into \eqref{eq38}, 
\begin{align*}
P_0\big(-2\pi irv_1\big(P_0\varphi+\overline{S}^{-1}P_1(2\pi i rv_1P_0\varphi)\big)\big) - \lambda P_0\varphi = 0.
\end{align*}
Denote $T:= P_0(-2\pi iv_1\overline{S}^{-1}P_1(2\pi iv_1P_0))$, then $(Tf,g)_{L^2}$ is smooth respect to $r$, and 
\begin{align}\label{eq37}
rP_0(-2\pi iv_1P_0\varphi)+r^2TP_0\varphi- \lambda P_0\varphi = 0.
\end{align}
By Lemma \ref{eigenP0} in appendix, the operator $P_0(-2\pi iv_1P_0)$ has eigenvalues $\eta_{0,j}$ and eigenfunctions $\psi_{0,j}$. We expand
\begin{align}\label{eqP0}
P_0\varphi = \sum^{d+1}_{j=0}C_j\psi_{0,j}.
\end{align}
Then 
\begin{align*}
-2\pi ir\sum^{d+1}_{j=0}C_j\eta_{0,j}\psi_{0,j}+r^2\sum^{d+1}_{j=0}C_jT\psi_{0,j}-\sum^{d+1}_{j=0}C_j\lambda\psi_{0,j}=0.
\end{align*}Taking inner product with $\{\psi_{0,k}\}_{k=0}^{d+1}$, we obtain
\begin{align}
-2\pi ir
\begin{pmatrix}
C_0\eta_{0,0}\\
\vdots\\
C_{d+1}\eta_{0,d+1}
\end{pmatrix}+r^2
\begin{pmatrix}
\sum^{d+1}_{j=0}C_j(T\psi_{0,j},\varphi_{0,0})_{L^2}\\
\vdots\\
\sum^{d+1}_{j=0}C_j(T\psi_{0,j},\varphi_{0,d+1})_{L^2}
\end{pmatrix}-\begin{pmatrix}
C_0\lambda\\
\vdots\\
C_{d+1}\lambda
\end{pmatrix}=0,\notag\\
\left[
\lambda I_{d+2}+2\pi ir
\begin{pmatrix}
\eta_{0,0}& &\\
&\ddots&\\
&&\eta_{0,d+1}
\end{pmatrix}-r^2
\begin{pmatrix}
(T\psi_{0,j},\psi_{0,k})_{L^2}
\end{pmatrix}^{d+1}_{k,j=0}\right]
\begin{pmatrix}
C_0\\
\vdots\\
C_{d+1}
\end{pmatrix}=0,\label{eigen_eq}
\end{align}where $\big((T\psi_{0,j},\psi_{0,k})_{L^2}\big)^{d+1}_{k,j=0}$ is the matrix with $T_{jk}:=(T\psi_{0,j},\psi_{0,k})_{L^2}$ being the element of its $k^{th}$ row, $j^{th}$ column. 
Notice that $(T\psi_{0,j},\psi_{0,k})_{L^2}=4\pi(\overline{S}^{-1}P_1(v_1\psi_{0,j}),P_1(v_1\psi_{0,k}))_{L^2}$ and $\psi_{0,j}=v_j\mu^{1/2}$ $(j=2,\dots,d)$. The reflection 
\begin{align*}
v\mapsto (v_1,\dots,-v_j,\dots,v_d)\, (j=2,\dots,d)
\end{align*} and rotation 
\begin{align*}
R_{ij}:=(\dots,{v_i},\dots,{v_j},\dots)\mapsto (\dots,{v_j},\dots,{v_i},\dots)\,(i,j=2,\dots,d)
\end{align*} 
commutes with $\overline{S}^{-1}, v_1, P_0, P_1$
Thus by reflection, for $2\le j\le d$, $0\le k\le d+1$, $k\neq j$, we have $v_1\overline{S}^{-1}P_1v_1\psi_{0,j}$ is odd about $v_j$ while $\psi_{0,k}$ is even about $v_j$ and hence, $T_{jk}=0$. 
By rotation, $T_{22}=T_{33}=\cdots=T_{dd}=4\pi(\overline{S}^{-1}P_1(v_1\psi_{0,j}),P_1(v_1\psi_{0,k}))_{L^2}<0$ and are independent of $r$. 

6. In this step, we will solve the eigenvalues $\lambda$ and eigenvector $(C_0,\dots,C_{d+1})$ of \eqref{eigen_eq} by implicit function theorem.
The eigen-equation \eqref{eigen_eq} has a non-trivial solution $(C_0,\dots,C_{d+1})$ if and only if the corresponding matrix is of $0$ determinant. That is 
\begin{align}
0 &= \left|
\lambda I_{d+2}+2\pi ir
\begin{pmatrix}
\eta_{0,0}& &\\
&\ddots&\\
&&\eta_{0,d+1}
\end{pmatrix}-r^2
\begin{pmatrix}
(T\psi_{0,j},\psi_{0,k})_{L^2}
\end{pmatrix}^{d+1}_{k,j=0}\right|\notag\\
&= \left|\lambda I_3+2\pi ir\begin{pmatrix}
\eta_{0,0}&&\\&\eta_{0,1}&\\&&\eta_{0,d+1}
\end{pmatrix}-r^2(T_{jk})_{k,j=0,1,d+1}\right|\prod^d_{j=2}(\lambda-r^2T_{jj}).\label{eq46}
\end{align}
Note that $T_{jk}=0$ for $2\le j\le d$ and $k\neq j$.  
We can easily get $d-1$ solutions: 
\begin{align}\label{choicelambda1}
\lambda_j(r) = r^2T_{jj},
\end{align}with $T_{jj}<0$. Choose the corresponding eigenvectors of \eqref{eigen_eq} to be $e_2,\dots,e_d$, the standard unit vector in $\R^{d+2}$. Then by \eqref{eqP0}, we can choose 
\begin{align*}
P_0\varphi_j = \psi_{0,j} = \psi_j,\ \ (j=2,\dots,d).
\end{align*}

In order to find the other $3$ solutions, we write $\lambda=(2\pi ir)\eta$, $\zeta:=(C_0,C_1,C_{d+1})$ and 
\begin{align*}
f(r,\eta):=\eta I_3+\begin{pmatrix}
\eta_{0,0}&&\\&\eta_{0,1}&\\&&\eta_{0,d+1}
\end{pmatrix}-\frac{r}{2\pi i}(T_{jk})_{k,j=0,1,d+1}
\end{align*}
Suppose $r\in[-1,1]$, $|\eta|\le\frac{\nu_1}{4\pi}$, then $\Re\lambda>-\nu_1$. 
Since the above eigenvectors are unit vector $e_j$ $(j=2,\dots,d)$, we obtain that the eigen-equation \eqref{eq46} is equivalent to 
\begin{align*}
f(r,\eta)\zeta^T=0,
\end{align*}
since $r=0$ means $\lambda=0$ from \eqref{eigen_eq}.
To apply the implicit function Theorem \ref{implicit_function}, we consider $\eta=\eta^R+i\eta^I$ and 
\begin{align*}
g(r,\eta^R,\eta^I):=\big(\Re(\det f(r,\eta)),\Im(\det f(r,\eta))\big).
\end{align*}
Then $g:(-1,1)\times(-\frac{\nu_1}{4\pi},\frac{\nu_1}{4\pi})^2\to \R^2$ is smooth and 
\begin{align*}
g|_{r=0,\eta=-\eta_{0,j}}&=0, \\
\det(\partial_{\eta^R,\eta^I}g)|_{r=0,\eta=-\eta_{0,j}} &= \prod_{k=0,1,d+1,k\neq j}(\eta_{0,k}-\eta_{0,j})^2\neq 0,
\end{align*}for $j=0,1,d+1$.
Hence, implicit function theorem yields that there exists $\xi_0>0$ and $\eta_j(r)\in C^\infty([-\xi_0,\xi_0];\C)$ such that for $j=0,1,d+1$, 
\begin{equation*}
\begin{split}
\eta_j(0)=-\eta_{0,j},\ \det f(r,\eta_j(r))=0,\text{ for }r\in[-\xi_0,\xi_0].
\end{split}
\end{equation*} Thus,  
\begin{align}\label{choicelambda2}
\lambda_j(r)=2\pi ir\eta_j(r)\in C^\infty([-\xi_0,\xi_0];\C),\ j=0,1,d+1,
\end{align}is the other 3 eigenvalues.

In order to find the corresponding eigenvectors, it's equivalent to solve 
\begin{align*}
f(r,\eta_j(r))\begin{pmatrix}
C_0\\C_1\\C_{d+1}
\end{pmatrix}=0,
\end{align*}where the characteristic polynomial $\det f(r,\eta)$ has three distinct solution $\eta_j(r)$ when $r$ is sufficiently small. Hence each eigenvalue $\eta_j$ corresponds to distinct eigenvector $\zeta:=(C_0,C_1,C_{d+1})$. 
Denote $f^{(ii)}(r,\eta)$ to be the matrix $f(r,\eta)$ with the $i$-th row and $i$-th column removed, $f^{(i)}(r,\eta)$ to be the $i$th column of $f(r,\eta)$ with the $i$-th component removed, $\zeta^i$ to be $\zeta$ with the $i$-th component removed. Then $f^{(jj)}(0,\eta_j)$ is invertible by our choice of $\eta_j(r)$ from \eqref{eigen_matrix_eq} and hence $f^{(jj)}(r,\eta_j(r))$ is invertible whenever $r$ is sufficiently small. So the $j$th row in $f(r,\eta_j(r))$ can be eliminated by Gaussian elimination and we can choose the $j$-th component of $\zeta_j(r)$ to be $1$ and determine the rest of $\zeta_j(r)$ by  $\zeta^j(r):=-(f^{(jj)}(r,\eta_j(r)))^{-1}f^{(j)}(r,\eta_j(r))$.
Then $\zeta^j(0)=0$ and $\zeta_j(r)$ is smooth with respect to $r\in[-\xi_0,\xi_0]$. We then select 
\begin{align}\label{eq44P0}
P_0\varphi_j(r) = \zeta_j(r)\cdot(\psi_{0,0},\psi_{0,1},\psi_{0,d+1})=:\psi_{0,j}+rC_{0,j}(r)\varphi_{0,j},
\end{align}for $j=0,1,d+1$. Then $P_0\varphi_j(r)\in\S$ is a linear combination of $\psi_{0,0},\psi_{0,1},\psi_{0,d+1}$ and $C_{1,j}(r)\in C^\infty([-\xi_0,\xi_0])$. 

Write $\lambda_j(r)=-2\pi ir\eta_{0,j}+r^2\eta_{1,j}+O(r^3)$ $(j=0,1,d+1)$ and substitute eigenvalue and eigenfuntion $\lambda_j(r)$, $P_0\varphi_j(r)$ into \eqref{eq37}, we obtain  
\begin{align*}
P_0(-2\pi iv_1P_0\varphi_j)+rTP_0\varphi_j-(-2\pi i\eta_{0,j}+r\eta_{1,j}+O(r^2)) P_0\varphi_j = 0,
\end{align*}as $r\to 0$ and hence considering order $O(r)$, 
\begin{align*}
\eta_{1,j} = \frac{(TP_0\varphi_j,P_0\varphi_j)_{L^2}}{\|P_0\varphi_j\|^2_{L^2}}<0.
\end{align*}

7. 
At last we can select $P_1\varphi_j(r)$ by \eqref{eq45} and statement (2) follows from \eqref{eq44P0a}\eqref{eq45}\eqref{eq44P0}. Then by step one, the rotation $(R^{-1})^*\varphi_j$ gives the eigenfunctions of $\B$.

Recall the choice \eqref{choicelambda1}\eqref{choicelambda2} of $\lambda_j$.
If $|\xi|\le\xi_0$ and $\lambda$ is any eigenvalues of $\B$ with $\Re\lambda\ge -2\sigma_0$, then $\lambda$ is the root of \eqref{eq46} and must be $\lambda_j(r)$ for some $j=0,\dots,d+1$. 
Fix $0<\sigma_0<\nu_1/2$ and choose $\xi_0>0$ sufficiently small such that for $|r|\le\xi_0$, $j=0,\dots,d+1$, 
\begin{align*}
|\lambda_j(r)|<\frac{\sigma_0}{2}.
\end{align*}
Then for $|\xi|\le\xi_0$, 
\begin{align*}
\sigma(\B)&\subset \{\lambda:-\frac{\sigma_0}{2}<\Re\lambda\le 0,\, |\Im\lambda|<\frac{\sigma_0}{2}\},\\
\rho(\B)&\supset \{\lambda:-2\sigma_0\le\Re\lambda\le -\frac{\sigma_0}{2}\}
\cup\{\lambda:\Re\lambda\ge-2\sigma_0,|\lambda|\ge \frac{\sigma_0}{2}\}.
\end{align*}
\qe\end{proof}

The following lemma provides the decay of $\|(\lambda I-\A)^{-1}K\|_{\mathscr{L}(L^2)}$ when $|\xi|$ or $|\tau|$ is large, which will gives the invertibility of $\lambda I-\B$ when $|\xi|$ or $|\tau|$ is large. 
\begin{Lem}\label{Largeness}
Write $\lambda = \sigma+i\tau$. Let $|\tau|\ge 1$, $\Re\lambda\ge -\nu_1$. Then
\begin{align*}
\|(\lambda I-\A)^{-1}K\|_{\mathscr{L}(L^2)}\to 0,\text{   as }|\xi|+|\tau|\to \infty.
\end{align*}
\end{Lem}
\begin{proof}
Let $f\in\S$ and denote $\varphi = (\lambda I-\A)^{-1}Kf$, $\Phi(v) = e^{-\pi |v|^2}$, $\Phi_\varepsilon(v)=\varepsilon^{-d}\Phi(\varepsilon^{-1}v)$ with $\varepsilon\in(0,1)$. Then $\widehat{\Phi}=\Phi$. 
\begin{align*}
\|\varphi\|_{L^2} \le \|\Phi_\varepsilon*\varphi-\varphi\|_{L^2}+\|\Phi_\varepsilon*\varphi\|_{L^2}.
\end{align*}
For the first term, we have 
\begin{align}
	\|\Phi_\varepsilon*\varphi-\varphi\|_{L^2} &= \|(\widehat{\Phi}(\varepsilon \eta)-1)\widehat{\varphi}\|_{L^2}\notag\\
	&=\big(\int\big|{(e^{-\pi|\varepsilon \eta|^2}-1)}\widehat{\varphi}\big|^2\,dv\big)^{1/2}\notag\\
	&\le \big(\int_{|\eta|\ge\frac{1}{\sqrt{\varepsilon}}}4\varepsilon^{s}|\<\eta\>^{s}\widehat{\varphi}|^2\,dv
	+\int_{|\eta|\le\frac{1}{\sqrt{\varepsilon}}}|1-e^{-\pi\varepsilon}|^2|\widehat{\varphi}|^2\,dv\big)^{1/2}\notag\\
	&\lesssim  (\varepsilon^{s/2}+(1-e^{-\pi\varepsilon}))\|\<\eta\>^s\widehat{\varphi}\|_{L^2}\notag\\
&\lesssim (\varepsilon^{s/2}+(1-e^{-\pi\varepsilon}))\|f\|_{L^2},\label{Phiconvolutionvarphi}
\end{align}where the last inequality following from Lemma \ref{ASchwarz} and the boundedness of $K$.
For the second term, we would like to apply the calculation similar to Lemma 4.2 in \cite{Yang2016}. Let $\sigma_2>1$, depending on $\tau$ and $|\xi|$, to be chosen later. Since $Kf\in\S$, we have $\varphi\in\S$ and so
\begin{align*}
(\lambda +2\pi i v\cdot\xi +A)\varphi(v) &= Kf,
\end{align*}and hence 
\begin{align*}
\varphi &=\frac{\sigma_2\<v\>^{-1}\varphi-\sigma\varphi+Kf-A\varphi}{\sigma_2\<v\>^{-1}+i\tau+2\pi i v\cdot\xi}.
\end{align*}
Noticing $\sigma_2>1$, we have 
\begin{align}
|\Phi_\varepsilon*\varphi|\notag
&= \big|\int\Phi_{\varepsilon}(v-u)\varphi(u)\,du\big|\\
&\le \int\Phi_{\varepsilon}(v-u)\frac{|\varphi|+|\sigma\<u\>\varphi|+\<u\>|Kf(u)|}{(1+(\tau+2\pi u\cdot\xi)^2\<u\>^2/\sigma_2^2)^{1/2}}\,du\notag\\
&\qquad\qquad + \Big|\int\Phi_{\varepsilon}(v-u)\frac{A\varphi(u)}{\sigma_2\<u\>^{-1}+i\tau+2\pi iu\cdot\xi}\,du\Big|\label{convolutionI1I2}\\
&=:I_1+I_2.\notag
\end{align}
For $I_1$, by H\"{o}lder inequality, 
\begin{align}
|I_1|^2&=\Big|\int\Phi_{\varepsilon}(v-u)\frac{|\varphi|+|\sigma\<u\>\varphi|+\<u\>|Kf|}{(1+(\tau+2\pi u\cdot\xi)^2\<u\>^2/\sigma_2^2)^{1/2}}\,du\Big|^2\notag\\
&\le \int\Phi_{\varepsilon}(v-u)\big(|\varphi|+|\sigma\<u\>\varphi|+\<u\>|Kf|\big)^2\,du
\times
\int\frac{\Phi_{\varepsilon}(v-u)}{1+(\tau+2\pi u\cdot\xi)^2\<u\>^2/\sigma_2^2}\,du.\label{I11}
\end{align}
We will use the decomposition: $u = \tilde{u}\frac{\xi}{|\xi|} + u'$ with $\tilde{u}=\frac{u\cdot\xi}{|\xi|}$ and $u'=u-\tilde{u}\frac{\xi}{|\xi|}$. Then $\xi\perp u'$ and 
\begin{align}\label{estimateofJ}
\int\frac{\Phi_{\varepsilon}(v-u)}{1+(\tau+2\pi u\cdot\xi)^2\<u\>^2/\sigma_2^2}\,du
&\le \int_\R\int_{\R^{d-1}}\frac{\varepsilon^{-d}e^{-\pi |\frac{\tilde{u}-\tilde{v}}{\varepsilon}|^2-\pi |\frac{u'-v'}{\varepsilon}|^2}}{1+(\tau+2\pi \tilde{u}|\xi|)^2\<\tilde{u}\>^2/\sigma_2^2}\,du'd\tilde{u}\\
&\le \int_\R\frac{\varepsilon^{-1}}{1+(\tau+2\pi\notag \tilde{u}|\xi|)^2\<\tilde{u}\>^2/\sigma_2^2}\,d\tilde{u}\\
&=:J.\notag
\end{align}
It's easy to get 
\begin{align*}
J \le \varepsilon^{-1}\int_\R\frac{1}{1+(\tau+2\pi \tilde{u}|\xi|)^2/\sigma_2^2}\,d\tilde{u}\lesssim \frac{\sigma_2}{\varepsilon|\xi|}.
\end{align*}
On the other hand, notice $|\tilde{u}|\le\frac{|\tau|}{4\pi|\xi|}$ implies $|\tau+2\pi \tilde{u}|\xi||\ge |\tau|-|2\pi\tilde{u}|\xi||\ge \frac{|\tau|}{2}$.
\begin{align*}
J &=\frac{1}{\varepsilon} \int_{|\tilde{u}|\le\frac{|\tau|}{4\pi|\xi|}}\frac{1}{1+(\tau+2\pi \tilde{u}|\xi|)^2\<\tilde{u}\>^2/\sigma_2^2}\,d\tilde{u}
+\frac{1}{\varepsilon}\int_{|\tilde{u}|\ge\frac{|\tau|}{4\pi|\xi|}}\frac{1}{1+(\tau+2\pi \tilde{u}|\xi|)^2\<\tilde{u}\>^2/\sigma_2^2}\,d\tilde{u}\\
&\le \frac{1}{\varepsilon} \int_{|\tilde{u}|\le\frac{|\tau|}{4\pi|\xi|}}\frac{1}{1+(\frac{|\tau|\tilde{u}}{2\sigma_2})^2}\,d\tilde{u}
+\frac{1}{\varepsilon}\int_{|\tilde{u}|\ge\frac{|\tau|}{4\pi|\xi|}}\frac{1}{1+(\tau+2\pi \tilde{u}|\xi|)^2\frac{|\tau|^2}{(4\pi|\xi|\sigma_2)^2}}\,d\tilde{u}\\
&\lesssim \frac{\sigma_2}{\varepsilon|\tau|}.
\end{align*}
Thus combining the above two estimates, 
\begin{align}
J\lesssim \frac{\sigma_2}{\varepsilon(|\tau|^2+|\xi|^2)^{1/2}}.\label{estimateofJ_2}
\end{align}
Recall \eqref{I11} and applying \eqref{varphi_v_beta}, we have 
\begin{align}
\|I_1\|_{L^2_v}^2&\lesssim J\times \int\int\Phi_{\varepsilon}(v-u)\big(|\varphi|+|\sigma\<u\>\varphi|+\<u\>|Kf|\big)^2\,dudv\notag\\
&\lesssim J\times \big\||\varphi|+|\sigma\<u\>\varphi|+\<u\>|Kf|\big\|^2_{L^2_u}\notag\\
&\lesssim J\times\big\|\<u\>Kf\big\|^2_{L^2_u}\notag\\
&\lesssim \frac{\sigma_2}{\varepsilon(|\tau|^2+|\xi|^2)^{1/2}}\big\|f\big\|^2_{L^2_u}\label{estimateI1},
\end{align}since $K\in S(\<v\>^{-1})$, where the constant may depend of $\sigma$, the real part of $\lambda$.
For part $I_2$,
\begin{align*}
I_2 &= \Big|\Big(\frac{\<u\>^{-2}\Phi_{\varepsilon}(v-u)}{\sigma_2\<u\>^{-1}+i\tau+2\pi iu\cdot\xi},\<u\>^2A\varphi(u)\Big)_{L^2_u}\Big|\\
&= \Big|\Big(\<v-u\>^{d}\<D_u\>^{s}\Big(\frac{\<u\>^{-2}\Phi_{\varepsilon}(v-u)}{\sigma_2\<u\>^{-1}+i\tau+2\pi iu\cdot\xi}\Big),\<v-u\>^{-d}\<D_u\>^{-s}\<u\>^2A\varphi(u)\Big)_{L^2_u}\Big|\\
&\le \Big\|\<v-u\>^{d}\<D_u\>^{s}\frac{\<u\>^{-2}\Phi_{\varepsilon}(v-u)}{\sigma_2\<u\>^{-1}+i\tau+2\pi iu\cdot\xi}\Big)\Big\|_{L^2_u}
\Big\|\<v-u\>^{-d}\<D_u\>^{-s}\<u\>^2A\varphi(u)\Big\|_{L^2_u}
\end{align*}
We analyze the first part: Note $\sigma_2>1$, we have 
\begin{align*}
&\quad\ \Big\|\<v-u\>^{d}\<D_u\>^{s}\Big(\frac{\<u\>^{-2}\Phi_{\varepsilon}(v-u)}{\sigma_2\<u\>^{-1}+i\tau+2\pi iu\cdot\xi}\Big)\Big\|_{L^2_u}\\
&\lesssim 
\Big\|\<D_u\>^{s}\Big(\frac{\<v-u\>^{d}\<u\>^{-2}\Phi_{\varepsilon}(v-u)}{\sigma_2\<u\>^{-1}+i\tau+2\pi iu\cdot\xi}\Big)\Big\|_{L^2_u}\\
&\lesssim \Big\|\frac{\<v-u\>^{d}\<u\>^{-2}\Phi_{\varepsilon}(v-u)}{\sigma_2\<u\>^{-1}+i\tau+2\pi iu\cdot\xi}\Big\|_{H^1_u}\\
&\lesssim 
\Big\|\frac{\<v-u\>^{d}\<u\>^{-2}\Phi_{\varepsilon}(v-u)}{\sigma_2\<u\>^{-1}+i\tau+2\pi iu\cdot\xi}\Big\|_{L^2_u}
+\Big\|\frac{\<v-u\>^{d-1}\<u\>^{-2}\Phi_{\varepsilon}(v-u)}{\sigma_2\<u\>^{-1}+i\tau+2\pi iu\cdot\xi}\Big\|_{L^2_u}\\
&\qquad+\Big\|\frac{\<v-u\>^{d}\<u\>^{-3}\Phi_{\varepsilon}(v-u)}{\sigma_2\<u\>^{-1}+i\tau+2\pi iu\cdot\xi}\Big\|_{L^2_u}
+\Big\|\frac{\<v-u\>^{d}\<u\>^{-2}\varepsilon^{-1}(\nabla_u\Phi)_{\varepsilon}(v-u)}{\sigma_2\<u\>^{-1}+i\tau+2\pi iu\cdot\xi}\Big\|_{L^2_u}\\
&\qquad+\Big\|\frac{\<v-u\>^{d}\<u\>^{-2}\Phi_{\varepsilon}(v-u)\big(\sigma_2\<u\>^{-2}+|\xi|\big)}{(\sigma_2\<u\>^{-1}+i\tau+2\pi iu\cdot\xi)^2}\Big\|_{L^2_u}\\
&\lesssim\Big(\int\Big|\frac{\sigma_2\varepsilon^{-d-1}\exp(-\frac{\pi}{2}\big|\frac{v-u}{\varepsilon}\big|^2)}{\sigma_2+(i\tau+2\pi iu\cdot\xi)\<u\>}\Big|^2\,du\Big)^{1/2} + \Big(\int\Big|\frac{|\xi|\varepsilon^{-d}\exp(-\frac{\pi}{2}\big|\frac{v-u}{\varepsilon}\big|^2)}{(\sigma_2+(i\tau+2\pi iu\cdot\xi)\<u\>)^2}\Big|^2\,du\Big)^{1/2}\\
&=:J_1.
\end{align*}
For the first integral, we apply the estimate \eqref{estimateofJ} \eqref{estimateofJ_2}, while for the second integral, we make a rough estimate by canceling $(i\tau+2\pi iu\cdot\xi)\<u\>$. Thus
\begin{align*}
|J_1|^2&\lesssim\int\frac{\sigma^2_2\varepsilon^{-d-2}\Phi_{\varepsilon}(v-u)}{\sigma^2_2+(\tau+2\pi u\cdot\xi)^2\<u\>^2}\,du + \int\frac{|\xi|^2\varepsilon^{-d}\Phi_{\varepsilon}(v-u)}{\sigma^4_2}\,du\\
&\lesssim \varepsilon^{-d-2}J+\varepsilon^{-d}\sigma_2^{-4}|\xi|^2\\
&\lesssim \frac{\sigma_2}{\varepsilon^{d+3}(|\tau|^2+|\xi|^2)^{1/2}} + \frac{|\xi|^2}{\varepsilon^d\sigma_2^4}.
\end{align*}
Thus similar to $I_1$, we have 
\begin{align*}
\|I_2\|^2_{L^2_v}
&\lesssim |J_1|^2\times
\|\<v-u\>^{-d}\<D_u\>^{-s}\<u\>^2A\varphi(u)\|^2_{L^2_v(L^2_u)}\\
&\lesssim |J_1|^2\times
\|\underbrace{\<D_u\>^{-s}\<u\>^2A\<u\>^{-2-\gamma/2-s}}_{\in Op(\<u\>^{\gamma/2+s})\subset Op(a^{1/2})}\<u\>^{2+\gamma/2+s}\varphi\|^2_{L^2_u}\\
&\lesssim|J_1|^2\times\|(a^{1/2})^w\<u\>^{2+\gamma/2+s}\varphi\|^2_{L^2_u}\\
&\lesssim|J_1|^2\times\|\<u\>^{2+\gamma/2+s}Kf\|^2_{L^2_u}\\
&\lesssim \Big(\frac{\sigma_2}{\varepsilon^{d+3}(|\tau|^2+|\xi|^2)^{1/2}} + \frac{|\xi|^2}{\varepsilon^d\sigma_2^4}\Big)\|f\|_{L^2}^2,
\end{align*}where we apply \eqref{varphi_v_beta} and $A\in S(a)\subset S(\<\eta\>^{2s}\<v\>^{\gamma+2s})$, $K\in S(\<u\>^{-2-\gamma/2-s})$.
Together with \eqref{convolutionI1I2} \eqref{estimateI1}, we have 
\begin{align*}
\|\Phi_{\varepsilon}*f\|^2_{L^2_v}&\lesssim \Big(\frac{\sigma_2}{\varepsilon(|\tau|^2+|\xi|^2)^{1/2}}+\frac{\sigma_2}{\varepsilon^{d+3}(|\tau|^2+|\xi|^2)^{1/2}} + \frac{|\xi|^2}{\varepsilon^d\sigma_2^4}\Big)\|f\|_{L^2}^2\\
&\lesssim \Big(\frac{\sigma_2}{\varepsilon^{d+3}(|\tau|^2+|\xi|^2)^{1/2}} + \frac{|\xi|^2}{\varepsilon^d\sigma_2^4}\Big)\|f\|_{L^2}^2,
\end{align*}for $\varepsilon\in(0,1)$. 
Choose $\delta\in(0,1/(6d+32))$ small and 
\begin{align*}
\sigma_2:=(1+|\tau|^2+|\xi|^2)^{\frac{1}{2}-(d+4)\delta},\quad
\varepsilon:=(1+|\tau|^2+|\xi|^2)^{-\delta}
\end{align*}
Then, since $|\tau|\ge 1$, 
\begin{align*}
\|\Phi_{\varepsilon}*f\|^2_{L^2_v}\lesssim  \Big(\frac{1}{(1+|\tau|^2+|\xi|^2)^{\delta}} + \frac{1}{(1+|\tau|^2+|\xi|^2)^{1-(3d+16)\delta}}\Big)\|f\|_{L^2}^2.
\end{align*}Together with \eqref{Phiconvolutionvarphi}, 
\begin{align}\label{eq156}
\|\varphi\|_{L^2_v}\lesssim \frac{1}{(1+|\tau|^2+|\xi|^2)^{s\delta/2}}\|f\|_{L^2}.
\end{align}
\qe\end{proof}

The next theorem shows the uniformly boundedness of $(\lambda I-\B)^{-1}$ and $(I-(\lambda I-\A)^{-1}K)^{-1}$.
\begin{Thm}\label{Bcontinuous}
Write $\lambda = \sigma+i\tau$. Take $\xi_0,\sigma_0>0$ from Theorem \ref{eigenstruture} and with this $\xi_0$, we choose $\sigma_1$ from the statement (4) in Theorem \ref{spectrum_struture}. 
Let $\kappa$ equal to $\sigma_0$ when $|\xi|\le\xi_0$, and equal to $\sigma_1$ when $|\xi|\ge\xi_0$. Such choice assures that $(\kappa+i\tau) I-\B$ is invertible.
Then 
\begin{equation}\begin{split}\label{eq41}
\sup_{\xi\in\Rd,\Re\lambda=-\kappa}\|(I-(\lambda I-\A)^{-1}K)^{-1}\|_{\mathscr{L}(L^2)}<\infty,\\
\sup_{\xi\in\Rd,\Re\lambda=-\kappa}\|(\lambda I-\B)^{-1}\|_{\mathscr{L}(L^2)}<\infty.\end{split}
\end{equation}
\end{Thm}
\begin{proof}
  1. Fix $\lambda \in \rho(\A)\cap\rho(\B)$.
  Since $\B = \A +K$, we have
  \begin{align*}
    \lambda I-\B &= \lambda I -\A - K = (\lambda I-\A)(I-(\lambda I-\A)^{-1}K),\\
    I &= (\lambda I-\A)^{-1}(\lambda -\B)+(\lambda -\A)^{-1}K.
  \end{align*}
Thus $I-(\lambda I-\A)^{-1}K$ is invertible on $L^2$ and
\begin{align*}
  (\lambda I-\B)^{-1} &= (I-(\lambda I-\A)^{-1}K)^{-1}(\lambda I-\A)^{-1}\\
  &= (\lambda I-\A)^{-1}+(\lambda -\A)^{-1}K(\lambda I-\B)^{-1}.
\end{align*}
Hence
\begin{align*}
  (\lambda I-\B)^{-1} &= (\lambda -\A)^{-1}+(\lambda -\A)^{-1}K(I-(\lambda I-\A)^{-1}K)^{-1}(\lambda -\A)^{-1}.
\end{align*}

Next we prove the continuity of $(I-(\lambda I-\A)^{-1}K)^{-1}$.
For any $\xi_1,\xi_2\in\R^d$, $\Re\lambda_1,\Re\lambda_2\ge2\sigma_1$ with $\lambda_1\in\rho(\widehat{A}(\xi_1))\cap\rho(\widehat{B}(\xi_1))$, $\lambda_2\in\rho(\widehat{A}(\xi_2))\cap\rho(\widehat{A}(\xi_2))$, we have
\begin{align*}
  &\quad\ (I-(\lambda_1 I-\widehat{A}(\xi_1))^{-1}K)^{-1} - (I-(\lambda_2 I-\widehat{A}(\xi_2))^{-1}K)^{-1}\\
  &= (I-(\lambda_1 I-\widehat{A}(\xi_1))^{-1}K)^{-1}\big((\lambda_1I-\widehat{A}(\xi_1))^{-1}-(\lambda_2I-\widehat{A}(\xi_2))^{-1}\big)K\\
&\qquad\qquad\qquad\qquad\qquad\qquad\qquad\qquad\qquad(I-(\lambda_2 I-\widehat{A}(\xi_2))^{-1}K)^{-1}
\end{align*}
We first deal with the middle term. Let $f\in\S$, then $Kf\in\S$, $(\lambda_2I-\widehat{A}(\xi_2))^{-1}Kf\in\S$ from Theorem \ref{L2v_regularity}. Thus
\begin{align*}
  &\quad\ \big((\lambda_1I-\widehat{A}(\xi_1))^{-1}-(\lambda_2I-\widehat{A}(\xi_2))^{-1}\big)Kf\\
  &= (\lambda_1I-\widehat{A}(\xi_1))^{-1}\big(\lambda_2I-\widehat{A}(\xi_2)-\lambda_1I+\widehat{A}(\xi_1)\big)(\lambda_2I-\widehat{A}(\xi_2))^{-1}Kf\\
  &= (\lambda_1I-\widehat{A}(\xi_1))^{-1}\big(\lambda_2I-\lambda_1I+2\pi i(\xi_2-\xi_1)\cdot v)\big)(\lambda_2I-\widehat{A}(\xi_2))^{-1}Kf.
\end{align*}
By \eqref{varphi_v_beta},
\begin{align*}
 &\quad\ \|\big((\lambda_1I-\widehat{A}(\xi_1))^{-1}-(\lambda_2I-\widehat{A}(\xi_2))^{-1}\big)Kf\|_{L^2_v}\\
 &\le C\|\big(\lambda_2-\lambda_1+2\pi i(\xi_2-\xi_1)\cdot v)\big)(\lambda_2I-\widehat{A}(\xi_2))^{-1}Kf\|_{L^2_v}\\
 &\le C\big(|\lambda_2-\lambda_1|+|\xi_2-\xi_1|\big)\|\<v\>(\lambda_2I-\widehat{A}(\xi_2))^{-1}Kf\|_{L^2_v}\\
 &\le C\big(|\lambda_2-\lambda_1|+|\xi_2-\xi_1|\big)\|\<v\>Kf\|_{L^2_v}\\
 &\le C\big(|\lambda_2-\lambda_1|+|\xi_2-\xi_1|\big)\|f\|_{L^2_v}.
\end{align*}
Thus $(\lambda I-\A)^{-1}$ is continuous with respect to $(\lambda,y)$. Since $\big((\lambda_1I-\widehat{A}(\xi_1))^{-1}-(\lambda_2I-\widehat{A}(\xi_2))^{-1}\big)K$ is bounded on $L^2_v$, the above estimate is also valid for $f\in L^2$ by density argument.
Fix $\lambda_2,\xi_2$ and set $(\lambda_1,\xi_1)$ sufficiently close to $(\lambda_2,\xi_2)$ such that $\|\big((\lambda_1I-\widehat{A}(\xi_1))^{-1}-(\lambda_2I-\widehat{A}(\xi_2))^{-1}\big)K\|_{\mathscr{L}(L^2)}\le \frac{1}{2\|(\lambda_2 I -\widehat{A}(\xi_2))^{-1}\|_{\mathscr{L}(L^2)}}$.
Therefore, applying Lemma \ref{operator_inverse} to $(I-(\lambda_1 I-\widehat{A}(\xi_1))^{-1}K)^{-1}$, we have 
\begin{align*}
  &\quad\ \|(I-(\lambda_1 I-\widehat{A}(\xi_1))^{-1}K)^{-1} - (I-(\lambda_2 I-\widehat{A}(\xi_2))^{-1}K)^{-1}\|_{\mathscr{L}(L^2)}\\
  &\le \frac{1}{2}\|\big((\lambda_1I-\widehat{A}(\xi_1))^{-1}-(\lambda_2I-\widehat{A}(\xi_2))^{-1}\big)K\|_{\mathscr{L}(L^2)}\|(I-(\lambda_2 I-\widehat{A}(\xi_2))^{-1}K)^{-1}\|_{\mathscr{L}(L^2)}\\
  &\le C\big(|\lambda_1-\lambda_2|+|\xi_1-\xi_2|\big)\|(I-(\lambda_2 I-\widehat{A}(\xi_2))^{-1}K)^{-1}\|_{\mathscr{L}(L^2)}.
\end{align*}
In conclusion, $(\lambda I-\A)^{-1}$, $(I-(\lambda I-\widehat{A}(\xi))^{-1}K)^{-1}$ and hence $(\lambda I-\B)^{-1}$, as operators on $L^2$, are continuous with respect to $\lambda\in \sigma(\B)\cap\sigma(\A)$ and $\xi\in\R^d$.

2. By \eqref{eq25} \eqref{estiA}, there exists $\tau_1>0$ such that for $\xi\in\R^d$, $\|(\lambda I-\A)^{-1}f\|_{\mathscr{L}(L^2)}$ and $\|(I-(\lambda I-\widehat{A}(\xi))^{-1}K)^{-1}\|_{\mathscr{L}(L^2)}$ are uniformly bounded on $\{\lambda\in\C:|\Im\lambda|\ge \tau_1\}$. For $|\tau|$ large, we can use this uniformly bounded estimate, while for $|\tau|$ small, we can apply the continuity in step one.

Choose $\xi_0,\sigma_0$ in Theorem \ref{eigenstruture} and apply this $\xi_0$ to Theorem \ref{spectrum_struture} (4) to find $\sigma_1\in(0,\nu_1)$ such that
\begin{align*}
\sigma(\B)\cap \{\lambda\in\C:-2\sigma_0\le\Re\lambda\le 0\}&=\{\lambda_j\}_{j=0}^{d+1}, \text{ for } |\xi|\le\xi_0. \\
\sigma(\B)\cap \{\lambda\in\C:-2\sigma_1\le\Re\lambda\le 0\}&=\emptyset, \qquad\text{ for } |\xi|\ge\xi_0. 
\end{align*}Then Theorem \ref{eigenstruture} gives that $\rho(\B)\supset\{\lambda:\Re\lambda=\kappa\}$ for $\xi\in\Rd$ where $(\lambda I-\B)^{-1}$ is continuous and we derive the bound \eqref{eq41}.
\qe\end{proof}

\section{Regularizing estimate for linearized equation}\label{sec4}

Assume $\gamma+2s\ge0$ in this section, then $\|\cdot\|_{L^2}\lesssim\|\cdot\|_{\H}$. We are concerned with the proof to our main Theorem \ref{Thm1}: the regularizing estimate of semigroup $e^{tB}$. 
\begin{Thm}\label{Bbound}
Let $f\in\S$, $\xi\in\Rd$, choose $\kappa$ as \eqref{sigma_xi}. Then for $\tau\in\R$, we have $\lambda:=\kappa+i\tau\in\rho(\B)\cap \rho(\A)$ as in Theorem \ref{Bcontinuous}. There exists $C>0$ independent of $y$ such that for $f\in\S$, 
\begin{align*}
\|(\lambda I-\B)^{-1}f\|_{H(a^{1/2})}\le C\|f\|_{H(a^{-1/2})}.
\end{align*}
\end{Thm}
\begin{proof}
Write $\varphi = (\lambda I-\B)^{-1}f$. Then 
\begin{align*}
(\lambda I-\A)\varphi = f + K\varphi.
\end{align*}
A similar argument to step 4 in Theorem \ref{spectrum_struture} gives $\varphi\in\S$ and hence taking inner product with $\varphi$, we have 
\begin{align*}
\|\varphi\|^2_{\H}
&\lesssim \Re((\lambda I-\A)\varphi,\varphi)_{L^2}\\
&=  \Re(f+K\varphi,\varphi)_{L^2}\\
&\lesssim \|f\|_{H(a^{-1/2})}\|\varphi\|_{\H}+\|\varphi\|^2_{L^2}.
\end{align*}
Recall that  $(\lambda I -\B)^{-1} = (I-(\lambda I-\A)^{-1}K)^{-1}(\lambda I-\A)^{-1}$
and notice the two inverse operator on the right are uniformly bounded in $L^2$ from Theorem \ref{Bcontinuous}, we have
\begin{align*}
\|\varphi\|_{\H}
&\lesssim \|f\|_{H(a^{-1/2})}+\|(\lambda I -\B)^{-1}f\|_{L^2}\\
&\lesssim \|f\|_{H(a^{-1/2})}+\|(\lambda I -\A)^{-1}f\|_{L^2}\\
&\lesssim \|f\|_{H(a^{-1/2})},
\end{align*}where the last inequality follows from \eqref{estiA}. Also these constants are independent of $y$. 
\qe\end{proof}

With the smoothing effect of $(\lambda I-\B)^{-1}$, we can prove our main Theorem \ref{Thm1}. 
\begin{proof}[Proof of Theorem \ref{Thm1}]

1. Write $\lambda = \sigma+i\tau$.
Take $\xi_0,\sigma_0>0$ from Theorem \ref{eigenstruture} and with this $\xi_0$, we choose $\sigma_1$ from the statement (4) in Theorem \ref{spectrum_struture}.
Define
\begin{align}\label{sigma_xi}\kappa = \left\{
\begin{aligned}
\sigma_0, \text{ if } |\xi|\le\xi_0,\\
\sigma_1, \text{ if } |\xi|\ge\xi_0.
\end{aligned}\right.
\end{align}
Then for $|\xi|\le\xi_0$, 
\begin{align*}
\rho(\B)&\supset\{\lambda:-2\kappa\le\Re\lambda\le -\frac{\kappa}{2}\}
\cup\{\lambda:\Re\lambda\ge-2\kappa,|\Im\lambda|\ge \frac{\kappa}{2}\},
\end{align*}
and for $|\xi|\ge\xi_0$, 
\begin{align*}
\rho(\B)\supset\{\lambda\in\C:-2\kappa<\Re\lambda\}.
\end{align*}

2. For $f\in \S$, by Corollary III.5.15 in \cite{Engel1999}, there exists $\sigma_2>0$ such that
  \begin{align*}
    e^{t\B}f &= \frac{1}{2\pi i}\lim_{n\to\infty}\int^{\sigma_2+in}_{\sigma_2-in}e^{\lambda t}(\lambda I-\B)^{-1}f\,d\lambda,
  \end{align*}where the limit is taken in $L^2$.
Since $e^{\lambda t}(\lambda I-\B)^{-1}f$ is analytic with respect to $\lambda\in\rho(\B)$, the Cauchy theorem on holomorphic function and Theorem XV.2.2 in \cite{Gohberg1990} yield that 
\begin{align}\label{eq65}\notag
&\int^{\sigma_2+in}_{\sigma_2-in}e^{\lambda t}(\lambda I-\B)^{-1}f\,d\lambda
\\&\quad= \Big(\int^{\sigma_2+in}_{-\kappa+in}+\int^{-\kappa-in}_{\sigma_2-in} +\int^{-\kappa+in}_{-\kappa-in}\Big)e^{\lambda t}(\lambda I-\B)^{-1}f\,d\lambda
+\1_{|\xi|\le\xi_0}e^{\lambda_j(|\xi|)}P_{j}f.
\end{align}where $n>\kappa$ and  $
P_{j}=\int_{\Gamma_j}(\lambda I-\B)^{-1}\,d\lambda,$ with $\Gamma_j$ being the smooth boundary of neighborhood that contains only eigenvalue $\lambda_j(|\xi|)$ and separate from other eigenvalues. 
The rectangular curve in Cauchy formula \eqref{eq65} is shown in Figure \ref{fig3}. 
  \begin{figure}[htbp]
	\centering
	\begin{minipage}[t]{0.25\linewidth}
		\centering
		\begin{figure}[H]
			\centering
			\begin{tikzpicture}[smooth,>=Stealth,scale=0.46]
				\draw[thick,->] (-4,0)--(6,0);
				\draw[thick,->] (0,-4)--(0,4);
				\draw[blue,->] (-1,-3)--(-1,1);
				\draw[blue,->] (-1,3)--(2,3);
				\draw[blue,->] (3,-3)--(1,-3);
				\draw[blue] (-1,-3) node[black,below left]{\tiny $-\kappa-in$};
				\draw[blue] (-1,3) node[black,above left]{\tiny $-\kappa+in$};
				\draw[blue] (3,-3) node[black,below right]{\tiny $\sigma_2-in$};
				\draw[blue] (3,3) node[black,above right]{\tiny $\sigma_2+in$};
				\draw [color=blue] (-1,-3) rectangle (3,3);
			\end{tikzpicture}
		\end{figure}
	\end{minipage}%
	\centering
	\caption{Resolvent set for $\B$ in the complex plane with the parameter $\xi\in\Rd$}
	\label{fig3}
\end{figure}

3. For the first term, noticing that $(\lambda I -\B)^{-1} = (I-(\lambda I-\A)^{-1}K)^{-1}(\lambda I-\A)^{-1}$, we apply \eqref{eq41}\eqref{eq156} to get 
\begin{align*}
  &\quad\ \lim_{n\to\infty}\Big\|\int^{\sigma_2+in}_{-\kappa+in}e^{\lambda t}(\lambda I-\B)^{-1}f\,d\lambda\Big\|_{L^2}\\&\le e^{\sigma_2t}(\sigma_2+\kappa)\lim_{n\to\infty}\sup_{-\sigma_1\le\Re\lambda\le\sigma_2}\|(\lambda I-\A)^{-1}f\|_{L^2}\,d\lambda\\
 &\le Ce^{\sigma_2t}\limsup_{n\to\infty}\frac{1}{(1+|n|^2+|\xi|^2)^{s\delta/2}}\|f\|_{L^2}\\
  &=0,
\end{align*}for some $\delta>0$, since $f\in\S$. Similarly, the second term satisfies
\begin{align*}
  \lim_{n\to\infty}\int^{-\kappa-in}_{\sigma_2-in}e^{\lambda t}(\lambda I-\B)^{-1}f\,d\lambda=0,\ \ \text{ in }L^2.
\end{align*}

4. Now we investigate the third integral. 
Notice that $\partial^k_\tau((\lambda I-\B)^{-1}f) = k!(-i)^k(\lambda I-\B)^{-k-1}f$, for $\lambda\in\rho(\B)$.
Using integration by parts, we have, for $k\ge 2$, 
\begin{align*}
&\quad\ \int^{-\kappa+in}_{-\kappa-in}e^{\lambda t}(\lambda I-\B)^{-1}f\,d\lambda\\
&=\sum^k_{j=1}\frac{e^{(-\kappa+i\tau)t}(j-1)!}{it^j}(\lambda I-\B)^{-j}f\Big|^{\tau=n}_{\tau=-n}+\frac{k!}{t^k}\int^{-\kappa+in}_{-\kappa-in}e^{\lambda t}(\lambda I-\B)^{-k-1}\,d\lambda. 
\end{align*}
Here we can take larger $k$ to deduce higher regularity from Theorem \ref{Bbound} and higher-order time decay from integration by parts. 
By Theorem \ref{Bcontinuous} and \eqref{eq156}, we have 
\begin{align*}
&(\lambda I-\B)^{-j}f = \left((I-(\lambda I-\A)^{-1}K)^{-1}(\lambda I-\A)^{-1}\right)^{j}f,\\
\|&(\lambda I-\B)^{-j}f\|_{L^2} \lesssim \|(\lambda I-\A)^{-1}f\|_{L^2}\to 0,\ \text{ as } \tau\to\infty. 
\end{align*}
Thus, 
\begin{align*}
\lim_{n\to\infty}\int^{-\kappa+in}_{-\kappa-in}e^{\lambda t}(\lambda I-\B)^{-1}f\,d\lambda
&= \lim_{n\to\infty}\frac{k!}{t^k}\int^{-\kappa+in}_{-\kappa-in}e^{\lambda t}(\lambda I-\B)^{-k-1}\,d\lambda,
\end{align*}with the limit taken in $L^2$. 
For $k\ge 2$, $g\in L^2$ and $l,m\ge 0$, we have 
\begin{align}
  &\quad\ \Big|\Big(\int^{-\kappa+in}_{-\kappa-in}e^{\lambda t}(\lambda I-\B)^{-k-1}f\,d\lambda,\,g\Big)_{L^2_v}\Big|\notag\\
  &\le e^{-\kappa t}\int^{-\kappa+in}_{-\kappa-in}\big|\big((\lambda I-\B)^{-k}f,\,(\lambda I-\B^*)^{-1}g\big)_{L^2_v}\big|\,d\lambda\notag\\
  &\le e^{-\kappa t}\int^{-\kappa+in}_{-\kappa-in}\|(\lambda I-\B)^{-k}f\|_{H(a^{1/2})}\,\|(\lambda I-\B^*)^{-1}g\|_{H(a^{-1/2})}\,d\lambda\notag\\
  &\le e^{-\kappa t}
  \Big(\int^{-\kappa+in}_{-\kappa-in}\|(\lambda I-\B)^{-k}f\|^2_{H(a^{1/2})}\,d\lambda\Big)^{1/2}\notag\\
&\qquad\qquad\times
  \Big(\int^{-\kappa+in}_{-\kappa-in}\|(\lambda I-\B^*)^{-1}g\|^2_{H(a^{-1/2})}\,d\lambda\Big)^{1/2}.\label{eq187}
\end{align}
Denote $\varphi=((-\kappa+i\tau) I-\B)^{-1}f$. By Theorem \ref{Bbound}, for $\beta\in\R$, $k\ge 2$, we have 
\begin{align}\notag
  &\quad\ \int^{-\kappa+in}_{-\kappa-in}\|(\lambda I-\B)^{-k}f\|^2_{H(a^{1/2})}\,d\lambda\\
  &\lesssim \int^{-\kappa+in}_{-\kappa-in}\|(\lambda I-\B)^{-1}f\|^2_{H(a^{-1/2})}\,d\lambda\notag\\
  &\lesssim \int^{in}_{-in}\lim_{\varepsilon\to 0} \big(\|\varphi-\Phi_\varepsilon*\varphi\|_{H(a^{-1/2})}+\|\Phi_\varepsilon*\varphi\|_{H(a^{-1/2})}\big)^2\,d\tau.\label{eq83}
\end{align}where $\widehat\Phi(\eta):=\<\eta\>^{-d-1}$ and $\Phi_{\varepsilon}:=\varepsilon^{-d}\Phi(\frac{v}{\varepsilon})$. That is $\Phi*\varphi=\F^{-1}(\widehat\Phi\F\varphi)=\widehat\Phi(\varepsilon\nabla_v)\varphi$. 
On one hand, 
\begin{align*}
  \|\varphi-\Phi_\varepsilon*\varphi\|_{H(a^{-1/2})}
  &\lesssim \|\varphi-\Phi_\varepsilon*\varphi\|_{L^2}
  \to 0,\quad\text{as $\varepsilon\to 0$.}
\end{align*}On the other hand, 
\begin{align*}
  \|\Phi_\varepsilon*\varphi\|_{H(a^{-1/2})}
  &\lesssim \|(a^{-1/2})^w\widehat{\Phi}(\varepsilon\nabla_v)\varphi\|_{L^2}.
\end{align*}
Write
\begin{align*}
  \varphi = \frac{-A\varphi+K\varphi+f}{-\kappa+i\tau + 2\pi i\xi\cdot v}.
\end{align*}Then 
\begin{align*}
  \|(a^{-1/2})^w\widehat{\Phi}(\varepsilon\nabla_v)\varphi\|_{L^2}
  &=\|(a^{-1/2})^w\widehat{\Phi}(\varepsilon\nabla_v)\big(\frac{-A\varphi+K\varphi+f}{-\kappa+i\tau + 2\pi i\xi\cdot v}\big)\|_{L^2}.
\end{align*}
Since $\kappa>0$, we have $(-\kappa+i\tau + 2\pi i\xi\cdot v)^{-1} \in S\big((\kappa^2+(\tau + 2\pi \xi\cdot v)^2)^{-1/2}\big)\subset S(1)$,
  $-\kappa+i\tau + 2\pi i\xi\cdot v \in S\big((\kappa^2+(\tau + 2\pi \xi\cdot v)^2)^{1/2}\big)$ uniformly in $y$. Denoting $\Psi(v)=-\kappa+i\tau + 2\pi i\xi\cdot v$, we have
\begin{align*}
  &\quad\ \|(a^{-1/2})^w\widehat{\Phi}(\varepsilon\nabla_v)\Big(\frac{-A\varphi+K\varphi+f}{-\kappa+i\tau + 2\pi i\xi\cdot v}\Big)\|_{L^2}\\
  &=\|\Psi^{-1}\underbrace{\Psi(a^{-1/2})^w\widehat{\Phi}(\varepsilon\nabla_v)\Psi^{-1}((a^{-1/2})^w)^{-1}}_{\in Op(\widehat{\Phi}(\varepsilon\eta))}(a^{-1/2})^w(-A\varphi+K\varphi+f)\|_{L^2}\\
  &\lesssim \left\|\frac{\widehat{\Phi}(\varepsilon\nabla_v)(a^{-1/2})^w(-A\varphi+K\varphi+f)}{-\kappa+i\tau + 2\pi i\xi\cdot v}\right\|_{L^2_v}.
\end{align*}
Therefore \eqref{eq83} becomes 
\begin{align*}
  \int^{-\kappa+in}_{-\kappa-in}&\|(\lambda I-\B)^{-k}f\|^2_{H(a^{1/2})}\,d\lambda\\
  &\lesssim \int^{in}_{-in}\lim_{\varepsilon\to 0} \|\Phi_\varepsilon*\varphi\|^2_{H(a^{-1/2})}\,d\tau\\
&\lesssim \int^{in}_{-in}\lim_{\varepsilon\to 0}\Big\|\frac{\widehat{\Phi}(\varepsilon\nabla_v)(a^{-1/2})^w(-A\varphi+K\varphi+f)}{-\kappa+i\tau + 2\pi i\xi\cdot v}\Big\|^2_{L^2_v}\,d\tau\\
&\lesssim \liminf_{\varepsilon\to 0}\int_{\Rd}\int_{\R}\frac{\big|\widehat{\Phi}(\varepsilon\nabla_v)(a^{-1/2})^w(-A\varphi+K\varphi+f)(v)\big|^2}{\kappa^2+(\tau + 2\pi\xi\cdot v)^2}\,d\tau dv\\
&\lesssim \liminf_{\varepsilon\to 0}\int_{\R}\frac{\big\|\Phi_\varepsilon*((a^{-1/2})^w(-A\varphi_1+K\varphi_1+f))(v)\big\|_{L^2}^2}{\kappa^2+\tau^2}\,d\tau\\
&\lesssim \int_{\R}\frac{\big\|(a^{-1/2})^w(-A\varphi_1+K\varphi_1+f)\big\|_{L^2}^2}{\kappa^2+\tau^2}\,d\tau\\
&\lesssim \|f\|^2_{H(a^{-1/2})}.
\end{align*}by Fatou's lemma, dominated convergence theorem and Theorem \ref{Bbound}, where $A\in Op(a)$ and the constant is independent of $y$ and $\varphi_1 := (\sigma+i(\tau-2\pi v \cdot{\xi})-\B)^{-1}f$.

Similarly, $\B^*=\widehat{B}(-y)$ in the second integral of \eqref{eq187} satisfies the same estimate. 
Therefore,
\begin{align*}
  \Big|\Big(\int^{-\kappa+in}_{-\kappa-in}e^{\lambda t}(\lambda I-\B)^{-k-1}f\,d\lambda,\,g\Big)_{L^2}\Big|
&\lesssim e^{-\kappa t}\|f\|_{H(a^{-1/2})}\|g\|_{H(a^{-1/2})},\\
  \Big\|\int^{-\kappa+in}_{-\kappa-in}e^{\lambda t}(\lambda I-\B)^{-1}f\,d\lambda\Big\|_{\H}&\lesssim \frac{e^{-\kappa t}k!}{t^k}\|f\|_{H(a^{-1/2})},
\end{align*} for $f\in\S$ and hence for $f\in H(a^{-1/2})$ by density.

5. Since when $|\xi|$ is fixed, $\lambda_j(|\xi|)$ is isolated eigenvalue of $\B$, by Theorem XV.2.2 in \cite{Gohberg1990}, we have for $|\xi|\le\xi_0$, 
\begin{align*}
  \text{Res}\{e^{\lambda t}(\lambda-\B)^{-1}f:\lambda=\lambda_j(\xi)\}= e^{\lambda_j(\xi)t}P_{j}f,
\end{align*}where $P_{j}$ is the projection from $L^2$ into the eigenspace corresponding $\lambda_j(\xi)$. Recall the behavior of $\lambda_j(|\xi|)$ and $P_{j}$ in Theorem \ref{eigenstruture}, we can choose $\xi_0$ so small that $\Re\lambda_j(\xi)\le -C_0|\xi|^2$, for any $|\xi|\le\xi_0$ and some $C_0>0$. Thus, substituting the estimate in step three and four into \eqref{eq65}, we have 
\begin{align*}
\|e^{t\B}f\|_{\H} &\lesssim \frac{e^{-\kappa t}k!}{t^k}\|f\|_{H(a^{-1/2})} + \1_{|\xi|\le\xi_0}\sum^{d+1}_{j=0} |e^{\lambda_j(\xi)t}|\|P_{j}f\|_{H(a^{1/2})}\\
&\lesssim  \frac{e^{-\kappa t}C_k}{t^k}\|f\|_{H(a^{-1/2})} + \1_{|\xi|\le\xi_0}e^{-C_0|\xi|^2t}\|f\|_{H(a^{-1/2})}.
\end{align*}
Since $e^{tB}$, $e^{t\B}$ generate strongly continuous semigroup on $L^2_{x,v}$, $L^2_v$ respectively, $B = \F^{-1}_x\B\F_x$ on $L^2$, we have, by \eqref{Fouriersemigroup},
\begin{align*}
  e^{tB} = \F^{-1}_xe^{t\B}\F_x.
\end{align*}
Then by Hausdorff-Young's inequality and H\"older's inequality, for $f\in\S(\R^{2d})$, $p\in[1,2)$, $q\in(1,\infty)$ satisfying $\frac{1}{p}+\frac{1}{2q}=1$, we have that
\begin{align*}
  &\quad\ \|e^{tB}f\|^2_{H(a^{1/2})H^m_x}\\
  &= \int\<\xi\>^{2m}\int|(a^{1/2})^we^{t\B}\F_xf|^2\,dvd\xi\\
  &\lesssim\int\<\xi\>^{2m}\Big(\frac{e^{-2\kappa t}C_k}{t^{2k}}+\1_{|\xi|\le\xi_0}e^{-2C_0|\xi|^2t}\Big)\int|(a^{-1/2})^w\F_xf|^2\,dvd\xi\\
  &\lesssim\frac{e^{-2\kappa t}C_k}{t^{2k}}\|f\|^2_{H(a^{-1/2})H^m_x}\\&\qquad +\int\Big(\int|\1_{|\xi|\le\xi_0}e^{-C_0|\xi|^2t}|^{2\cdot\frac{q}{q-1}}\,d\xi\Big)^{\frac{q-1}{q}}
  \Big(\int|\F_x(a^{-1/2})^wf|^{2q}\,d\xi\Big)^\frac{1}{q}\,dv\\
  &\lesssim \frac{e^{-2\kappa t}C_k}{t^{2k}}\|f\|^2_{H(a^{-1/2})H^m_x}
  +\frac{C_p}{(1+t)^{d/2(2/p-1)}}\|(a^{-1/2})^wf\|^2_{L^2_v(L^p_x)},
\end{align*}where constant $C_p$ is uniformly bounded on $p\in[1,2)$.
\qe\end{proof}

\section{Global existence for hard potential}\label{sec5}
In this section, we will discuss the global existence to Boltzmann equation \eqref{eq00} for hard potential $\gamma+2s\ge0$. Here we introduce a norm $X$ similar to \cite{Gualdani2017} and deduce an energy estimate on this space. Also we will apply the estimate on $\Gamma$ from \cite{Gressman2011}.

\subsection{Estimate on nonlinear term $\Gamma$} By Theorem 2.1 in \cite{Gressman2011} and \eqref{equivalent_norm}, for $f,g\in\S$, $n\ge 0$, 
\begin{align}\label{eq87}
|(\Gamma(f,g),h)_{L^2_v}|&\le C \|f\|_{L^2_v}\|(a^{1/2})^wg\|_{L^2_v}\|(a^{1/2})^wh\|_{L^2_v}.
\end{align}
Note that $\F_x\Gamma(f,g)=\int\Gamma(\widehat{f}(\xi-\zeta),\widehat{g}(\zeta))\,d\zeta$, where $\widehat{f}=\F_xf$.
Hence, for $m\in\R$, 
\begin{align*}
&\quad\ \Big|\Big(\<D_x\>^m(a^{-1/2})^w\Gamma(f,g),\<D_x\>^m(a^{1/2})^wh\Big)_{L^2_{x,v}}\Big|\\
&= \Big|\Big(\<\xi\>^m(a^{-1/2})^w\int\Gamma(\widehat{f}(\xi-\zeta),\widehat{g}(\zeta))\,dz,\<\xi\>^m(a^{1/2})^w\widehat{h}\Big)_{L^2_{\xi,v}}\Big|\\
&\le  \int\int\<\xi\>^{2m}\Big|\Big((a^{-1/2})^w\Gamma(\widehat{f}(\xi-\zeta),\widehat{g}(\zeta)),(a^{1/2})^w\widehat{h}(\xi)\Big)_{L^2_{v}}\Big|\,d\zeta d\xi\\
&\le C_m\int\int\|\<\xi-\zeta\>^m\widehat{f}(\xi-\zeta)\|_{L^2_{v}}\|(a^{1/2})^w\widehat{g}(\zeta)\|_{L^2_{v}}\|\<\xi\>^m(a^{1/2})^w\widehat{h}(\xi)\|_{L^2_{v}}\,d\zeta d\xi\\
&\qquad+ C_m\int\int\|\widehat{f}(\xi-\zeta)\|_{L^2_{v}}\|\<z\>^m(a^{1/2})^w\widehat{g}(\zeta)\|_{L^2_{v}}\|\<\xi\>^m(a^{1/2})^w\widehat{h}(\xi)\|_{L^2_{v}}\,d\zeta d\xi\\
&\le C_m\|\<D_x\>^mf\|_{L^2_{x,v}}\big\|\|(a^{1/2})^w\widehat{g}(\xi)\|_{L^2_{v}}\big\|_{L^1_\xi}\|\<D_x\>^m(a^{1/2})^w\widehat{h}\|_{L^2_{x,v}}\\
&\qquad + C_m\big\|\|\widehat{f}(\xi)\|_{L^2_{v}}\big\|_{L^1_\xi}\|\<D_x\>^m(a^{1/2})^wg\|_{L^2_{x,v}}\|\<D_x\>^m(a^{1/2})^w{h}\|_{L^2_{x,v}},
\end{align*}by H\"{o}lder's inequality and Fubini's theorem since $f,g\in\S$. 
In particular, when $m>\frac{d}{2}$, we have $\|\widehat{g}\|_{L^1_x}\le \|\<\xi\>^{-m}\|_{L^2_\xi}\|\<\xi\>^m\widehat{g}\|_{L^2_\xi}\le C\|\<D_x\>^mg\|_{L^2_x}$ and hence, 
\begin{align}
\Big|\Big(\<D_x\>^m&(a^{-1/2})^w\Gamma(f,g),\<D_x\>^m(a^{1/2})^wh\Big)_{L^2_{x,v}}\Big|\notag\\
&\le\notag C_m\|\<D_x\>^mf\|_{L^2_{x,v}}\|\<D_x\>^m(a^{1/2})^wg\|_{L^2_{x,v}}\|\<D_x\>^m(a^{1/2})^wh\|_{L^2_{x,v}},\\
\Big\|\<D_x\>^m&(a^{-1/2})^w\Gamma(f,g)\Big\|_{L^2_{x,v}}\le C_m\|\<D_x\>^mf\|_{L^2_{x,v}}\|\<D_x\>^m(a^{1/2})^wg\|_{L^2_{x,v}}.\label{eq88}
\end{align}
On the other hand, applying H\"{o}lder's inequality on $x$ in \eqref{eq87}, 
\begin{align}
\Big((a^{-1/2})^w\Gamma(f,g),(a^{1/2})^wh\Big)_{L^2_{x,v}}
&\le\notag C\|f\|_{L^2_{x,v}}\|(a^{1/2})^wg\|_{L^2_{x,v}}\sup_{x\in\Rd}\|(a^{1/2})^wh\|_{L^2_{x,v}},\\
\big\|\|(a^{1/2})^w\Gamma(f,g)\|_{L^2_v}\big\|_{L^1_x}&\le C\|f\|_{L^2_{x,v}}\|(a^{1/2})^wg\|_{L^2_{x,v}}.\label{eq89}
\end{align}
Thus $\Gamma(f,g)$, initially defined on $\S(\Rd)\times\S(\Rd)$, can uniquely extend to a bilinear continuous operator on $L^2_vH^m_x\times \H H^m_x$ and $L^2_{x,v}\times \H L^2_x$.

\subsection{Estimate on space $X$}Let $\delta>0$ be a small constant chosen later. We define inner product 
\begin{align}\label{DefX}
(f,g)_X = \delta(f,g)_{L^2_vH^m_x} + \int^{\infty}_0(e^{\tau B}f,e^{\tau B}g)_{L^2_vH^m_x}\,d\tau,
\end{align}and the corresponding norm $\|\cdot\|^2_X:=(\cdot,\cdot)_X$.

For $m\in\R$, assume $f_0\in L^2_{v}H^m_x\cap D(B)$. By semigroup theory, the solution to 
\begin{align*}
f_t = Bf, \qquad f|_{t=0}=f_0, 
\end{align*}is $f = e^{tB}f_0\in D(B)$. Notice that $\<D_x\>$ commutes with $A$, thus by a similar argument in Theorem \ref{ASchwarz} (1), we can estimate the closure of $v\cdot\nabla_x+A$ on $L^2_{x,v}$ as following. 
\begin{align}
\Re((v\cdot\nabla_x+A)f,f)_{L^2_vH^{m}_x} &\ge \nu_0\|f\|_{\H H^m_x},\text{ for } f\in H(\<v\>\<\xi\>)\cap H(a),\notag\\
\Re((\overline{v\cdot\nabla_x+A})f,f)_{L^2_vH^{m}_x} &\ge \nu_0\|f\|_{\H H^m_x},\text{ for } f\in D(\overline{v\cdot\nabla_x+A})=D(B).\label{eq77}
\end{align}
Thus for $m\in\R$, $t_0>0$, 
\begin{align*}
\frac{1}{2}\frac{d}{dt}\|f\|^2_{L^2_vH^m_x} = \Re((\overline{-v\cdot\nabla_x-A}+K)f,f)_{L^2_vH^m_x} &\le -\nu_0\|f\|^2_{\H H^m_x} + C\|f\|^2_{L^2_vH^m_x},\\
\frac{1}{2}\sup_{0\le t\le t_0}\|f(t)\|^2_{L^2_vH^m_x}+\nu_0\int^{t_0}_0\|f(t)\|^2_{\H H^m_x}\,dt &\le \frac{1}{2}\|f(0)\|^2_{L^2_vH^m_x}+C\int^{t_0}_0\|f\|^2_{L^2_vH^m_x}\,dt\\
&\le \frac{1}{2}\|f_0\|^2_{L^2_vH^m_x}+Ct_0\sup_{0\le t\le t_0}\|f\|^2_{L^2_vH^m_x},\\
\frac{1}{4}\sup_{0\le t\le t_0}\|e^{tB}f_0\|^2_{L^2_vH^m_x}+\nu_0\int^{t_0}_0\|e^{tB}f_0\|^2_{\H H^m_x}\,dt 
&\le \frac{1}{2}\|f_0\|^2_{L^2_vH^m_x}.
\end{align*}if we choose $t_0=\frac{1}{4C}$. 
The estimate on the sun dual semigroup $(e^{tB})^{\odot}$ of $e^{tB}$ satisfies the same estimate, cf. II.2.6 in \cite{Engel1999}, since $B^*=(-v\cdot\nabla_x+L)^*=\overline{v\cdot\nabla_x+L}$. Therefore
\begin{align}\notag
\int^{t_0}_0\|f(t)\|^2_{L^2_vH^m_x}\,dt
&=\int^{t_0}_0\|e^{t B}f_0\|^2_{L^2_vH^m_x}\,dt\\
&=\int^{t_0}_0\lim_{n\to\infty}|(e^{t B}f_0,g_n)_{L^2_vL^2_x}|^2\,dt\notag\\
&\le\liminf_{n\to\infty}\|f_0\|^2_{H(a^{-1/2})H^m_x}\int^{t_0}_0\|e^{t B^*}g_n\|_{\H H^{-m}_x}^2\,dt\notag\\
&\le\frac{1}{2\nu_0}\liminf_{n\to\infty}\|f_0\|^2_{H(a^{-1/2})H^m_x}\|g_n\|_{L^2_v H^{-m}_x}^2\notag\\
&\le\frac{1}{2\nu_0}\|f_0\|^2_{H(a^{-1/2})H^m_x},\label{eqq4}
\end{align}for some sequence $\{g_n\}\subset\S$ with $\|g_n\|_{L^2_vH^{-m}_x}=1$. 
For large time $t\ge t_0$, we apply Theorem \eqref{Thm1} to $f=e^{tB}f_0$ with $k=2$, then 
\begin{align*}
&\quad\ \int^\infty_{t_0}\|e^{tB}f_0\|^2_{H(a^{1/2})H^m_x}\,dt\\
  &\le \int^\infty_{t_0}\Big(\frac{e^{-2\kappa t}C}{t^{4}}\|(a^{-1/2})^wf_0\|^2_{L^2_vH^m_x}
  +\frac{C}{(1+t)^{d/2(2/p-1)}}\|(a^{-1/2})^wf_0\|^2_{L^2_v(L^p_x)}\Big)\,dt\\
  &\le C\Big(\|(a^{-1/2})^wf_0\|^2_{L^2_vH^m_x}+\|(a^{-1/2})^wf_0\|^2_{L^2_v(L^p_x)}\Big),
\end{align*}for $d\ge 3$ and $\frac{d}{2}(\frac{2}{p}-1)>1$. Together with \eqref{eqq4}, we have 
\begin{align}\label{eqq1}
\int^\infty_{0}\|e^{tB}f_0\|^2_{L^2_vH^m_x}\,dt&\le C\Big(\|(a^{-1/2})^wf_0\|^2_{L^2_vH^m_x}+\|(a^{-1/2})^wf_0\|^2_{L^2_v(L^p_x)}\Big).
\end{align}

\subsection{Proof of Theorem \ref{Thm2}}
Combining the above two sections, we can prove the global existence to 
\begin{align}\label{eq99}
f_t = Bf + \Gamma(f,f),\quad f|_{t=0}=f_0,
\end{align}with $f\in X$.
For notational convenience during the proof, we define a total norm 
\begin{align}\label{Defg}
\g(f(t)) := \delta\|f(t)\|^2_{L^2_vH^m_x}+\int^\infty_0\|e^{\tau B}f(t)\|^2_{L^2_vH^m_x}\,d\tau+2\delta\nu_0\int^t_0\|f(\tau)\|^2_{\H H^m_x}\,d\tau.
\end{align}
\begin{Thm}\label{Thmlinear}Let $d\ge 3$, $m>\frac{d}{2}$. There exists $\varepsilon_0>0$ such that if
\begin{align*}
\|f(0)\|^2_{X}\le \varepsilon^2_0\ \text{ and }\ \sup_{0\le t<\infty}\g(g)\le 2\varepsilon^2_0,
\end{align*}  then the solution $f$ to linear equation
\begin{align}\label{linearequation}
f_t = Bf + \Gamma(g,f),\quad f|_{t=0}=f_0, 
\end{align}
is well defined and satisfies
\begin{align}
\label{eq91}\sup_{0\le t< \infty}\g(f) \le 2\varepsilon^2_0.
\end{align}
\end{Thm}
\begin{proof}
A rather standard procedure (adding the vanishing term $\varepsilon\<D_{x,v}\>^{M}(1+|v|^2+|x|^2)^{2M}\<D_{x,v}\>^{M}f$ for some large $M$ and standard parabolic equation theory with mollified initial data, cf. \cite{MichaelRenardy2004}) gives the existence of solution to \eqref{linearequation}. So we will focus on the proof of \eqref{eq91}. 
Let $f=f(t)$ be a local solution to \eqref{linearequation}, then we have 
\begin{align*}
\frac{1}{2}\frac{d}{dt}\|f\|^2_{X}
&=\frac{\delta}{2}\frac{d}{dt}\|f\|^2_{L^2_vH^m_x}+ \frac{1}{2}\frac{d}{dt}\int^\infty_0\|e^{\tau B}f\|^2_{L^2_vH^m_x}\,d\tau\\
&=\delta\,\Re(Bf,f)_{L^2_vH^m_x}+\int^\infty_0\Re(e^{\tau B}Bf,e^{\tau B}f)_{L^2_vH^m_x}\,d\tau\\
&\qquad+\delta\,\Re(\Gamma(g,f),f)_{L^2_vH^m_x}+\int^\infty_0\Re(e^{\tau B}\Gamma(g,f),e^{\tau B}f)_{L^2_vH^m_x}\,d\tau\\
&=: I_1+I_2+I_3+I_4. 
\end{align*}
For $I_1$, $I_2$, we use Theorem \ref{Thm0}, $\frac{d}{d\tau}e^{\tau B}f=e^{\tau B}Bf$ and \eqref{eq77} to obtain 
\begin{align*}
I_1+I_2&\le -\delta\nu_0\|f\|^2_{\H H^m_x} +C\delta\|f\|^2_{L^2_vH^m_x}+\int^\infty_0\frac{1}{2}\frac{d}{d\tau} \|e^{\tau B}f\|^2_{L^2_vH^m_x}\,d\tau\\
&\le  -\delta\nu_0\|f\|^2_{\H H^m_x} +C\delta\|f\|^2_{L^2_vH^m_x}-\frac{1}{2}{\|f\|^2_{L^2_vH^m_x}}\\
&\le -\delta\nu_0\|f\|^2_{\H H^m_x}-\frac{1}{4}\|f\|^2_{L^2_vH^m_x},
\end{align*}where we choose $\delta=\frac{1}{4C}$. For $I_3$, $I_4$, we apply \eqref{eq88} to get 
\begin{align*}
I_3&\le C\delta\|(a^{-1/2})^w\Gamma(g,f)\|_{L^2_v H^m_x}\|f\|_{\H H^m_x}\\
&\le C\delta\|g\|_{L^2_vH^m_x}\|f\|^2_{\H H^m_x}.
\end{align*}
At last, we apply \eqref{eq88}\eqref{eq89}\eqref{eqq1} to $I_4$, 
\begin{align*}
I_4&\le\int^\infty_0\|e^{\tau B}\Gamma(g,f)\|_{L^2_vH^m_x}\|e^{\tau B}f\|_{L^2_vH^m_x}\,d\tau\\
&\le C\Big(\int^\infty_0\|e^{\tau B}\Gamma(g,f)\|^2_{L^2_vH^m_x}\,d\tau\Big)^{1/2}
\Big(\int^\infty_0\|e^{\tau B}f\|^2_{L^2_vH^m_x}\,d\tau\Big)^{1/2}\\
&\le C\Big(\|(a^{-1/2})^w\Gamma(g,f)\|^2_{L^2_vH^m_x}+\|(a^{-1/2})^w\Gamma(g,f)\|^2_{L^2_v(L^1_x)}\Big)^{1/2}
\Big(\int^\infty_0\|e^{\tau B}f\|^2_{L^2_vH^m_x}\,d\tau\Big)^{1/2}\\
&\le C\|g\|_{L^2_vH^m_x}\|f\|_{\H H^m_x}
\Big(\int^\infty_0\|e^{\tau B}f\|^2_{L^2_vH^m_x}\,d\tau\Big)^{1/2}
\end{align*}
As a conclusion, 
\begin{align*}
\delta\,&\frac{d}{dt}\|f\|^2_{L^2_vH^m_x}+ \frac{d}{dt}\int^\infty_0\|e^{\tau B}f\|^2_{L^2_vH^m_x}\,d\tau+2\delta\nu_0\|f\|^2_{\H H^m_x}\\
&\le C\|g\|_{L^2_vH^m_x}\|f\|^2_{\H H^m_x}+C\|g\|_{L^2_vH^m_x}\|f\|_{\H H^m_x}
\Big(\int^\infty_0\|e^{\tau B}f\|^2_{L^2_vH^m_x}\,d\tau\Big)^{1/2}.
\end{align*}
Then taking the integral on $t$, thanks to \eqref{Defg}, we have
\begin{align*}
\sup_{0\le t< \infty}\g(f)
&\le \|f(0)\|^2_{X}+C\sup_{0\le t< \infty}\|g\|_{L^2_vH^m_x}\int^\infty_0\|f(t)\|^2_{\H H^m_x}\,dt\\
&\quad+C\int^\infty_0\|g(t)\|_{\H H^m_x}\|f(t)\|_{\H H^m_x}\,dt
\sup_{0\le t< \infty}\Big(\int^\infty_0\|e^{\tau B}f\|_{L^2_vH^m_x}\,d\tau\Big)^{1/2}\\
&\le \|f(0)\|^2_{X}+ \frac{C'}{\sqrt{2}}\sup_{0\le t< \infty}\g(g)^{1/2}\g(f).
\end{align*}since $\|\cdot\|_{L^2_v}\le C\|\cdot\|_{\H}$ for hard potential. 
Thus if $\g(g)\le 2\varepsilon^2_0$ and $\|f(0)\|^2_{X}\le \varepsilon^2_0$ with $\varepsilon_0\in(0,\frac{1}{2C'})$, we have 
\begin{align*}
\sup_{0\le t< \infty}\g(f)&\le \|f(0)\|^2_{X} + C'\sup_{0\le t< \infty}\varepsilon_0\g(f),\\
\sup_{0\le t< \infty}\g(f)&\le 2\|f(0)\|^2_{X} \le 2\varepsilon^2_0.
\end{align*}
\qe\end{proof}

After obtaining the uniform energy estimate, we can apply a standard iteration to prove our global existence result.

\begin{proof}[Proof of Theorem \ref{Thm2}]
1. 
Let $f^0=0$ and $f^{n+1}$ $(n\ge 0)$ be the solutions to 
\begin{align*}
f^{n+1}_t = Bf^{n+1} + \Gamma(f^n,f^{n+1}),\quad f^{n+1}|_{t=0}=f_0.
\end{align*}
Then $d^n=f^{n+1}-f^n$ $(n\ge 1)$ solves 
\begin{align*}
d^n_t = Bd^n + \Gamma(f^n,d^n)+\Gamma(d^{n-1},f^{n}),\quad d^{n+1}|_{t=0}=0,
\end{align*}while $d^0=f^1$ satisfies $\g(d^0)\le 2\varepsilon_0^2$. Next we will assume $\g(d^{n-1})\le (2C'\varepsilon_0)^{2n}$ $(n\ge 1)$ for some constant $C'$ found at \eqref{eq90}, then similar to the proof in Theorem \ref{Thmlinear}, we have 
\begin{align*}
\frac{1}{2}\frac{d}{dt}\|d^n\|^2_{X}
&=\frac{\delta}{2}\frac{d}{dt}\|d^n\|^2_{L^2_vH^m_x}+ \frac{1}{2}\frac{d}{dt}\int^\infty_0\|e^{\tau B}d^n\|^2_{L^2_vH^m_x}\,d\tau\\
&=\delta\Re(Bd^n,d^n)_{L^2_vH^m_x}+\int^\infty_0\Re(e^{\tau B}Bd^n,e^{\tau B}d^n)_{L^2_vH^m_x}\,d\tau\\
&\quad+\delta(\Gamma(f^n,d^n)+\Gamma(d^{n-1},f^{n})),d^n)_{L^2_vH^m_x}\\
&\quad+\int^\infty_0\Re(e^{\tau B}\Gamma(f^n,d^n)+\Gamma(d^{n-1},f^{n})),e^{\tau B}d^n)_{L^2_vH^m_x}\,d\tau\\
&=:I_1+I_2+I_3+I_4.
\end{align*}
As in Theorem \ref{Thmlinear}, we have \begin{align*}
I_1+I_2&\le -\delta\nu_0\|d^n\|^2_{\H H^m_x} + {C\delta\|d^n\|^2_{L^2_vH^m_x}}+\int^\infty_0\frac{1}{2}\frac{d}{d\tau} \|e^{\tau B}d^n\|^2_{L^2_vH^m_x}\,d\tau\\
&\le -\delta\nu_0\|d^n\|^2_{\H H^m_x} - {\frac{1}{4}\|d^n\|_{L^2_vH^m_x}},\end{align*}\begin{align*}
I_3&\le C\delta\big(\|f^n\|_{L^2_vH^m_x}\|d^n\|_{\H H^m_x}
+\|d^{n-1}\|_{L^2_vH^m_x}\|f^n\|_{\H H^m_x}\big)\|d^n\|_{\H H^m_x},
\end{align*}\begin{align*}
I_4&\le \int^\infty_0\|e^{\tau B}\big(\Gamma(f^n,d^n)+\Gamma(d^{n-1},f^{n}))\big)\|_{L^2_vH^m_x}\|e^{\tau B}d^n\|_{L^2_vH^m_x}\,d\tau\\
&\le C\big(\|f^n\|_{L^2_vH^m_x}\|d^n\|_{\H H^m_x}
+\|d^{n-1}\|_{L^2_vH^m_x}\|f^n\|_{\H H^m_x}\big)\Big(\int^\infty_0\|e^{\tau B}d^n\|^2_{L^2_vH^m_x}\,d\tau\Big)^{\frac{1}{2}},
\end{align*}where $\delta=\frac{1}{4C}$ as in Theorem \ref{Thmlinear}. Thus 
\begin{align*}
&\quad\ \frac{1}{2}\frac{d}{dt}\|d^n\|^2_{X}+\delta\nu_0\|d^n\|^2_{\H H^m_x}\\
&\le C\delta\big(\|f^n\|_{L^2_vH^m_x}\|d^n\|_{\H H^m_x}
+\|d^{n-1}\|_{L^2_vH^m_x}\|f^n\|_{\H H^m_x}\big)\|d^n\|_{\H H^m_x}\\
&\quad+C\big(\|f^n\|_{L^2_vH^m_x}\|d^n\|_{\H H^m_x}
+\|d^{n-1}\|_{L^2_vH^m_x}\|f^n\|_{\H H^m_x}\big)\g(d^n)
\end{align*}Taking integral on $t$ and using the uniform energy bound in Theorem \ref{linearequation}, we deduce that
\begin{align}
&\quad\sup_{0\le t<\infty}\g(d^n)\notag\\
&\le  C'\sup_{0\le t<\infty}\|f^n\|_{L^2_vH^m_x}\int^\infty_0\|d^n\|^2_{\H H^m_x}\,dt\notag\\
&\quad+\sup_{0\le t<\infty}\|d^{n-1}\|_{L^2_vH^m_x}\int^\infty_0\|f^n\|_{\H H^m_x}\|d^n\|_{\H H^m_x}\,dt\notag\\
&\quad+C\int^\infty_0\big(\|f^n\|_{\H H^m_x}\|d^n\|_{\H H^m_x}
+\|d^{n-1}\|_{L^2_vH^m_x}\|f^n\|_{\H H^m_x}\big)\,dt\g(d^n)^{1/2}\notag\\
&\le C'\sup_{0\le t<\infty}\big(\g(f^n)^{1/2}\g(d^n) + \g(d^{n-1})^{1/2}\g(f^n)^{1/2}\g(d^n)^{1/2}).\label{eq90}
\end{align}where $C'$ is independent of $n$. Thus, when $\g(d^{n-1})\le (2C'\varepsilon_0)^{2n}$, using \eqref{eq91}, we have 
\begin{align*}
\sup_{0\le t<\infty}\g(d^n)&\le C'\big( \sqrt{2}\varepsilon_0\,\sup_{0\le t<\infty}\g(d^n) + (2C'\varepsilon_0)^n\sqrt{2}\varepsilon_0\sup_{0\le t<\infty}\g(d^n)^{1/2}\big),\\
\sup_{0\le t<\infty}\g(d^n)&\le (2C'\varepsilon_0)^{2(n+1)},
\end{align*}where we choose $\varepsilon_0$ such that $C'\sqrt{2}\varepsilon_0\le1-\frac{1}{\sqrt{2}}$. Recalling the energy \eqref{Defg} and choose $\varepsilon_0$ sufficiently small, the sequence $\{f^n\}_{n\in\N}$ is a Cauchy sequence in $L^\infty([0,\infty);L^2_vH^m_x)$ and $L^2([0,\infty);\H H^m_x)$. Hence its limit $f$ solves \eqref{eq99} in the weak sense. We then deduce that 
\begin{align}\label{eq92}
\|f\|_{L^\infty([0,\infty);L^2_vH^m_x)}+\|f\|_{L^2([0,\infty);\H H^m_x)}\le C''\varepsilon_0,
\end{align}by passing the limit $n\to\infty$. 

2. To prove the uniqueness, we suppose that there exists another solution $g$ with the same initial data satisfying \eqref{eq92}. Then the difference $f-g$ satisfies 
\begin{align*}
\partial_t(f-g)=B(f-g) + \Gamma(f,f-g)+\Gamma(f-g,g),
\end{align*}in the weak sense. 
Then by \eqref{eq89}\eqref{eq77}, for $T>0$, 
\begin{align*}
&\quad\ \frac{1}{2}\frac{d}{dt}\|f-g\|^2_{L^2_vH^m_x}\\ &= \Re(B(f-g),f-g)_{L^2_vH^m_x}+\Re(\Gamma(f,f-g)+\Gamma(f-g,g),f-g)_{L^2_vH^m_x}\\
&\le -\nu_0\|f-g\|^2_{\H H^m_x}+C\|f-g\|^2_{L^2_vH^m_x}+C\|f\|_{L^2_vH^m_x}\|f-g\|^2_{\H H^m_x}\\&\quad+C\|f-g\|_{L^2_vH^m_x}\|g\|_{\H H^m_x}\|f-g\|_{\H H^m_x}.
\end{align*}
Taking integral on $t\in[0,T]$,\begin{align*}
&\quad\sup_{0\le t\le T}\|f-g\|^2_{L^2_vH^m_x}+2\nu_0\int^T_0\|f-g\|^2_{\H H^m_x}\,dt\\&\le C\int^T_0\|f-g\|^2_{L^2_vH^m_x}\,dt + C\sup_{0\le t\le T}\|f\|_{L^2_vH^m_x}\int^T_0\|f-g\|^2_{\H H^m_x}\,dt \\&\quad+ C\sup_{0\le t\le T}\|f-g\|_{L^2_vH^m_x}\int^T_0\|g\|_{\H H^m_x}\|f-g\|_{\H H^m_x}\,dt\\
&\le C\int^T_0\|f-g\|^2_{L^2_vH^m_x}\,dt + C''\varepsilon_0\int^T_0\|f-g\|^2_{\H H^m_x}\,dt \\&\quad+ C''\varepsilon_0\sup_{0\le t\le T}\|f-g\|_{L^2_vH^m_x}\Big(\int^T_0\|f-g\|^2_{\H H^m_x}\,dt\Big)^{1/2},
\end{align*}and when $\varepsilon_0>0$ is sufficiently small, we can deduce the uniqueness by Gronwall's inequality. 
\qe\end{proof}

\medskip

{\bf Acknowledgment} The author would like to thank Professor Yang Tong and the
colleagues and reviewers for their valuable comments on this paper. 
D.-Q Deng was supported by Direct Grant from BIMSA and YMSC. 
Data sharing not applicable to this article as no datasets were generated or analysed during the current study. 

\section{Appendix}\label{sec6}

\subsection{Operator theory}
For operator theory, one may refer to \cite{Pedersen2001, Schmuedgen2012, Gohberg1990}.
\begin{Def}\label{closure}
Let $(A,D(A))$ be a closable linear operator on Banach space $X$. Let $
G(A) := \{(x,Ax)|x\in D(A)\}$ and 
\begin{align*}
D(\overline{A}):=\{x\in X:\exists\, y\in X\, s.t.\, (x,y)\in \overline{G(A)}\}.
\end{align*}
Denote $\overline{A}$ maps $x\in D(\overline{A})$ to the corresponding $y$. Such $\overline{A}$ is well-defined and is called the closure of $A$. Then $G(\overline{A})=\overline{G(A)}$. 
\end{Def}
\begin{Def}\label{Defadjoint}
	Let $(A,D(A))$ be a linear unbounded densely defined operator from Hilbert space $H_1$ into Hilbert space $H_2$ with domain $D(A)$. Define 
	\begin{align*}
	D(A^*) = \Big\{y\in H_2\big|x\mapsto\<Ax,y\> \text{ is continuous from $H_1$ to $\C$} \Big\}.
	\end{align*}
	Take $y\in D(A^*)$, since $\overline{D(A)}=H_1$, the functional $F_y(x):=\<Ax,y\>$ has a unique bounded extension on $H_1$. Hence Riesz representation theorem ensures the existence of a unique $z\in H_1$ such that $F_y(x) = \<x,z\>$ for $x\in H_1$. Define $A^*y:=z$, then 
	\begin{align*}
	\<Ax,y\> = \<x,A^*y\>.
	\end{align*}The operator $A^*(H_2\to H_1)$ is linear and is called the adjoint of $A$. Then $(\overline{A})^*=A^*$ and  $(A^*)^{-1}=({A}^{-1})^*$ if ${A}^{-1}$ exists.
\end{Def}

\begin{Lem}[\cite{Lax2002}, Theorem 17.2]\label{operator_inverse}
  Let $T$,$S$ be any operator in $\mathscr{L}(L^2)$ such that $S$ is invertible and $\|T\|<\frac{1}{\|S^{-1}\|}$, then
  \begin{align*}
    (S-T)^{-1} = \sum^\infty_{n=0}(S^{-1}T)^nS^{-1},
  \end{align*}and hence $
    \|(S-T)^{-1}\|\le (1-\|S^{-1}\|\|T\|)\|S^{-1}\|.$
\end{Lem}

 \subsection{Pseudo-differential calculus}\label{PD}

   We recall some notation and theorem of pseudo differential calculus. For details, one may refer to Chapter 2 in the book \cite{Lerner2010}, Proposition 1.1 in \cite{Bony1998-1999} and \cite{Beals1981,Bony1994}. Set $\Gamma=|dv|^2+|d\eta|^2$, but also note that the following are also valid for general admissible metric.
   Let $M$ be an $\Gamma$-admissible weight function. That is, $M:\R^{2d}\to (0,+\infty)$ satisfies the following conditions:\\
   (a). (slowly varying) there exists $\delta>0$ such that for any $X,Y\in\R^{2d}$, $|X-Y|\le \delta$ implies
   \begin{align*}
     M(X)\approx M(Y);
   \end{align*}
   (b) (temperance) there exists $C>0$, $N\in\R$, such that for $X,Y\in \R^{2d}$,
   \begin{align*}
     \frac{M(X)}{M(Y)}\le C\<X-Y\>^N.
   \end{align*}
   A direct result is that if $M_1,M_2$ are two $\Gamma$-admissible weight, then so is $M_1+M_2$ and $M_1M_2$. Consider symbols $a(v,\eta,\xi)$ as a function of $(v,\eta)$ with parameters $\xi$. We say that
   $a\in S(\Gamma)=S(M,\Gamma)$ uniformly in $\xi$, if for $\alpha,\beta\in \N^d$, $v,\eta\in\Rd$,
   \begin{align*}
     |\partial^\alpha_v\partial^\beta_\eta a(v,\eta,\xi)|\le C_{\alpha,\beta}M,
   \end{align*}with $C_{\alpha,\beta}$ a constant depending only on $\alpha$ and $\beta$, but independent of $\xi$. The space $S(M,\Gamma)$ endowed with the seminorms
   \begin{align*}
     \|a\|_{k;S(M,\Gamma)} = \max_{0\le|\alpha|+|\beta|\le k}\sup_{(v,\eta)\in\R^{2d}}
     |M(v,\eta)^{-1}\partial^\alpha_v\partial^\beta_\eta a(v,\eta,\xi)|,
   \end{align*}becomes a Fr\'{e}chet space.
   Sometimes we write $\partial_\eta a\in S(M,\Gamma)$ to mean that $\partial_{\eta_j} a\in S(M,\Gamma)$ $(1\le j\le d)$ equipped with the same seminorms.
   We formally define the pseudo-differential operator by
   \begin{align*}
     (op_ta)u(x)=\int_\Rd\int_\Rd e^{2\pi i (x-y)\cdot\xi}a((1-t)x+ty,\xi)u(y)\,dyd\xi,
   \end{align*}for $t\in\R$, $f\in\S$.
   In particular, denote $a(v,D_v)=op_0a$ to be the standard pseudo-differential operator and
   $a^w(v,D_v)=op_{1/2}a$ to be the Weyl quantization of symbol $a$. We write $A\in Op(M,\Gamma)$ to represent that $A$ is a Weyl quantization with symbol belongs to class $S(M,\Gamma)$. One important property for Weyl quantization of a real-valued symbol is the self-adjoint on $L^2$ with domain $\S$. 

 Let $a_1(v,\eta)\in S(M_1,\Gamma),a_2(v,\eta)\in S(M_2,\Gamma)$, then $a_1^wa_2^w=(a_1\#a_2)^w$, $a_1\#a_2\in S(M_1M_2,\Gamma)$ with
 \begin{align*}
   a_1\#a_2(v,\eta)&=a_1(v,\eta)a_2(v,\eta)
   +\int^1_0(\partial_{\eta}a_1\#_\theta \partial_{v} a_2-\partial_{v} a_1\#_\theta \partial_{\eta} a_2)\,d\theta,\\
   g\#_\theta h(Y):&=\frac{2^{2d}}{\theta^{-2n}}\int_\Rd\int_\Rd e^{-\frac{4\pi i}{\theta}\sigma(X-Y_1)\cdot(X-Y_2)}(4\pi i)^{-1}\<\sigma\partial_{Y_1}, \partial_{Y_2}\>g(Y_1) h(Y_2)\,dY_1dY_2,
 \end{align*}with $Y=(v,\eta)$, $\sigma=\begin{pmatrix}
0&I\\-I&0
\end{pmatrix}$.
 For any non-negative integer $k$, there exists $l,C$ independent of $\theta\in[0,1]$ such that
 \begin{align}\label{sharp_theta}
   \|g\#_\theta h\|_{k;S(M_1M_2,\Gamma)}\le C\|g\|_{l,S(M_1,\Gamma)}\|h\|_{l,S(M_2,\Gamma)}.
 \end{align}
 Thus if $\partial_{\eta}a_1,\partial_{\eta}a_2\in S(M'_1,\Gamma)$ and $\partial_{v}a_1,\partial_{v}a_2\in S(M'_2,\Gamma)$, then $[a_1,a_2]\in S(M'_1M'_2,\Gamma)$, where $[\cdot,\cdot]$ is the commutator defined by $[A,B]:=AB-BA$.

 We can define a Hilbert space $H(M,\Gamma):=\{u\in\S':\|u\|_{H(M,\Gamma)}<\infty\}$, where
 \begin{align}\label{sobolev_space}
   \|u\|_{H(M,\Gamma)}:=\int M(Y)^2\|\varphi^w_Yu\|^2_{L^2}|g_Y|^{1/2}\,dY<\infty,
 \end{align}and $(\varphi_Y)_{Y\in\R^{2d}}$ is any uniformly confined family of symbols which is a partition of unity. If $a\in S(M)$ is a isomorphism from $H(M')$ to $H(M'M^{-1})$, then $(a^wu,a^wv)$ is an equivalent Hilbertian structure on $H(M)$. Moreover, the space $\S(\Rd)$ is dense in $H(M)$ and $H(1)=L^2$.

 Let $a\in S(M,\Gamma)$, then
     $a^w:H(M_1,\Gamma)\to H(M_1/M,\Gamma)$ is linear continuous, in the sense of unique bounded extension from $\S$ to $H(M_1,\Gamma)$.
 Also the existence of $b\in S(M^{-1},\Gamma)$ such that $b\#a = a\#b = 1$ is equivalent to the invertibility of $a^w$ as an operator from $H(MM_1,\Gamma)$
 onto $H(M_1,\Gamma)$ for some $\Gamma$-admissible weight function $M_1$.

 For the metric $\Gamma=|dv|^2+|d\eta|^2$, the map $J^t=\exp(2\pi i D_v\cdot D_\eta)$ is an isomorphism of the Fr\'{e}chet space $S(M,\Gamma)$, with polynomial bounds in the real variable $t$, where $D_v=\partial_v/i$, $D_\eta=\partial_\eta/i$. Moreover, $a(x,D_v)=(J^{-1/2}a)^w$.

Let $m_K(v,\eta)$ be a $\Gamma$-admissible weight function depending on $K$, $c$ be any $\Gamma$-admissible weight. Then Lemmas 2.1 and 2.3 in \cite{Deng2020a} can be reformulated as the following.
 \begin{Lem}\label{inverse_pseudo}
   Assume $a_K\in S(m_K)$, $\partial_\eta (a_{K})\in S(K^{-\kappa}m_{K})$ uniformly in $K$ and $|a_{K}|\gtrsim m_{K}$. Then \\
(1). $a^{-1}_{K}\in S(m^{-1}_{K})$, uniformly in $K$, for $K>1$.\\
(2). There exists $K_0>1$ sufficiently large such that for all $K>K_0$, $a^w_{K}:H(m_Kc)\to H(c)$ is invertible and its inverse $(a^w_{K,l})^{-1}: H(c) \to H(m_Kc)$ satisfies
   \begin{align*}
     (a^w_{K})^{-1} = G_{1,K}(a^{-1}_{K})^w = (a^{-1}_{K})^wG_{2,K},
   \end{align*}where $G_{1,K}\in \L(H(m_Kc))$, $G_{2,K}\in \L(H(c))$ with operator norm smaller than $2$. Also, by the equivalence of invertibility, $(a_K^w)^{-1}\in Op(m^{-1}_K)$.
 \end{Lem}

 \begin{Lem}\label{inverse_bounded_lemma}Let $m,c$ be $\Gamma$-admissible weight and $a\in S(m)$.
   Assume $a^w:H(mc)\to H(c)$ is invertible.
     If $b\in S(m)$, then there exists $C>0$, depending only on the seminorms of symbols to $(a^w)^{-1}$ and $b^w$, such that for $f\in H(mc)$,
     \begin{align*}
       \|b(v,D_v)f\|_{H(c)}+\|b^w(v,D_v)f\|_{H(c)}\le C\|a^w(v,D_v)f\|_{H(c)}.
     \end{align*}
Consequently, if $a^w:H(m_1)\to L^2\in Op(m_1)$, $b^w:H(m_2)\to L^2\in Op(m_2)$ are invertible, then for $f\in\S$, 
\begin{align*}
\|b^wa^wf\|_{L^2}\lesssim \|a^wb^wf\|_{L^2},
\end{align*}where the constant depends only on seminorms of symbols to $a^w,b^w,(a^w)^{-1},(b^w)^{-1}$.
 \end{Lem}

\begin{Lem}\label{formal_ajoint}
Let $a\in S(m)$ be real-valued. Then for $f,g\in\S$, 
\begin{align*}
(a^wf,g)_{L^2}=(f,a^wg)_{L^2}.
\end{align*}Thus by density argument, this identity is also valid for $f,g\in H(m)$. 
\end{Lem}

 \subsection{Carleman representation and cancellation lemma}

 Now we have a short review of some useful facts in the theory of Boltzmann equation. One may refer to \cite{Alexandre2000,Global2019} for details. The first one is the so called Carleman representation. For measurable function $F(v,v_*,v',v'_*)$, if any sides of the following equation is well-defined, then
 \begin{align}
   &\int_{\R^d}\int_{\mathbb{S}^{d-1}}b(\cos\theta)|v-v_*|^\gamma F(v,v_*,v',v'_*)\,d\sigma dv_*\notag\\
   &\quad=\int_{\R^d_h}\int_{E_{0,h}}\tilde{b}(\alpha,h)\1_{|\alpha|\ge|h|}\frac{|\alpha+h|^{\gamma+1+2s}}{|h|^{d+2s}}F(v,v+\alpha-h,v-h,v+\alpha)\,d\alpha dh,\label{Carleman}
 \end{align}where $\tilde{b}(\alpha,h)$ is bounded from below and above by positive constants, and $\tilde{b}(\alpha,h)=\tilde{b}(|\alpha|,|h|)$, $E_{0,h}$ is the hyper-plane orthogonal to $h$ containing the origin. The second is the cancellation lemma. Consider a measurable function $G(|v-v_*|,|v-v'|)$, then for $f\in\S$,
 \begin{align*}
   \int_{\R^d}\int_{\mathbb{S}^{d-1}}G(|v-v_*|,|v-v'|)b(\cos\theta)(f'_*-f_*)\,d\sigma dv_* = S*_{v_*}f(v),
 \end{align*}where $S$ is defined by, for $z\in\R^d$,
 \begin{align*}
   S(z)=2\pi \int^{\pi/2}_0 b(\cos\theta)\sin\theta\left(G(\frac{|z|}{\cos\theta/2},\frac{|z|\sin\theta/2}{\cos\theta/2}) - G(|z|,|z|\sin(\theta/2))\right)\,d\theta.
 \end{align*}

 \subsection{Semigroup theory}

 Here we write some well-known result from semigroup theory. One may refer to \cite{Engel1999} for more details.
 \begin{Def}
   An operator $(A,D(A))$ is dissipative if and only if for every $x \in D(A)$ there exists $j(x) \in \{x'\in X':\<x,x'\>=\|x\|^2=\|x'\|^2\}$ such that
 \begin{align*}
   \Re\<Ax,j(x)\>\le 0.
 \end{align*}
 \end{Def}
 \begin{Thm}\label{semigroup}
   For a densely defined, dissipative operator $(A,D(A))$ on a Banach space $X$ the following statements are equivalent.\\
 (a) The closure $\overline{A}$ of $A$ generates a contraction semigroup.\\
 (b) $Im(\lambda I - A)$ is dense in $X$ for some (hence all) $\lambda>0$.
 \end{Thm}	
\begin{Thm}(Hille-Yosida Theorem)\label{semigroup_Hille}
	For a linear operator $(A,D(A))$ on a Banach space $X$, the following properties are equivalent.
		
	(a) $(A,D(A))$ generates a strongly continuous contraction semigroup.
		
	(b) $(A,D(A))$ is closed, densely defined, and for every $\lambda\in\C$ with $\Re\lambda>0$ one has $\lambda\in\rho(A)$ and
	$\|(\lambda I-A)^{-1}\|\le \frac{1}{\Re\lambda}$.
\end{Thm} 
 \begin{Thm}\label{semigroup_bounded_per}
   Let $(A,D(A))$ be the generator of a strongly continuous semigroup
   $(T(t))_{t\ge 0}$ on a Banach space $X$
   satisfying $\|T(t)\|\le Me^{\omega t}$ for all $t \ge 0$
   and some $\omega\in\R$, $M\ge 1$. If $B\in L(X)$, then
   $C := A + B$ with $D(C) := D(A)$
   generates a strongly continuous semigroup $(S(t))_{t\ge 0}$ satisfying
   $\|S(t)\|\le Me^{(\omega+M\|B\|)t}$ for all $t \ge 0$.
 \end{Thm}
At last, we state the lemma that was applied in our main theorem.
\begin{Thm}On $L^2_{x,v}$, we have 
\begin{align}\label{Fouriersemigroup}
 e^{tB} = \F^{-1}_xe^{t\B}\F_x.
\end{align}
\end{Thm}
\begin{proof}
Recall that $B$, $\B$ generate strongly continuous semigroup $e^{tB}$, $e^{t\B}$ on $L^2_{x,v}$ and $L^2_v$ respectively. Thus, it suffices to proves \eqref{Fouriersemigroup} on $\S(\R^{2d}_{x,v})$. By the construction of semigroup as Hille-Yoshida Theorem II.3.5 in \cite{Engel1999}, for $n>0$ sufficiently large, we let  
\begin{align*}
B_n:&=n B(n I-B)^{-1}=n^2(n I-B)^{-1}-n I,\\
\widehat{B}_n(\xi):&=n\widehat{B}_n(\xi)(n I-\widehat{B}_n(\xi))^{-1}=n^2(n I-\widehat{B}_n(\xi))^{-1}-n I.
\end{align*}
Then $B_n$, $\widehat{B}_n(\xi)$ are bounded operator on $L^2_{x,v}$ and $L^2_v$ respectively and for $f\in\S$, $B_nf\to Bf$, $\widehat{B}_n(\xi)f\to \B f$ as $n\to\infty$. Also,
\begin{align*}
e^{tB}f=\lim_{n \to\infty}e^{tB_n}f,\ \text{ for } f\in L^2(\R^{2d}_{x,v}),\\
e^{t\widehat{B}(\xi)}f=\lim_{n \to\infty}e^{t\widehat{B}_n(\xi)}f,\ \text{ for } f\in L^2(\R^{d}_{v}).
\end{align*}
Now fix $f\in\S(\R^{2d}_{x,v})$ and let $\varphi(x,v)=(nI-B)^{-1}f\in L^2_{x,v}$. Then $(\overline{nI+v\cdot\nabla_x+A})\varphi = f+K\varphi$. Applying Theorem \ref{adjointofclosure} to operator $nI+v\cdot\nabla_x+A$, we have $\overline{nI+v\cdot\nabla_x+A}^*=\overline{nI-v\cdot\nabla_x+A}$ on $L^2$. Thus for $\psi\in\S(\R^{2d})$, 
\begin{align*}
(\varphi,(n-v\cdot\nabla_x+A)\psi)_{L^2_{x,v}}=(f+K\varphi,\psi)_{L^2_{x,v}}.
\end{align*}
On the other hand, applying Theorem \ref{L2xv_regularity} to $f+K\varphi\in\S(\R^{2d})$, there exists $g\in \cap^\infty_{k=1}H((K_0+|v|^2+|\xi|^4+|\eta|^2)^k)$ $(K_0>>1)$ such that for $\psi\in\S(\R^{2d})$, 
\begin{align*}
(g,(n-v\cdot\nabla_x+A)\psi)_{L^2_{x,v}}=(f+K\varphi,\psi)_{L^2_{x,v}}.
\end{align*}
We choose $k>>1$ such that $ H((K_0+|v|^2+|\xi|^4+|\eta|^2)^k)\subset H(a)\cap H(\<v\>\<\xi\>)$. Then by density of $\S$ in $H((K_0+|v|^2+|\xi|^4+|\eta|^2)^k)$, we have for $\psi\in H((K_0+|v|^2+|\xi|^4+|\eta|^2)^k)$, 
\begin{align*}
(\varphi-g,(n-v\cdot\nabla_x+A)\psi)_{L^2_{x,v}}=0.
\end{align*}
Applying Theorem \ref{L2xv_regularity} and its Remark \ref{L2xv_rem} to any $h\in\S(\R^{2d})$, there exists $\psi\in H((K_0+|v|^2+|\xi|^4+|\eta|^2)^k)$ such that 
\begin{align*}
(n-v\cdot\nabla_x+A)\psi = h. 
\end{align*}
Hence for $h\in\S(\R^{2d}_{x,v})$, 
\begin{align*}
(\varphi-g,h)_{L^2_{x,v}}=0.
\end{align*}So $g=\varphi = (nI-B)^{-1}f\in H((K_0+|v|^2+|\xi|^4+|\eta|^2)^k)\subset H(a)\cap H(\<v\>\<\xi\>)$.
Therefore, $(nI+v\cdot\nabla_x-L)\varphi = f$, where we can cancel the closure. Taking the Fourier transform $\F$ acting on spatial variable, $(nI+2\pi iv\cdot \xi-L)\F\varphi(y,v) = \F f(y,v)$ and so
\begin{align*}
\F\varphi(y,v) &= (nI-\B)^{-1}\F f,\\
\varphi(x,v) &= (nI-B)^{-1}f =  \F^{-1}(nI-\B)^{-1}\F f.
\end{align*}
By definition of strongly continuous semigroup generated by bounded operator, we have 
\begin{align*}
e^{tB_n}f = \lim_{N\to\infty}\sum^N_{j=0}\frac{(tn^2(nI-B)^{-1}-tn)^j}{j!}f,\ \text{ in }L^2_{x,v},\\
e^{t\widehat{B}_n(\xi)}f = \lim_{N\to\infty}\sum^N_{j=0}\frac{(tn^2(nI-\widehat{B}_n(\xi))^{-1}-tn)^j}{j!}f,\ \text{ in }L^2_{v}.
\end{align*}Hence the above limits are also valid in the sense of almost everywhere.
For each $j\ge 0$, 
\begin{align*}
(tn^2(nI-B)^{-1}-tn)^jf &= \F^{-1}(tn^2(nI-\widehat{B}_n(\xi))^{-1}-tn)^j\F f,\\
\lim_{N\to\infty}\sum^N_{j=0}\frac{(tn^2(nI-B)^{-1}-tn)^jf}{j!} &= \lim_{N\to\infty}\sum^N_{j=0}\frac{\F^{-1}(tn^2(nI-\widehat{B}_n(\xi))^{-1}-tn)^j\F f}{j!},
\end{align*}where the limit is taken in $L^2_{x,v}$. Noting that for any convergent sequence $\{f_N\}\subset L^2_{x,v}$, $\{\F^{-1}f_N\}$ is also a convergent sequence in $L^2_{x,v}$ and 
\begin{align*}
&\quad\ \|\lim_{N\to\infty}\F^{-1}f_N-\F^{-1}\lim_{N\to\infty}f_N\|_{L^2_{x,v}}\\
&\le \|\lim_{N\to\infty}\F^{-1}f_N-\F^{-1}f_N\|_{L^2_{x,v}} +\|\F^{-1}f_N-\F^{-1}\lim_{N\to\infty}f_N\|_{L^2_{x,v}}\\&\to_{N\to\infty} 0.
\end{align*}
Therefore, combining the above estimates, 
\begin{align*}
e^{tB_n}f &= \lim_{N\to\infty}\sum^N_{j=0}\frac{(tn^2(nI-B)^{-1}-tn)^jf}{j!}\\
&= \F^{-1}\lim_{N\to\infty}\sum^N_{j=0}\frac{(tn^2(nI-\widehat{B}_n(\xi))^{-1}-tn)^j\F f}{j!}\\
&= \F^{-1}e^{t\widehat{B}_n(\xi)}\F f,
\end{align*}where the first limit is taken in $L^2_{x,v}$, the last equality is viewed as almost everywhere point-wise limit (up to a subsequence of $N$).
\qe\end{proof}

\subsection{Preliminary lemma on linearized Boltzmann operator $L$}

\begin{Lem}
\label{basicL}
(1). $Ker{L} = \{\mu,v_1\mu,\cdots,v_d\mu,|v|^2\mu\}.$

(2).
For $f,g\in\S$, 
\begin{align}\label{eq114}
(L_1f,g)_{L^2}=(f,L_1g)_{L^2},\qquad
(L_2f,g)_{L^2}=(f,L_2g)_{L^2}.
\end{align}

(3). For any rotation $R$ on $\Rd$, we have $RL_1=L_1R$ and $RL_2=L_2R$. 
\end{Lem}
\begin{proof}
1. 
Let $f\in Ker{L}$, then 
\begin{align*}
\overline{L}f &= 0,\\
\overline{A}f &= K f.
\end{align*}
By Theorem \ref{L2v_regularity}, $f\in\S$ and so the collision invariant identity can be applied to get $\Ker L$ as in the cuf-off case, for instance \cite{Cercignani1994}.

2. In order to prove \eqref{eq114}, we can apply the following change of variable in $(L_1f,g)$ and $(L_2f,g)$ respectively:
\begin{align*}
&\quad \ \, \mu^{1/2}(v+\alpha-h)g(v)\mu^{1/2}(v+\alpha)f(v-h)\\
&\mapsto \mu^{1/2}(v+\alpha)g(v+h)\mu^{1/2}(v+\alpha+h)f(v),\quad v\mapsto v+h,\\
&\mapsto \mu^{1/2}(v+\alpha)g(v-h)\mu^{1/2}(v+\alpha-h)f(v),\quad h\mapsto -h,
\end{align*}and
\begin{align*}
&\quad \ \,\mu^{1/2}(v+\alpha-h)g(v)\mu^{1/2}(v-h)f(v+\alpha)\\
&\mapsto\mu^{1/2}(v-\alpha-h)g(v)\mu^{1/2}(v-h)f(v-\alpha),\quad \alpha\mapsto-\alpha,\\
&\mapsto\mu^{1/2}(v-h)g(v+\alpha)\mu^{1/2}(v+\alpha-h)f(v),\quad v\mapsto v+\alpha,\\
&\quad \ \,\mu^{1/2}(v+\alpha-h)g(v)\mu^{1/2}(v)f(v+\alpha-h)\\
&\mapsto\mu^{1/2}(v)g(v-\alpha+h)\mu^{1/2}(v-\alpha+h)f(v),\quad v\mapsto v-\alpha+h,\\
&\mapsto\mu^{1/2}(v)g(v+\alpha-h)\mu^{1/2}(v+\alpha-h)f(v),\quad (h,\alpha)\mapsto (-h,-\alpha).
\end{align*}

3. For any rotation $R$, i.e. orthogonal matrix acting on $v$. Noticing $\alpha\perp h$ and $\mu(v)$ depends only on $|v|$, we apply change of variable $(\alpha,h)\mapsto (R\alpha,R h)$ to get 
\begin{align*}
RL_1f(v) :&= L_1f(Rv) \\
&= \lim_{\varepsilon\to 0}\int_{|h|\ge\varepsilon}\int_{E_{0,h}}\tilde{b}(\alpha,h)\1_{|\alpha|\ge|h|}\frac{|\alpha+h|^{\gamma+1+2s}}{|h|^{d+2s}}\mu^{1/2}(Rv+\alpha-h)\notag\\	&\qquad\qquad\Big((\mu^{1/2}(Rv+\alpha)f(Rv-h)-\mu^{1/2}(Rv+\alpha-h)f(Rv)\Big)\,d\alpha dh\\
&= \lim_{\varepsilon\to 0}\int_{|h|\ge\varepsilon}\int_{E_{0,h}}\tilde{b}(\alpha,h)\1_{|\alpha|\ge|h|}\frac{|\alpha+h|^{\gamma+1+2s}}{|h|^{d+2s}}\mu^{1/2}(v+R^T\alpha-R^T h)\\	
&\qquad\qquad\Big((\mu^{1/2}(v+R^T\alpha)f(Rv-h)-\mu^{1/2}(v+R^T\alpha-R^T h)f(Rv)\Big)\,d\alpha dh\\
&= \lim_{\varepsilon\to 0}\int_{|h|\ge\varepsilon}\int_{E_{0,h}}\tilde{b}(\alpha,h)\1_{|\alpha|\ge|h|}\frac{|\alpha+h|^{\gamma+1+2s}}{|h|^{d+2s}}\mu^{1/2}(v+\alpha- h)\\	
&\qquad\qquad\Big((\mu^{1/2}(v+\alpha)f(Rv-Rh)-\mu^{1/2}(v+\alpha-h)f(Rv)\Big)\,d\alpha dh\\
&= L_1Rf(v).
\end{align*}$L_2$ is similar. 
\qe\end{proof}

\begin{Thm}{Implicit Function Theorem}
\label{implicit_function}
Let $U\subset \R^n\times\R^m$ be an open subset, $f\in C^k(U;\R^m)$, $k\in\N^+$.
Assume $(x_0,y_0)\in U$, $z_0\in\R^m$ and
 $f(x_0,y_0)=z_0$, $|\text{det}\nabla_yf(x_0,y_0)|\neq 0$.
Then there exists an open set $V\subset U$ with $(x_0,y_0)\in V$, an open set $W\subset\R^n$ with $x_0\in W$ and a $C^k$ mapping $g:W\to\R^m$ such that 

(i). $g(x_0)=y_0$,

(ii). $f(x,g(x))=z_0$, for $x\in W$, 

(iii) and if $(x,y)\in V$ and $f(x,y)=z_0$, then $y=g(x)$. 
\end{Thm}

Introduce an orthonormal basis of $\Ker L$:
\begin{align}\label{Kernel}
\psi_0=\mu^{1/2},\ \psi_i=v_i\mu^{1/2}\,(i=1,\cdots,d),\ \psi_{d+1}=\frac{1}{\sqrt{2d}}(|v|^2-d)\mu^{1/2}.
\end{align}
\begin{Lem}\label{eigenP0}
The eigenvalue problem in $\Ker L$:
\begin{align*}
P_0(v_1P_0\varphi) = \eta\varphi
\end{align*}has eigenvalues and orthonormal eigenfunctions:
\begin{equation}\label{eigen_matrix_eq}\left\{
\begin{aligned}
\eta_{0,d+1}=-\sqrt{\tfrac{d+2}{d}},\  \psi_{0,d+1}&= \sqrt{\tfrac{d}{2(d+2)}}\psi_0-\tfrac{1}{\sqrt{2}}\psi_1+\sqrt{\tfrac{1}{d+2}}\psi_{d+1},\\
\eta_{0,0}=\sqrt{\tfrac{d+2}{d}},\quad\quad \psi_{0,0} &= \sqrt{\tfrac{d}{2(d+2)}}\psi_0+\tfrac{1}{\sqrt{2}}\psi_1+\sqrt{\tfrac{1}{d+2}}\psi_{d+1},\\
\eta_{0,1}=0,\quad\quad\quad\quad \psi_{0,1} &= -\sqrt{\tfrac{2}{d+2}}\psi_0+\sqrt{\tfrac{d}{d+2}}\psi_{d+1},\\
\eta_{0,j}=0,\quad\quad\quad\quad \psi_{0,j} &= \psi_{j},\quad\quad\text{j=2,\dots,d}.
\end{aligned}\right.
\end{equation}

\end{Lem}
\begin{proof}
Denote $A = P_0(v_1)P_0$ in this proof. Since $P_0f = \sum^{d+1}_{i=0}(f,\psi_i)_{L^2}\psi_i$, we have 
\begin{align*}
Af = P_0(v_1\sum^{d+1}_{i=0}(f,\psi_i)_{L^2}\psi_i)
&= \sum^{d+1}_{i,j=0}(f,\psi_i)_{L^2_v}(v_1\psi_i,\psi_j)_{L^2}\psi_j,
\end{align*}and hence the matrix representation of $A$ is symmetric. Denote $e_1:=(1,0,\dots,0)$ to be a $d$ dimension vector, then 
\begin{align*}
A\begin{pmatrix}
\psi_0\\
\vdots\\
\psi_{d+1}
\end{pmatrix}&=
\begin{pmatrix}
(v_1\psi_0,\psi_0)_{L^2} &  \cdots & (v_1\psi_0,\psi_{d+1})_{L^2}\\
\vdots & \ddots & \vdots\\
(v_1\psi_{d+1},\psi_0)_{L^2})  & \cdots & (v_1\psi_{d+1},\psi_{d+1})_{L^2}
\end{pmatrix}
\begin{pmatrix}
\psi_0\\
\vdots\\
\psi_{d+1}
\end{pmatrix}\\
&=
\begin{pmatrix}
0 & e_1  & 0\\
e^T_1 & 0  & \sqrt{\frac{2}{d}}e^T_1\\
0 & \sqrt{\frac{2}{d}}e_1 & 0
\end{pmatrix}
\begin{pmatrix}
\psi_0\\
\vdots\\
\psi_{d+1}
\end{pmatrix}
\end{align*}
by the properties of Gamma function and noticing that $\mu^{1/2}$ is even about $v_1$ while $v_1\psi_i$ is odd about $v_1$ if and only if $i\neq 1$. 
Then 
\begin{align*}
|\lambda I-A|
&= \begin{vmatrix}
\lambda & -e_1 & 0\\
-e_1 & \lambda I_d & -\sqrt{\frac{2}{d}}e^T_1\\
0 & -\sqrt{\frac{2}{d}}e_1 & \lambda
\end{vmatrix}
= \lambda^{d-1}\begin{vmatrix}
\lambda & -1 & 0\\-1 & \lambda &-\sqrt{\frac{2}{d}}\\
0 & -\sqrt{\frac{2}{d}} & 0
\end{vmatrix}
=\lambda^d(\lambda^2-\frac{d+2}{d}).
\end{align*}
Thus $\lambda = -\sqrt{\frac{d+2}{d}}, \sqrt{\frac{d+2}{d}}, 0$ ($d$ multiplicity). For $\lambda=0$, the corresponding unit eigenvectors are 
\begin{align*}
\begin{pmatrix}
-\sqrt\frac{2}{d+2}, 
0,
\cdots,
0,
\sqrt{\frac{d}{d+2}}
\end{pmatrix},\ e_2,\ \cdots,\ e_d,
\end{align*}where $e_j$ are unit vectors in $\R^{d+2}$ whose standard orthonormal base is $(e_0,e_1,\dots,e_{d+1})$. For $\lambda =\sqrt\frac{d+2}{d}, -\sqrt{\frac{d+2}{d}}$, the corresponding unit eigenvectors are 

\begin{align*}
\begin{matrix}
\left(\sqrt\frac{d}{2(d+2)}, \frac{1}{\sqrt{2}}, 0, \dots, 0, \sqrt{\frac{1}{d+2}}\right),\ \left(\sqrt\frac{d}{2(d+2)}, -\frac{1}{\sqrt{2}}, 0, \dots, 0, \sqrt{\frac{1}{d+2}}\right)
\end{matrix}
\end{align*}respectively. 
Therefore, the eigenvalues and eigenfuntions of $A$ are as in \eqref{eigen_matrix_eq}.
\qe\end{proof}

\providecommand{\bysame}{\leavevmode\hbox to3em{\hrulefill}\thinspace}
\providecommand{\MR}{\relax\ifhmode\unskip\space\fi MR }
\providecommand{\MRhref}[2]{%
	\href{http://www.ams.org/mathscinet-getitem?mr=#1}{#2}
}
\providecommand{\href}[2]{#2}

\end{document}